\documentclass[12pt]{article}

\usepackage[margin=1in]{geometry}

\usepackage[sectionbib]{natbib}
\usepackage{hyperref}
\usepackage{graphicx}
\usepackage{epsfig}

\usepackage{amsmath}
\usepackage{amssymb}
\usepackage{amsfonts}
\usepackage{amsthm}
\usepackage{dsfont}
\usepackage{bbm}

\usepackage{multirow}
\usepackage[table]{xcolor}
\usepackage{cancel}
\usepackage{microtype}        
\usepackage{fancyhdr}
\usepackage{sectsty, secdot}

\newtheorem{thm}{Theorem}
\newtheorem{lem}{Lemma}

\newtheorem{applem}{Lemma A.}
\theoremstyle{definition}

\newtheorem{example}{Example}

\begin{document}

\renewcommand{\P}{\mathbb{P}}
\newcommand{\E}{\mathbb{E}}
\newcommand{\R}{\mathbb{R}}
\newcommand{\Q}{\mathbb{Q}}
\newcommand{\X}{\mathcal{X}}
\newcommand{\Y}{\mathcal{Y}}
\newcommand{\Hcal}{\mathcal{H}}
\newcommand{\F}{\mathrm{F}}
\newcommand{\1}{\mathbbm{1}}
\newcommand{\Var}{\mathrm{Var}}
\newcommand{\var}{\mathrm{Var}}
\newcommand{\tr}{\operatorname{tr}}
\newcommand{\Vol}{\mathrm{Vol}}
\newcommand{\norm}[1]{\left\lVert #1 \right\rVert}
\newcommand{\floor}[1]{\lfloor #1 \rfloor}
\newcommand{\veps}{\varepsilon}
\newcommand{\GenJM}{\widetilde{\mathscr{E}}}
\newcommand{\GenBG}{\widetilde{\mathscr{D}}}
\newcommand{\GammaT}{\tilde{\gamma}}
\newcommand{\thetaT}{\tilde{\theta}}
\newcommand{\Xvec}{\mathbf{X}}
\newcommand{\Yvec}{\mathbf{Y}}
\newcommand{\Zvec}{\mathbf{Z}}
\newcommand{\zvec}{\mathbf{z}}
\newcommand{\Wvec}{\mathbf{W}}
\newcommand{\xvec}{\mathbf{x}}
\newcommand{\yvec}{\mathbf{y}}
\newcommand{\wvec}{\mathbf{w}}
\newcommand{\thetavec}{\boldsymbol{\theta}}
\newcommand{\sigmat}{\boldsymbol{\Sigma}}
\newcommand{\muvec}{\boldsymbol{\mu}}

\begin{center}
  {\Large\bfseries On High-Dimensional Change-Point Detection Based on Pairwise Distances\par}
  \vspace{1.2em}
  {\large Spandan Ghoshal$^{1}$, Bilol Banerjee$^{2}$ and Anil K.\ Ghosh$^{1}$\par}
  \vspace{0.7em}
  {\itshape
    $^{1}$Theoretical Statistics and Mathematics Unit, Indian Statistical Institute, Kolkata\\[2pt]
    $^{2}$Department of Statistics and Data Science, National University of Singapore
  }
\end{center}


\begin{quotation}
\noindent {\it Abstract:}
{In change-point analysis, one aims at finding the locations of abrupt distributional changes (if any) in a sequence of multivariate observations. In this article, we propose some pairwise distance-based methods for this purpose. These proposed methods can be conveniently used for high-dimensional data even when the dimension is much larger than the sample size. We carry out theoretical investigations on their behaviour not only when the dimension of the data remains fixed, and the sample size grows to infinity, but also when the dimension diverges to infinity while the sample size may or may not grow with the dimension. Several simulated and real datasets are analyzed to compare the empirical performance of these proposed methods against some state-of-the-art methods.}\\

\vspace{9pt}
\par
\end{quotation}\par

\def\thefigure{\arabic{figure}}
\def\thetable{\arabic{table}}

\renewcommand{\theequation}{\thesection.\arabic{equation}}

\fontsize{12}{14pt plus.8pt minus .6pt}\selectfont

\section{Introduction}

Change-point detection is a classical problem in statistics. It originated in the field of statistical quality control \citep[see][]{girshick1952bayes} and now has widespread applications in many areas, including 
finance,
climate monitoring 
and genomics \citep[see][and the references therein]{TurongReview2020}. 
In a sequence of $d$-dimensional random vectors $\Zvec_1\sim F_1,\ldots, \Zvec_n\sim F_n$, if $F_{\tau}$ differs from $F_{\tau+1}$ ($1\le \tau \le n-1$), $\tau$ is called a \text{change-point}. A sequence may have one or more change-points, or it may have none at all. So, the aim of change-point analysis is two-fold; (i) deciding whether a sequence has any change-point, and (ii) finding the change-points, if any. In single change-point analysis, we assume that the sequence can have at most one change-point. In such cases, we test the null hypothesis $H_0$: $F_1 = F_2= \cdots= F_n$ against the alternative hypothesis $H_1$: $F_1=\cdots=F_{\tau}\neq F_{\tau+1}=\cdots = F_n$ for 
    some unknown $\tau$, and also look for an estimate of the change-point if $H_0$ is rejected. 
    However, in multiple change-point analysis, we 
    allow the number of change-points to be more than one and aim to detect all of them. 
        
    Starting from  \cite{girshick1952bayes} and \cite{Page1955}, several methods for change-point detection have been proposed in the literature. For univariate data, some notable parametric methods include \cite{sen1975tests,cox1982partitioning,Worsley1986} and \cite{James1987univ}. Some nonparametric methods \citep[see, e.g.][]{bhattacharyya1968nonparametric,pettitt1979non,wolfe1984nonparametric,csorgo1987nonparametric}) are also available. We refer the interested readers to the monographs by \cite{chen2011parametric} and \cite{brodsky2013nonparametric}
for a comprehensive review
on change-point methods for univariate data. 

    For the multivariate data, some notable parametric methods include \cite{srivastava1986likelihood,james1992asymptotic,zhang2010detecting}. However, they mainly focus on detecting location shifts in a sequence of independent Gaussian random vectors. Several nonparametric methods have also been developed for detecting location changes \citep[see, e.g.,][]{lung2015homogeneity,
    aston2014change, wang2018,shao2010self} and changes in covariance matrices \citep[see, e.g.,][]{avanesov2018change,Zhong2019,Keshavarz2020,Londschien2021,Kaul2023}, which can also be used for high-dimensional data.

        
    Some graph-based methods were proposed \citep{chen2015graph, shi2017consistent,dawn2025clustcp} to detect more general distributional changes. Most of them use graph-based two-sample test statistics for change-point detection. For instance, \cite{chen2015graph} developed methods based on minimum distance pairing \citep{rosenbaum2005exact}, nearest neighbors \citep{schilling1986multivariate,henze1988multivariate}, and minimum spanning tree \citep{friedman1979multivariate}. But these graph-based tests are inefficient in the classical asymptotic regime \citep[see][]{Bhattacharya2019}. This is likely to have a negative impact on the performance of the associated change-point detection methods. Some kernel-based methods are also available \citep[see, e.g.][]{
harchaoui2008kernel,li2015m,Harchaoui2019Kernel}. However, as mentioned in \cite{chen2015graph}, their
performance depends heavily on the choice of the kernel and the associated bandwidth, which is crucial in higher dimensions. 
Another popular approach was proposed by \cite{Matteson2014} using 
the two-sample energy statistic \citep[see, e.g.,][]{Baringhaus2004}. 
However, the test based on this energy statistic often yields poor performance for high-dimensions \citep[see][]{Biswas2014}. This may adversely affect the performance of the associated change-point detection method. 

To demonstrate this, 
consider two simple examples, each involving 50 observations from 200-dimensional normal distributions 
with a change-point at $25$. 
Descriptions of these two examples are given below.

\begin{example}\label{exa:1}
		$\Zvec_1,\ldots,\Zvec_{25} \stackrel{iid}{\sim} \mathcal{N}({\bf 0}_{200},{\bf I}_{200})$ and $\Zvec_{26},\ldots,\Zvec_{50} \stackrel{iid}{\sim} \mathcal{N}(0.3 {\bf 1}_{200},{\bf I}_{200})$,
		where ${\bf 0}_d=(0,0,\ldots,0)^\top$ and ${\bf 1}_d=(1,1,\ldots,1)^\top$ are a $d$-dimensional vectors,
        ${\bf I}_d$ is the $d \times d$ identity matrix, and $\mathcal{N}(\muvec,\sigmat)$ denotes the normal distribution with the mean vector $\muvec$
	and the dispersion matrix $\sigmat$.
    \end{example}
	
	\begin{example}\label{exa:2}
		$\Zvec_1,\ldots,\Zvec_{25} \stackrel{iid}{\sim} \mathcal{N}({\bf 0}_{200},{\bf I}_{200})$ and $\Zvec_{26},\ldots,\Zvec_{50} \stackrel{iid}{\sim} \mathcal{N}({\bf 0}_{200}, 1.3 {\bf I}_{200})$. 
	\end{example}

    For these examples, we use the distance-based method of \cite{Matteson2014} (referred to as MJ), kernel change-point analysis (KCPA) proposed by \cite{Harchaoui2019Kernel} and the graph-based methods of \cite{chen2015graph} based on minimum spanning tree (MST) and nearest neighbor graph (NNG) with five neighbors. For KCPA, we use the Gaussian kernel, where the bandwidth is chosen based on a pilot study. Each experiment is repeated $1000$ times. 
    These methods use a suitable statistic ($S(t)$, say) to compute the distributional difference between the two sets $\{\Zvec_1,\ldots,\Zvec_t\}$ and $\{\Zvec_{t+1},\ldots,\Zvec_n\}$. This is done for several choices of $t$, and the maximizer of $S(t)$ is considered as the potential change-point; call it $t_0$. One uses $S(t_0)$ to test for the statistical significance of this change-point and rejects $H_0$ (the hypothesis of no change) when $S(t_0)$  is significantly large.
    If $H_0$ is rejected, $t_0$ is considered as the estimated or detected change-point. The grey and black bars in Figure~\ref{fig:histogram_plots} show the frequency distributions of potential and detected (at level $\alpha=0.05$) change-points, respectively. 

    \begin{figure}[t]
\centering
\includegraphics[width=0.95\textwidth]{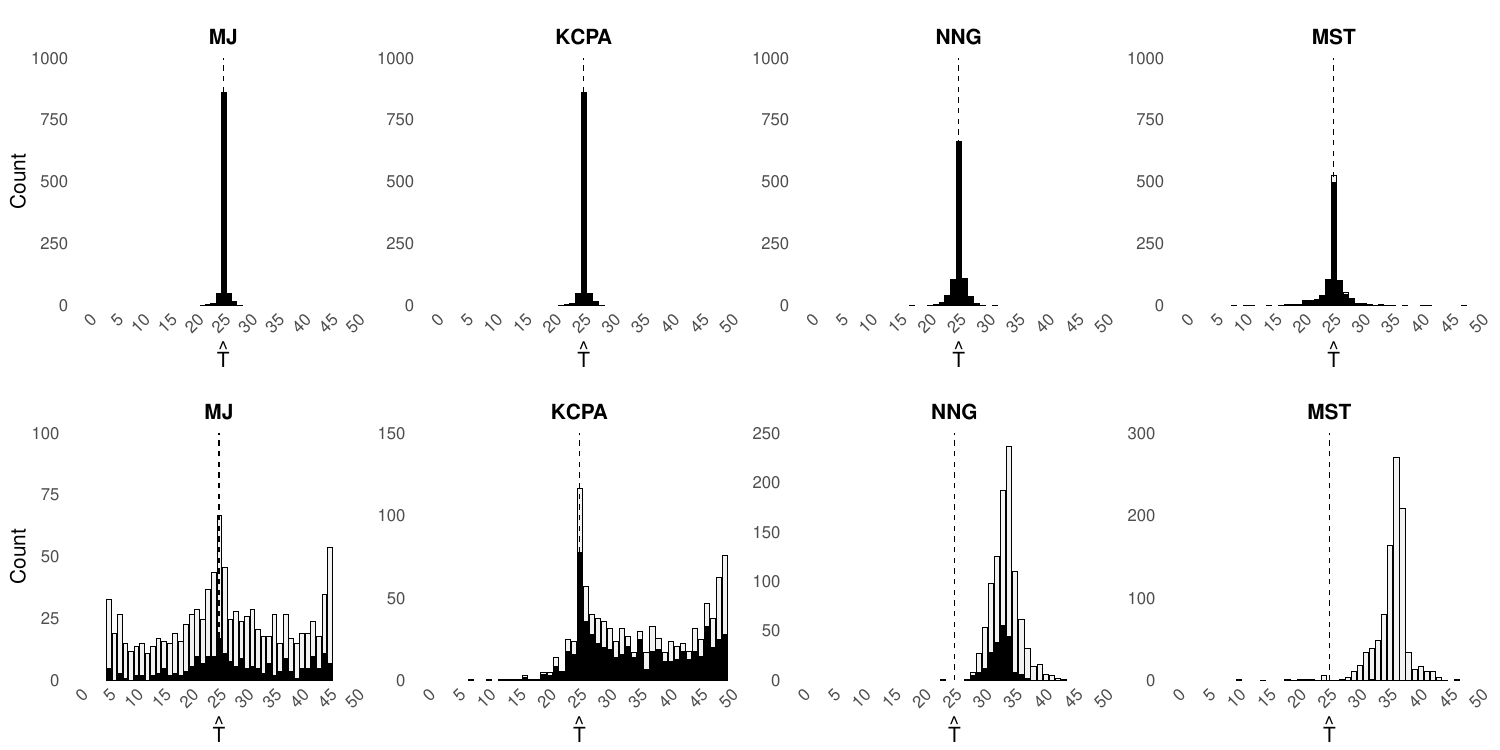}
\caption[Counts of statistically significant and insignificant potential change-points at various locations]{Frequency distributions of potential (grey bars) and  detected 
change-points (black bars) in Examples \ref{exa:1} (top row) and \ref{exa:2} (bottom row). 
}
\label{fig:histogram_plots}
\end{figure}
		
    In Example \ref{exa:1}, all methods had satisfactory performance in detecting the location shift. Among them, MJ had the best performance followed by KCPA. 
    However, in Example \ref{exa:2}, they had poor performance in detecting the underlying change in scale. 
    Among them, the performance of KCPA was slightly better. 
Though NNG detected change-points (not at the true location) in a few cases, MST could not even do that on a single occasion.
    
	


    To take care of such issues, in this article, we propose some new distance-based methods, 
     and investigate their performance both from theoretical and empirical perspectives. Our contributions are summarised as follows.
	
In Section \ref{sec:algorithm}, we propose a method based on $\ell_2$ distance for single change-point detection and establish its large sample consistency against general alternatives.
In Section \ref{sec:hd-behaviour}, we investigate its performance for high dimensional data. We show that the proposed method based on $\ell_2$ distance may fail to detect changes outside the first two moments, but the use of other suitable distances can take care of this problem. We also use a non-asymptotic framework to study the behavior of our proposed methods when the dimension and the sample size grow simultaneously.  
Some high dimensional simulated data sets are analyzed in Section \ref{sec:Simulation} to demonstrate the utility of the proposed methods.  
In Section \ref{subsec:multiple}, we extend our methods for multiple change-point detection. A real life financial data set is analyzed in Section \ref{subsec:real}.
       Finally, Section \ref{sec:Concluding} contains a brief summary of the work and some remarks on possible directions for future research.  All proofs and mathematical details are given in the Supplementary material.

\section{The proposed method}
	\label{sec:algorithm} 
     For independent random vectors $\Xvec_1,\Xvec_2\overset{iid}{\sim}\F_{1}$ and $\Yvec_1,\Yvec_2\overset{iid}{\sim}\F_{2}$, \cite{Maa1996} proved that
     $\norm{\Xvec_1-\Yvec_1}$, $\norm{\Xvec_1-\Xvec_2}$ and $\norm{\Yvec_1-\Yvec_2}$ follow the same distribution if and only if $\F_1$ and $\F_2$ are identical. 
     Assuming $\E[\|\Xvec_1\|+\|\Yvec_2\|]<\infty$, \cite{Baringhaus2004} defined the energy distance
    ${\cal E}(\F_1,\F_2)=2\E\norm{\Xvec_1-\Yvec_1}-\E\norm{\Xvec_1-\Xvec_2}-\E\norm{\Yvec_1-\Yvec_2}$ and proved that $\mathcal{E}(\F_1,\F_2)\geq 0$ where the equality holds if and only if $\F_1=\F_2$. 
    Now, define $\muvec_{1}=(\E\norm{\Xvec_1-\Xvec_2},\E\norm{\Xvec_1-\Yvec_2})^{\top}$ and ${\muvec}_{2}=(\E\norm{\Yvec_1-\Xvec_2},\E\norm{\Yvec_1-\Yvec_2})^\top$.
    Note that the energy distance between $\F_1$ and $\F_2$ 
    can also be written as ${\cal E}(\F_1,\F_2)={\bf a}^{\prime}\left({\muvec}_{1}-{\muvec}_{2}\right)$, i.e., the distance between $\muvec_1$ and $\muvec_2$ along the direction ${\bf a}=(-1,1)^{\top}$. Instead, one can use the actual distance $\|\muvec_1-\muvec_2\|$ or its square $D(F_1,F_2)=\|\muvec_1-\muvec_2\|^2$ as a measure of divergence. 
    Clearly, $D(\F_1,\F_2)$ is non-negative, and the following lemma shows that $D(\F_1,\F_2)=0$ if and only if $\F_1$ and $\F_2$ are identical.
    
    \begin{lem}
    \label{lem:1} If $\Xvec_1,\Xvec_2\overset{iid}{\sim}\F_{1}$
    and $\Yvec_1,\Yvec_2\overset{iid}{\sim}\F_{2}$ are independent with $\E\norm{\Xvec_1}+\E\norm{\Yvec_1}<\infty$,
    then $D(F_1,F_2)=0$ if and only if $\F_{1}=\F_{2}$. 
    \end{lem}
    
    We use this divergence measure $D(\cdot,\cdot)$ to construct our change-point detection algorithm. 
     First, for some $t$, divide the sequence 
     $\left\{ \Zvec_{1},\Zvec_{2},\dots,\Zvec_{n}\right\}$ into two parts  $\mathcal{X}_{t}=\{\Zvec_{1},\ldots,\Zvec_{t}\}$ and
    $\mathcal{Y}_{t}=\{\Zvec_{t+1},\ldots,\Zvec_{n}\}$ 
    and compute    
    \begin{align*}
    T_{11}\left(t\right) & ={t \choose 2}^{-1} \hspace{-0.1in}\sum\limits_{1\leq i<j\leq t}\norm{\Zvec{}_{i}-\Zvec_{j}},\quad T_{22}\left(t\right)={n-t \choose 2}^{-1} \hspace{-0.1in} \sum\limits_{t+1\leq i<j\leq n}\norm{\Zvec{}_{i}-\Zvec_{j}}\\
     & ~\quad\textnormal{ }\quad T_{12}\left(t\right)=\frac{1}{t\left(n-t\right)}\sum\limits_{1\leq i\leq t}\sum\limits_{t+1\leq j\leq n}\norm{\Zvec{}_{i}-\Zvec_{j}}.
    \end{align*}
    Define
    $\mathcal{D}_{n}(\mathcal{X}_{t},\mathcal{Y}_{t})=\left\{ T_{12}\left(t\right)-T_{11}\left(t\right)\right\} ^{2}+\left\{ T_{12}\left(t\right)-T_{22}\left(t\right)\right\} ^{2}$
    as an empirical analog of $D(\cdot, \cdot)$. If $t$ is a true change-point,  $\mathcal{D}_{n}(\mathcal{X}_{t},\mathcal{Y}_{t})$ is expected to be large. 
    So, we compute $\mathcal{D}_{n}(\mathcal{X}_{t},\mathcal{Y}_{t})$ for different values of $t$ and consider 
    \begin{align}
    {\widehat T}_n & =\underset{t\in \mathcal{A}_n}{\text{argmax}}\frac{t\left(n-t\right)}{n^{2}}\mathcal{D}_{n}(\mathcal{X}_{t},\mathcal{Y}_{t})\label{eq:estimator}
    \end{align}
    as the potential change-point. Here we assume that there is at least $\delta_n$ (a pre-defined small quantity) proportion of observations both before and after the change-point (i.e., $\lceil n \delta_n\rceil \le \tau \le \lfloor n(1-\delta_n)\rfloor$, where for any $x \in {\mathbb R}$, $\lfloor x\rfloor=\inf\{k\in {\mathbb N}: k\ge x\}$ and $\lceil x \rceil=\sup\{k\in {\mathbb N}: k\le x\}$ are two integers). So, we find the maximiser over the set $\mathcal{A}_n = \{\lceil n\delta_n\rceil, \ldots, \lfloor n(1-\delta_n)\rfloor\}$.
    The potential change-point ${\widehat T}_n$ is considered to be statistically significant at level $\alpha$ ($0<\alpha<1$) if the value of the test statistic 
\begin{align}
    \widehat{S}_{n} & =\underset{t\in\mathcal{A}_n}{\text{max }}\frac{t\left(n-t\right)}{n^{2}}\mathcal{D}_{n}(\mathcal{X}_{t},\mathcal{Y}_{t})\label{eq:test-statistic}
\end{align}
    is significantly large. The cut-off is computed using the permutation method. In such cases, ${\widehat T}_n$ is considered as the estimated or detected change-point. 

    As $n$ increases, if we shrink the value of $\delta_n$ at an appropriate rate, ${\widehat \theta}_n=\widehat T_n/n$  becomes a strongly consistent estimator of $\theta=\displaystyle \lim_{n\rightarrow\infty}\tau/n$. Moreover, if the underlying distributions are sub-exponential \citep[see, e.g.,][] {foss2013introduction},  ${\widehat T}_n$ turns out to be statistically significant with probability tending to one as $n$ diverges to infinity. These results are asserted by Theorem 1.
    


\begin{thm}
    \label{thm:large-sam-consistency}
      Suppose $\Zvec_1,\ldots, \Zvec_{\lfloor n\theta \rfloor}\stackrel{iid}{\sim}\F_1$ and $\Zvec_{\lfloor n\theta \rfloor +1},\ldots, \Zvec_{n}\stackrel{iid}{\sim}\F_2$ for some $\theta \in(0,1)$ $(\F_1 \neq \F_2)$.
      If $\delta_n \rightarrow 0$ and $n \delta_n/\log n \rightarrow \infty$ as $n \rightarrow \infty$, then 
\begin{enumerate}
    \item[$(a)$] ${\widehat{ \theta}}_n={\widehat{T}}_n/n$ converges to $\theta$ almost surely.
    \item[$(b)$] If $\F_1$ and $\F_2$ are sub-exponential, for any given level $\alpha ~(0<\alpha<1)$,
    the change-point $\tau$ is detected (by $\widehat{T}_n$) with probability tending to one.
\end{enumerate}
\end{thm}

	\begin{figure}[t]
		\centering
		\setlength{\tabcolsep}{-2pt}
		\begin{tabular}{cccc}
			Example \ref{exa:1} & Example \ref{exa:2} \\
			\includegraphics[height=1.25in,width  = 0.4\textwidth]{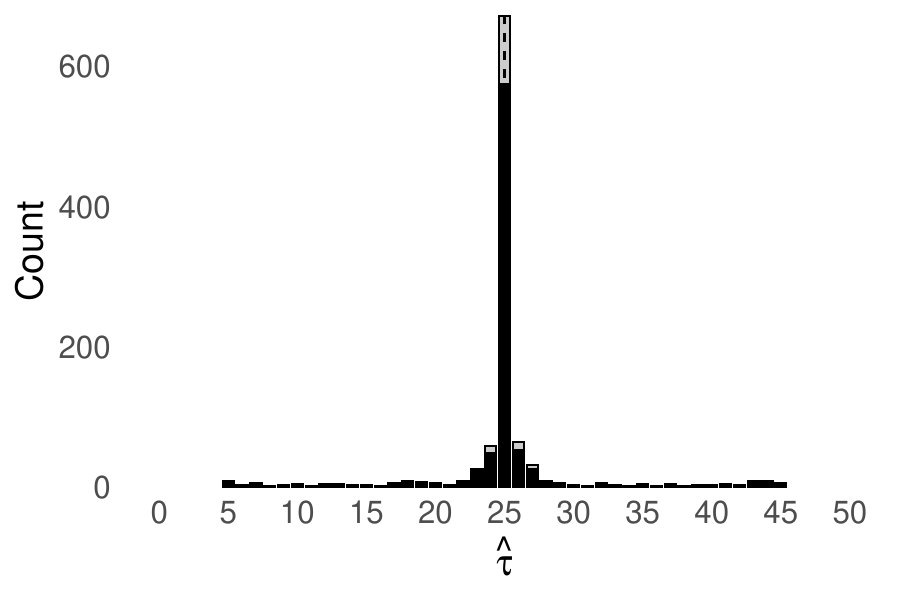} &
			\includegraphics[height=1.25in,width  = 0.4\textwidth]{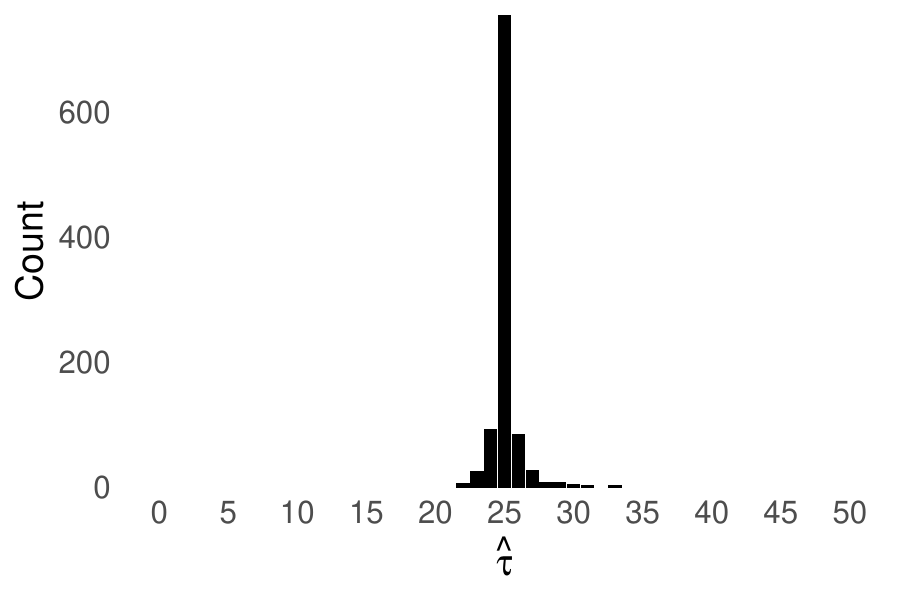}\\
		\end{tabular}
		
\caption{Frequency distribution of potential change-points (grey bars) and detected change-points (black bars) by the proposed method in Examples 1 and 2. The dashed line represents the true change-point $\tau = 25$.}
		\label{fig:bg_histogram_plots}
	\end{figure}

    
    Figure \ref{fig:bg_histogram_plots}  shows that the proposed method had an excellent performance in Examples \ref{exa:1} and \ref{exa:2}. In these examples, we used $\delta_n=0.05$ (i.e., $n\delta_n=2.5)$ to have $|{\cal X}_t| \ge 2$ and $|{\cal Y}_t| \ge 2$ for all $t \in {\cal A}_n$. In Example \ref{exa:1}, the
    performance of the proposed method was comparable to that of NNG and KCPA, and better than MST. 
    But in Example \ref{exa:2}, while all other competing methods performed poorly, the proposed method had an outstanding performance. 



\section{High-dimensional behaviour of the proposed method }
	\label{sec:hd-behaviour}

    In Section \ref{sec:algorithm}, we have seen that the proposed method can perform well in high dimensions even when the dimension exceeds the sample size. This motivates us to investigate its behaviour in the high dimension low sample size (HDLSS) asymptotic regime, where the sample size is fixed while the dimension grows to infinity. For this investigation, we assume  $\Zvec_1,\ldots,\Zvec_{\tau} \stackrel{iid}{\sim} \F_1$ and $\Zvec_{\tau+1},\ldots,\Zvec_{n} \stackrel{iid}{\sim} \F_2$, where $\F_i$ has the mean vector $\muvec_{i}$ and the scatter matrix $\sigmat_{i}$ ($i=1,2$).
We also assume that
\begin{enumerate}
    \item[(A1)] Coordinate variables in $\F_i$ have uniformly bounded fourth moments.
    \item[(A2)] If $\Zvec = (Z^{(1)},\ldots,Z^{(d)})^{\top}\sim \F_r$ and $\Zvec_{*} = 
(Z_*^{(1)},\ldots,Z_*^{(d)})^{\top}\sim \F_s$ ($r,s=1,2$) are independent, 
    for ${\bf W}=\Zvec-\Zvec_*$,  
    $\sum_{q \neq q'}\big|\text{Corr}\big((W^{(q)})^2,(W^{(q')})^2\big)\big|$
    is of the order $o(d^2)$.
\end{enumerate}

Under (A1) and (A2), the weak law of large numbers (WLLN)  holds for the sequence of
possibly dependent and non-identically distributed 
random variables $\{{W^{(q)}}^2:q\ge 1\}$ (i.e., $\left|\frac{1}{d}\|{\bf W}\|^2-E\big(\frac{1}{d}\|{\bf W}\|^2\big)\right| \stackrel{P}{\rightarrow}0$ as $d \rightarrow \infty$), and (A3) gives the limiting (as $d \rightarrow \infty$) values of $E\big(\frac{1}{d}\|{\bf W}\|^2\big)$ and hence those of 
$\frac{1}{d}\|{\bf W}\|^2$ for different choices of $r$ and $s$ ($r,s=1,2$). However, if the coordinate variables are i.i.d., as in the Examples \ref{exa:1} and \ref{exa:2}, WLLN holds under the second moment assumption; (A1) and (A2) are not needed there. Assumptions (A1)-(A3) are quite common in the HDLSS literature. In addition to (A1) and (A3), \cite{hall2005geometric} assumed 
a $\rho$-mixing property of the coordinate variables for distance convergence. 
 Assumption (A2), which assumes weak dependence among the coordinate variables, holds under that $\rho$-mixing condition.  Similar conditions for WLLN and the convergence of pairwise distances were also considered in \cite{ahn2007high, jung2009pca,aoshima2018} and \cite{banerjee2025high}.
{\color{red}

}

For a permutation $\pi$ of $\{1,\ldots,n\}$ and $t\in\mathcal A_n$, 
define ${\cal X}_t^{\pi}=\{\Zvec_{\pi(1)},\ldots,\Zvec_{\pi(t)}\}$
and ${\cal Y}_t^{\pi}=\{\Zvec_{\pi(t+1)},\ldots,\Zvec_{\pi(n)}\}$.
Let $k_\pi(t) :=\#\{i\le t:\pi(i)\le \tau\}$ be the number of 
observations in ${\cal X}_t^{\pi}$ from $\F_1$. So, the proportion 
of observations from $\F_1$ in ${\cal X}_t^{\pi}$ and ${\cal Y}_t^{\pi}$ 
are given by $p_\pi(t)=k_\pi(t)/t$ and $q_\pi(t)=\{\tau-k_\pi(t)\}/(n-t)$, 
respectively. For any fixed $n$ and $\tau$, define a threshold 
\begin{equation}\label{eq:rho_def}
\rho_{n,\tau}:=
-2\,\frac{w(m)}{m-1}
+\sqrt{
4\Big(\frac{w(m)}{m-1}\Big)^2
+2\Big\{
w(\tau)-\frac{w(m)}{2(m-1)^2}
\Big\}
},
\end{equation}
where $m=\lfloor n\delta_n\rfloor$ and $w(t)=t(n-t)/n^2$. Also, 
define the set of \emph{unbalanced permutations} 
$\mathcal{B}(\rho_{n,\tau}):=\Big\{\pi:\max_{t\in\mathcal A_n}
\Big|p_\pi(t)-q_\pi(t)\Big|>\rho_{n,\tau}\Big\}.$
If the cardinality of this set is smaller than $\alpha\cdot n!$, 
under (A1)--(A3), we have the HDLSS consistency of the proposed 
test. This result is stated below.

\begin{thm}
\label{thm:HDLSS-consistency}
Suppose $\Zvec_1,\ldots,\Zvec_{\tau}\stackrel{iid}{\sim}\F_1$ and
$\Zvec_{\tau+1},\ldots,\Zvec_n\stackrel{iid}{\sim}\F_2$, where 
$\F_1$ and $\F_2$ satisfy (A1)--(A3) with $\nu^2>0$ or 
$\sigma_1^2\neq\sigma_2^2$. If $\tau\in\mathcal A_n$,
\begin{enumerate}
\item[$(a)$] $\P(\widehat T_n=\tau)\to 1$ as $d\to\infty$.
\item[$(b)$] If $|\mathcal{B}(\rho_{n,\tau})|/n!\le\alpha$, then
the change-point $\tau$ is detected (by $\widehat T_n$) at level $\alpha$
$(0<\alpha<1)$ with probability converging to one as $d\to\infty$.
\end{enumerate}
\end{thm}

If $\Pi$ is uniform over all permutations of $\{1,2\ldots,n\}$, then we have $\P(\Pi\in \mathcal{B} (\rho_{n,\tau}))=| \mathcal{B} (\rho_{n,\tau})|/n!$. Clearly, this does not depend on the 
underlying distributions, and using Hoeffding's inequality for sampling 
without replacement, one gets
\begin{align}
\label{eq:Bbound}
\nonumber
& \P\bigl(\Pi \in \mathcal{B}(\rho_{n,\tau})\bigr)
\leq \sum_{t \in \mathcal{A}_n}
     \P\!\left(\left|\frac{k_\Pi(t)}{t} - \frac{\tau}{n}\right|
               > \frac{(n-t)\,\rho_{n,\tau}}{n}\right) \\
& \leq \sum_{t \in \mathcal{A}_n}
     2\exp\!\left(-\frac{2t(n-t)^2\rho_{n,\tau}^{2}}{n^{2}}\right) \le 2n\exp\!\left(-2n\,\delta_n^{2}(1-\delta_n)\,\rho_{n,\tau}^{2}\right).
\end{align}
So, for any fixed $\alpha \in (0,1)$, $|\mathcal{B}(\rho_{n,\tau})|/n!\le\alpha$ is guaranteed 
whenever $n \geq \log(2n/\alpha)/\{2\delta_n^2(1-\delta_n)
\rho_{n,\tau}^2\}$. However, this is a sufficient condition only. In practice, a much smaller value of $n$ satisfies the condition. Table~\ref{permute} below gives the exact value of $|\mathcal{B}(\rho_{n,\tau})|/n!$ based on $10^6$ Monte Carlo simulations. It clearly shows that $n=50$ 
suffices for $\alpha=0.05$ and all $\theta\in[0.1,0.9]$, with 
the required sample size increasing as $\theta$ moves away from 
$0.5$.

\begin{table}[h]
\centering\small
\setlength{\tabcolsep}{2pt}
\renewcommand{\arraystretch}{1}
\begin{tabular}{c|ccccccc}
\hline
$\theta$ & $n=30$ & $n=35$ & $n=40$ & $n=45$ & $\mathbf{n=50}$ 
         & $n=55$ & $n=60$ \\
\hline
0.1,\,0.9 & 0.145119 & 0.129557 & 0.097348 & 0.118382 
          & \textbf{0.032213} & 0.030700 & 0.023948 \\
0.2,\,0.8 & 0.213176 & 0.052443 & \textbf{0.051041} & 0.024689 
          & 0.022464 & 0.023782 & 0.006243 \\
0.4,\,0.6 & 0.118129 & \textbf{0.043077} & 0.041161 & 0.014480 
          & 0.014655 & 0.005141 & 0.005500 \\
0.5       & \textbf{0.036767} & 0.020103 & 0.002790 & 0.000636 
          & 0.000077 & 0.000036 & 0.000005 \\
\hline
\end{tabular}
\caption{
$\P(\mathcal{B}(\rho_{n,\tau}))$ computed based on $10^6$ Monte Carlo simulations for different choices 
$\theta=\tau/n
\in[0.1,0.9]$ and the sample size $n$. 
Bold entries indicate the smallest $n$ at which the probability first falls below $\alpha=0.05$.}
\label{permute}
\end{table}
Theorem 2 shows that if 
$\nu^2>0$ or $\sigma^2_{1}\not=\sigma^2_{2}$, the proposed method successfully detects { the distributional change} 
    with high probability even when $n$ is much smaller $d$. We  observed this in Examples \ref{exa:1} and \ref{exa:2}. Now, consider two high-dimensional examples, where we have $\nu^2=0$ and $\sigma_1^2=\sigma_2^2$. 
	\begin{example}\label{exa:3}
$\Zvec_1,\ldots,\Zvec_{25} \stackrel{iid}{\sim}\mathcal{N}({\bf 0}_{200}, 2{\bf I}_{200})$ and $\Zvec_{26},\ldots,\Zvec_{50}\stackrel{iid}{\sim}{\mathcal T}_{200}(4)$, where ${\cal T}_{d}(k)$ denotes the $d$-dimensional distribution with i.i.d. {coordinates each following} the standard $t$ distribution with $k$ degrees of freedom.
	\end{example} 

\begin{example}\label{exa:4}
$\Zvec_1,\ldots,\Zvec_{25} \stackrel{iid}{\sim}\mathcal{N}({\bf 0}_{200}, {\sigmat}_1)$ and $\Zvec_{26},\ldots,\Zvec_{50}\stackrel{iid}{\sim}\mathcal{N}({\bf 0}_{200}, {\sigmat}_2)$, where $\sigmat_1$ and $\sigmat_2$ are diagonal matrices. {The first $d/2$ diagonal elements of $\sigmat_1$ are $1$ and the rest are $3$, while  all diagonal elements of $\sigmat_2$ are $2$.}
	\end{example} 
    
\begin{figure}[t]
\centering
\setlength{\tabcolsep}{2pt} 
\renewcommand{\arraystretch}{1.2} 
\begin{tabular}{cc}
  Example~\ref{exa:3}: $\ell_2$ distance &
  Example~\ref{exa:3}: $\ell_1$ distance \\
  \includegraphics[height=1.0in,width=.48\textwidth]{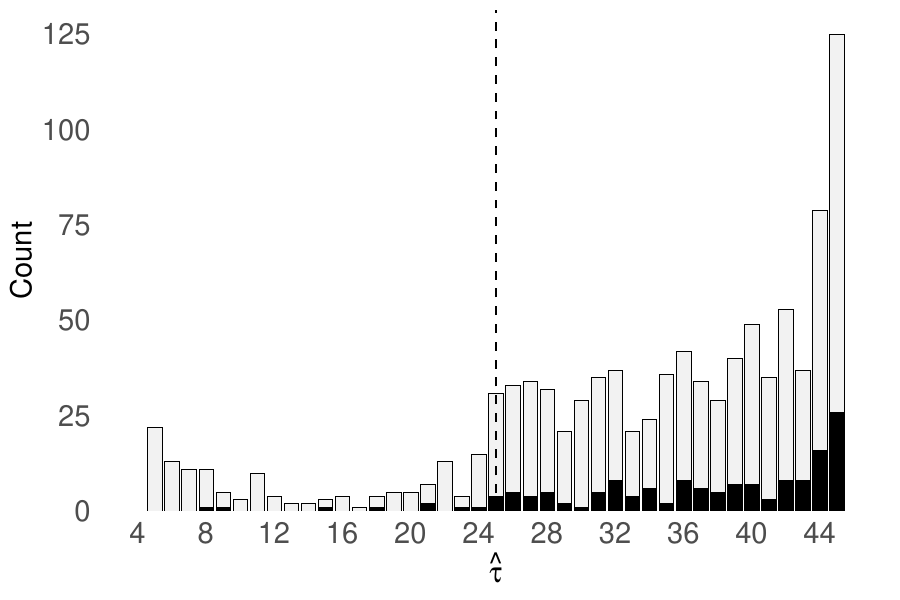} &
  \includegraphics[height=1.0in,width=.48\textwidth]{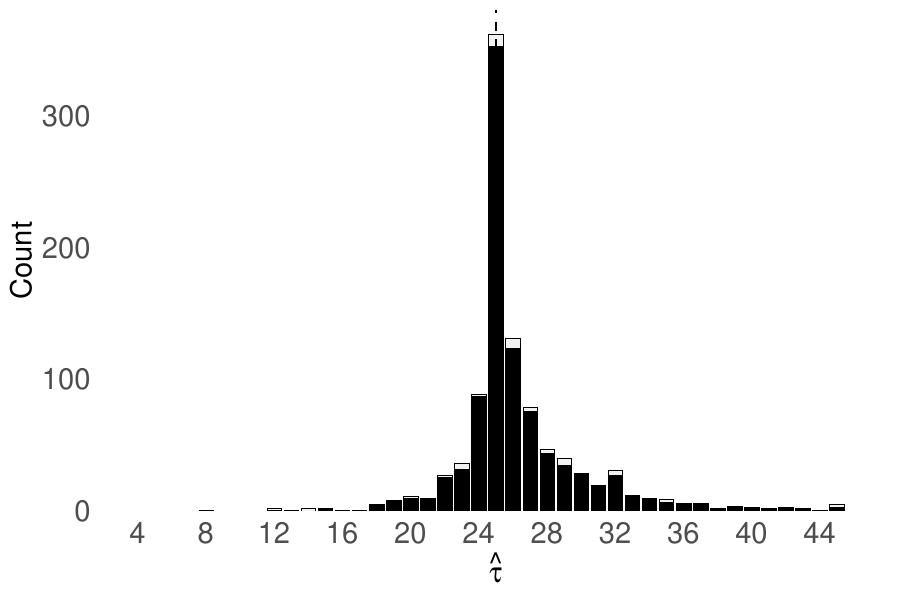} \\
  Example~\ref{exa:4}: $\ell_2$ distance &
  Example~\ref{exa:4}: $\ell_1$ distance \\
  \includegraphics[height=1.0in,width=.48\textwidth]{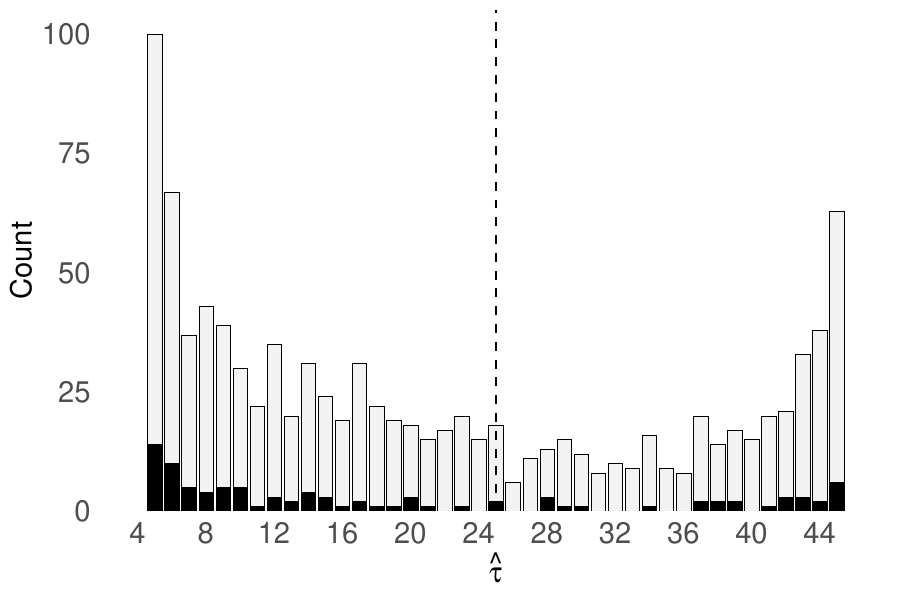} &
  \includegraphics[height=1.0in,width=.48\textwidth]{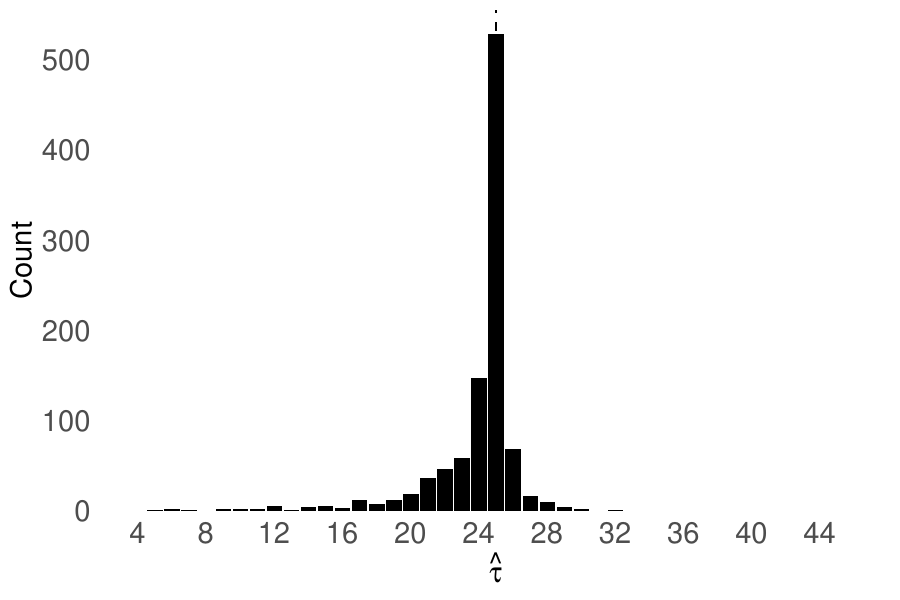} \\
\end{tabular}
%
\caption{{Frequency distributions of potential (grey bar) and detected (black bar) change-points in Examples~\ref{exa:3} and~\ref{exa:4} }}.
\label{fig:bg_histogram_plots_ex3ex4}
\end{figure}

Figure 3 shows that in these two examples, the proposed method based on the $\ell_2$ distance {had poor performance}. But, { using the $\ell_1$ distance, we got excellent results}. This { serves as a motivation} to investigate the high-dimensional behaviour of the method based on other distance functions. 

    \subsection{Methods based on other distance functions}
    \label{sec:3.1}
	
Here we consider the class of generalised distances proposed in \cite{sarkar2018some}. The generalized distance between two $d$-dimensional random vectors $\xvec = (x^{(1)},\ldots, x^{(d)})$ and $\yvec = (y^{(1)},\ldots, y^{(d)})$ is defined as 
	$$\varphi_{h,\psi}(\xvec,\yvec) = h\bigg\{\frac{1}{d}\sum_{k=1}^d \psi\Big(\big|x^{(k)} - y^{(k)}\big|\Big)\bigg\},$$
	where $h,\psi:\R_{+}\to\R_{+}$ are continuous monotonically increasing functions
	with $h(0)=\psi(0)=0$. Note that taking $h(t)=t^{1/p}$ and $\psi(p)=t^{p}$, one gets the $\ell_{p}$ distance (up to a multiplicative factor). Not all choices of $h$ and $\psi$ lead to distance functions. But if $\psi$ satisfies the triangle inequality (e.g., $\psi(t)=1-e^{-t}$), the use of $h(t)=t$
    makes $\varphi_{h,\psi}$ a distance function.
    Using  $\varphi_{h,\psi}(\cdot,\cdot)$, we can { develop a change-point detection method}. For different choices of $t$,
    we can compute
    \begin{align*}
    T^{h,\psi}_{11}\left(t\right) & ={t \choose 2}^{-1}\hspace{-0.2in}\sum\limits_{1\leq i<j\leq t}\varphi_{h,\psi}(\Zvec_{i},\Zvec_{j}),~~T^{h,\psi}_{22}\left(t\right)={n-t \choose 2}^{-1}\hspace{-0.2in}\sum\limits_{t+1\leq i<j\leq n}\varphi_{h,\psi}(\Zvec_{i},\Zvec_{j})\\
     & ~\quad\textnormal{ }\quad T^{h,\psi}_{12}\left(t\right)=\frac{1}{t\left(n-t\right)}\sum\limits_{1\leq i\leq t}\sum\limits_{t+1\leq j\leq n}\varphi_{h,\psi}(\Zvec_{i},\Zvec_{j}) \text{   and  }
    \end{align*}
$\mathcal{D}_{n}^{{h,\psi}}(\mathcal{X}_{t},\mathcal{Y}_{t})=\big\{ T_{12}^{{h,\psi}}\left(t\right)-T_{11}^{{h,\psi}}\left(t\right)\big\} ^{2}+\big\{ T_{12}^{{h,\psi}}\left(t\right)-T_{22}^{{h,\psi}}\left(t\right)\big\} ^{2}$. {We can use}
	\begin{align*}
		\label{eq:gen-estimator}
\widehat{T}_{n}^{h,\psi}=\underset{t\in {\cal A}_n}{\text{argmax }}\frac{t\left(n-t\right)}{n^2}\mathcal{D}_{n}^{{h,\psi}}(\mathcal{X}_{t},\mathcal{Y}_{t}) \text{ and }
      \widehat{S}_{n}^{h,\psi}=\underset{t\in {\cal A}_n}{\text{max }}\frac{t\left(n-t\right)}{n^2}\mathcal{D}_{n}^{{h,\psi}}(\mathcal{X}_{t},\mathcal{Y}_{t}),
    \end{align*} 
     { to develop a change-point detection algorithm as before}. 
     { To investigate its behaviour for HDLSS data, we make the following assumptions} 

    \begin{enumerate}
    \item[(A4)] If $\Zvec = (Z^{(1)},\ldots,Z^{(d)})^{\top}\sim \F_r$ and $\Zvec_{*} = (Z_*^{(1)},\ldots,Z_*^{(d)})^{\top}\sim \F_s$ ($r,s=1,2$) are independent $d$-dimensional random vectors, for ${\bf W}=\Zvec-\Zvec_*$ \begin{enumerate}
        \item 
$\psi(|W^{(q)}|)$'s have uniformly bounded second moments.
         \item $\sum_{q \neq q^{'}}|\text{Corr}\big(\psi(|W^{(q)}|),\psi(|W^{(q^{'})}|)\big)|$
    is of the order $o(d^2)$.
    \end{enumerate}
    \end{enumerate}


Assumption~(A4) ensures WLLN for the sequence of possibly dependent 
and non-identically distributed variables $\{\psi(|W^{(q)}|):q\ge1\}$, i.e. as $d\to\infty$
$\bigl|d^{-1}\sum_{q=1}^{d}\psi(|W^{(q)}|)
-d^{-1}\sum_{q=1}^{d}\E\psi(|W^{(q)}|)\bigr|
\overset{P}{\to} 0$.
If the coordinates of $\mathbf{W}$ are independent, condition~(b) 
in~(A4) is trivially satisfied, and if $\psi$ is bounded 
(e.g., $\psi(t)=1-e^{-t}$), condition~(a) is trivially satisfied as well.

Define  $\phi^*_{h,\psi}(\F_r,\F_s)
=h\bigl(\frac{1}{d}\sum_{q=1}^{d}\E\psi(|Z^{(q)}-Z_*^{(q)}|)\bigr)$,
where $\Zvec\sim\F_r$ and $\Zvec_*\sim\F_s$ ($r,s=1,2$) are independent.
By~(A4) and the continuous mapping theorem, we have
$\varphi_{h,\psi}(\Zvec,\Zvec^*) - \phi^*_{h,\psi}(\F_r,\F_s)\overset{P}{\to}0$
as $d\to\infty$. Note that,
$\mathcal{E}^{(d)}_{h,\psi}(\F_1,\F_2)
=2\phi^*_{h,\psi}(\F_1,\F_2)
-\phi^*_{h,\psi}(\F_1,\F_1)
-\phi^*_{h,\psi}(\F_2,\F_2)$
serves as a measure of distributional difference between $\F_1$ 
and~$\F_2$. In this context, we have the following result 
from~\cite{sarkar2018some}.

\begin{lem}
		\label{lem:sarkar}
		Suppose that $h$ is concave and  {$\psi^{'}(t)/t$} is a non-constant completely
		monotone function. Then $\mathcal{E}^{(d)}_{h,\psi}\left(\F_1,\F_2\right)\geq 0$
		and equality holds if and only if $\F_1$ and $\F_2$ have the same one dimensional marginal
		distributions.
	\end{lem}

     

Note that $\psi(t)=t^2$ used in $\ell_2$ distance does not satisfy this property, but $\psi(t)=1-e^{-t}$ and $\psi(t)=t$ used in $\ell_1$ distance satisfy this. Figure \ref{exa:3,4:exp-dist} shows that in Examples \ref{exa:3} and \ref{exa:4},  the method based on $\psi(t)=1-e^{-t}$ and $h(t)=t$ had { an edge over that} based on the $\ell_1$ distance. { For methods based on these generalised distances we have the following result.}


     \begin{figure}[t]
    \centering
    \setlength{\tabcolsep}{2pt} 
    \renewcommand{\arraystretch}{1} 
    \begin{tabular}{cc}
    Example~\ref{exa:3} & Example~\ref{exa:4} \\
    \includegraphics[height=1.10in,width=0.48\textwidth]{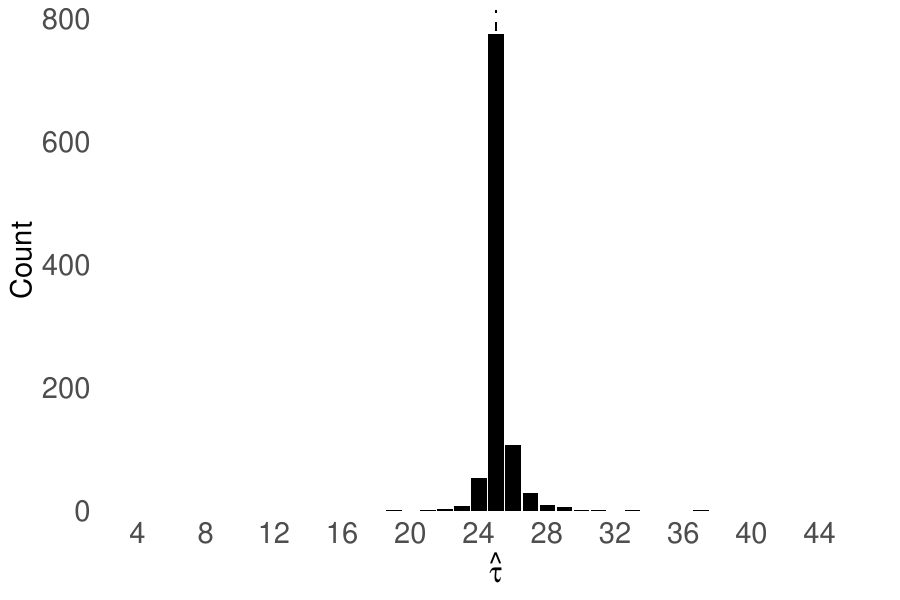} &
    \includegraphics[height=1.10in,width=0.48\textwidth]{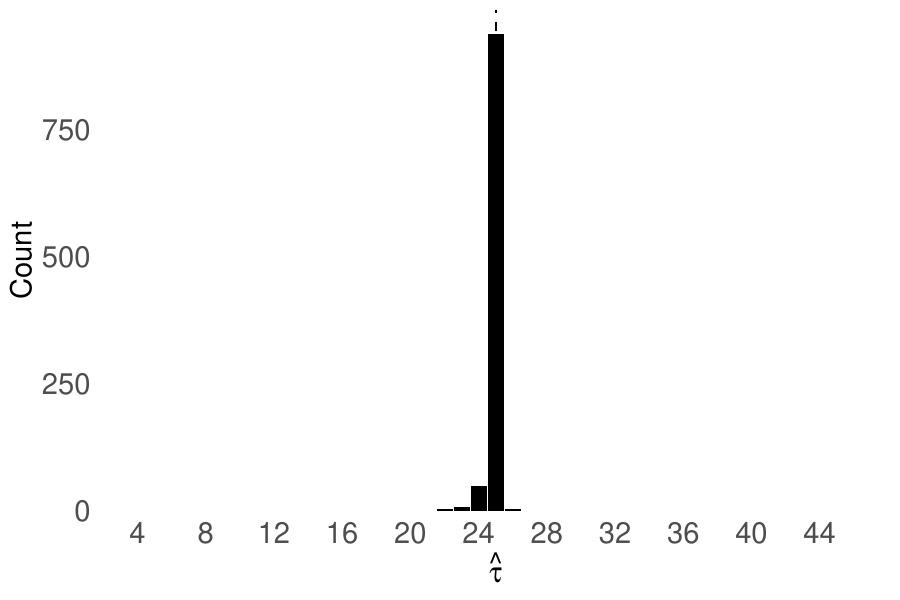} \\
    \end{tabular}
    %
          \caption{{Frequency distribution of the }estimated change-points for Examples~\ref{exa:3} and \ref{exa:4} when $h(t)=t,\;\psi(t)=1-\exp(-t/2)$.}
    \label{exa:3,4:exp-dist}
    \end{figure}
    
	

    \begin{thm}
        \label{thm:HDLSS-gen-consistency} 
        Suppose $\Zvec_1,\ldots, \Zvec_{\tau}\stackrel{iid}{\sim}\F_1$ and $\Zvec_{\tau +1},\ldots, \Zvec_{n}\stackrel{iid}{\sim}\F_2$, where $\F_1$ and $\F_2$ satisfy (A4). Also, assume that $h$ and $\psi$ satisfy the conditions mentioned in Lemma 2, 
        $\underset{d\to\infty}{\lim\,}\mathcal{E}_{h,\psi}^{(d)}\left(\F_1,\F_2\right)>0$ and $\tau \in {\cal A}_n$. Then, 
\begin{enumerate}
    \item[$(a)$] $P({\widehat{ T}}^{h,\psi}_n=\tau)$ converges to $1$ as $d$ tends to infinity.
     
    \item[$(b)$] 
     { If  $|\mathcal{B}({\rho_{n,\tau}})|/n!\le \alpha$, (as in Theorem \ref{thm:HDLSS-consistency}), { the change-point $\tau$ is detected (by ${\widehat T}^{h,\psi}_n$) at level $\alpha$ ($0<\alpha<1$) with probability  converging to one as $d$ grows to infinity.} }
\end{enumerate}
\end{thm}
    {  The assumption $\lim_{d \rightarrow \infty} {\cal E}^{(d)}_{h, \psi}(\F_1, \F_2)>0$ 
    implies that the average of distributional differences along one-dimensional marginals of   $\F_1$ and $\F_2$ is asymptotically non-negligible. This is natural for several high-dimensional alternatives.}
In Examples \ref{exa:3} and \ref{exa:4}, 
Assumption (A4) and the condition $\underset{d\to\infty}{\lim\,}\mathcal{E}_{h,\psi}^{(d)}\left(\F_1,\F_2\right)>0$ holds for both $\psi(t) = t$ and $\psi(t)=1-\exp(-t)$ with $h(t)=t$. As a result, methods based on corresponding distances worked well (henceforth, referred to as Proposed-$\ell_1$ and Proposed-exp). But, for the method based on the $\ell_2$ distance (henceforth, referred to as Proposed-$\ell_2$), we have $\psi(t)=t^2$, $h(t)=t^{1/2}$, 
and hence, $\underset{d\to\infty}{\lim\,}\mathcal{E}_{h,\psi}^{(d)}\left(F_1,F_2\right)=\sqrt{\nu^2+\sigma_1^2+\sigma_2^2} -\sigma_1\sqrt{2}-\sigma_2\sqrt{2}>0$ holds if and only if $\nu^2>0$ or $\sigma_1^2\neq\sigma_2^2$ (as in Theorem \ref{thm:HDLSS-consistency}). This was not the case in these two examples.
\subsection{Performance under sparse signals}
Theorem \ref{thm:HDLSS-consistency} shows the high-dimensional consistency of Proposed-$\ell_2$ when 
$\mu^2>0$ or $\sigma_1\not=\sigma_2$. This assumes that $\|\muvec_{1,d}-\muvec_{2,d}\|^2$ or $\left|\text{trace}(\sigmat_{1,d})-\text{trace}(\sigmat_{2,d})\right|$ grows at least linearly with $d$. However, it is possible to relax this condition in some situations. 
For $\Zvec\sim \F_r$ and $\Zvec_* \sim \F_s$ ($r,s=1,2$), let $v^2_{r,s,d}$ be  the asymptotic order of $\text{Var}\left(\|\Zvec-\Zvec_*\|^2\right)$ (i.e., Var$\left(\|\Zvec-\Zvec_*\|^2\right) \asymp v^2_{r,s,d}$) and $v_d^2=\max\{v^2_{1,1,d},v^2_{1,2,d},v^2_{2,2,d}\}$.
For the high-dimensional  consistency of Proposed-$\ell_2$, we only need the signal (either $\|\muvec_{1,d}-\muvec_{2,d}\|^2$ or $|\text{trace}(\sigmat_{1,d})-\text{trace}(\sigmat_{2,d})|$) to be of higher asymptotic order than that of the noise or stochastic variation $v_d$ (see Theorem \ref{thm:sparse1} below). 
So, if the coordinate variables in $\F_1$ and $\F_2$ are independent or $m$-dependent (i.e., $v_d 
 \asymp d^{1/2}$), it is enough to have $\|\muvec_{1,d}-\muvec_{2,d}\|^2 \gg d^{1/2}$ or $|\text{trace}(\sigmat_{1,d})-\text{trace}(\sigmat_{2,d})|\gg d^{1/2}$ (here $a_d \gg b_d$ means $a_d$ is of higher asymptotic order than $b_d$). 

\begin{thm}
  \label{thm:sparse1}
  Suppose $\Zvec_1,\ldots,\Zvec_{\tau} \stackrel{iid}{\sim} \F_1$ and 
  $\Zvec_{\tau+1},\ldots,\Zvec_{n} \stackrel{iid}{\sim} \F_2$, where $\F_1$ 
  and $\F_2$ satisfy (A1)--(A2). Also assume 
  $\liminf_{d \rightarrow \infty} \min\{\mathrm{trace}(\sigmat_{1,d}),\,
  \mathrm{trace}(\sigmat_{2,d})\}/d > 0$.
  If $\tau \in \mathcal{A}_n$ and $\|\muvec_{1,d} - \muvec_{2,d}\|^2/v_d 
  \rightarrow \infty$ or $|\mathrm{trace}(\sigmat_{1,d}) - 
  \mathrm{trace}(\sigmat_{2,d})|/v_d \rightarrow \infty$ as $d \rightarrow 
  \infty$, we have the following results.
  \begin{enumerate}
    \item[$(a)$] $P(\widehat{T}_n = \tau)$ converges to $1$ as $d$ tends to 
    infinity.
    \item[$(b)$] If $|\mathcal{B}(\rho_{n,\tau})|/n! \le \alpha$,
    the change-point $\tau$ is detected (by $\widehat{T}_n$) at level $\alpha$
    $(0 < \alpha < 1)$ with probability converging to one as $d$ grows to
    infinity.
  \end{enumerate}
\end{thm}


Similarly, for Proposed-$\ell_1$ and Proposed-exp, it is possible to relax the condition $\lim\inf_{d \rightarrow \infty}{{\cal E}}^{(d)}_{h,\psi}(\F_1,\F_2)>0$ used in Theorem \ref{thm:HDLSS-gen-consistency}.
Since $h(t)=t$ in these cases, we have ${{\cal E}}^{(d)}_{h,\psi}(\F_1,\F_2)=\frac{1}{d}{\cal E}_{\psi}(\F_1,\F_2)$ where
\begin{align*}
{\cal E}_{\psi}(\F_1,\F_2) =\sum_{q=1}^{d}{\cal E}_{\psi}(\F_1^{(q)},\F_2^{(q)})=\sum_{q=1}^{d}\left[ 2E\psi|X_1^{(q)}-Y_1^{(q)}|-E\psi|X_1^{(q)}-X_2^{(q)}|-E\psi|Y_1^{(q)}-Y_2^{(q)}|\right],    
\end{align*}
for $\Xvec_1,\Xvec_2$ $\stackrel{iid}{\sim}\F_1$ $\Yvec_1,\Yvec_2\stackrel{iid}{\sim}\F_2$.  So, Theorem \ref{thm:HDLSS-gen-consistency} assumes that ${\cal E}_{\psi}(\F_1,\F_2)$ increases at least linearly with $d$. Now, for 
$\Zvec\sim \F_r$ and $\Zvec^*\sim \F_s$ ($r,s=1,2$), let $v^2_{\psi,r,s,d}$ be the asymptotic order of 
$\text{Var}\big(\sum_{q=1}^{d}
\psi\big|\Zvec^{(q)}-\Zvec^{*(q)}|\big)$ and $v^2_{\psi,d}=\max\{v^2_{\psi,1,1,d},v^2_{\psi,1,2,d},v^2_{\psi,2,2,d}\}$. 
Theorem \ref{thm:sparse2} shows that if  ${\cal E}_{\psi}(\F_1,\F_2) \gg v_{\psi,d}$, we have high dimensional consistency of Proposed-$\ell_1$ and Proposed-exp. So, if the coordinate variables in $\F_1$ and $\F_2$ are independent or $m$-dependent, it is  enough to have ${\cal E}_{\psi}(\F_1,\F_2)\gg d^{-1/2}$. One does not need $\lim\inf_{d \rightarrow \infty} {\cal E}^{(d)}_{h,\psi}(\F_1,\F_2)>0$.  

\begin{thm}
  \label{thm:sparse2}
    Suppose $\Zvec_1,\ldots,\Zvec_{\tau} \stackrel{iid}{\sim} \F_1$ and $\Zvec_{\tau+1},\ldots,\Zvec_{n} \stackrel{iid}{\sim} \F_2$, where  $\F_1$ and $\F_2$ satisfy Assumption (A4). Also assume that ${\cal E}_{\psi}(\F_1,\F_2)/ v_{\psi,d} \rightarrow \infty$ as $d \rightarrow \infty$. Then for Proposed-$\ell_1$ and Proposed-exp, we have the following results.
    \begin{enumerate}
    \item[$(a)$] $P({\widehat{ T}}_n^{h,\psi}=\tau)$ converges to $1$ as $d$ tends to infinity.
   \item[$(b)$] If  $|\mathcal{B}({\rho_{n,\tau}})|/n!\le \alpha$, the change-point $\tau$ is detected (by ${\widehat T}^{h,\psi}_n$) at level $\alpha$ ($0<\alpha<1$) with probability  converging to one as $d$ grows to infinity.
\end{enumerate}
\end{thm}

To demonstrate the performance of the proposed methods under sparse signals, we consider a location and a scale problem, where $\F_1$ and $\F_2$ differ only in $d^{\beta}$ ($0<\beta<1$) many coordinates.

\begin{example}\label{exa:5} $\Zvec_1,\Zvec_2,\ldots,\Zvec_{25}
\stackrel{iid}{\sim}
\mathcal{N}\left(\boldsymbol{0}_{d},\boldsymbol{I}_{d}\right)$
		and $\Zvec_{26
        },\Zvec_{27},\ldots,\Zvec_{50}\stackrel{iid}{\sim}\mathcal{N}\left(\muvec,\boldsymbol{I}_{d}\right)$, where the first $\floor{d^{\beta}}$ many components of $\muvec$ are $1$ and rest are $0$.
	\end{example}

\begin{example}\label{exa:6}
Let $\Zvec_1,\Zvec_2,\ldots,\Zvec_{25}
\stackrel{iid}{\sim}\mathcal{N}\!\left(\boldsymbol{0}_{d},\,\boldsymbol{I}_{d}\right)$
        and 
        $\Zvec_{26},\Zvec_{27},\ldots,\Zvec_{50}
\stackrel{iid}{\sim}\mathcal{N}\!\left(\boldsymbol{0}_{d},\,\boldsymbol{\Sigma}\right)$,
        where $\boldsymbol{\Sigma}
        $ is a diagonal matrix with the first $\lfloor d^{\beta}\rfloor$ elements $3$ and the rest are $1$.
        \end{example}
            
			

For each example, we conducted our experiment with different choices of $d$ ($d=2^i$ for $i=1,2,\ldots,10$). In each case, we repeated the experiment 1000 times and computed the success rates (i.e, the proportions of times the true change-point is detected) for Proposed-$\ell_2$, Proposed-$\ell_1$ and Proposed-exp. These success rates are reported in Figure \ref{fig:combined_success_loc_scale_sparse}.

        \begin{figure}[t]
              \centering
            \includegraphics[width=0.95\textwidth]{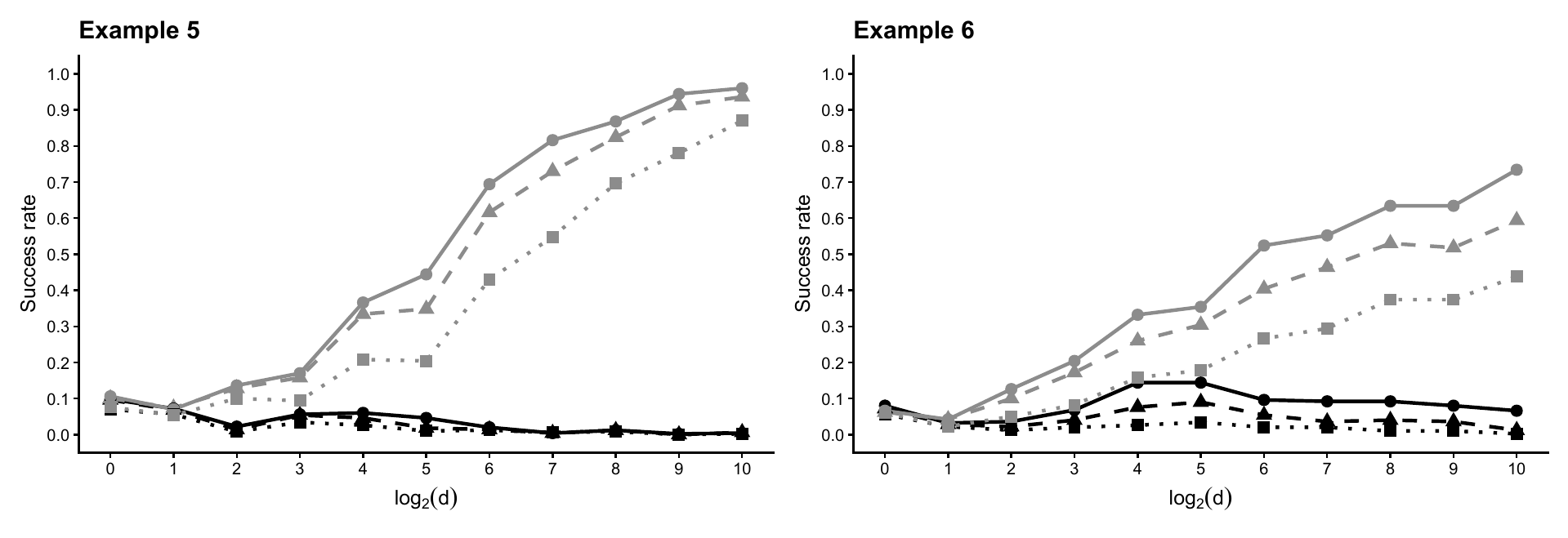}
              \caption{Success rates of Proposed-$\ell_2$ ($\bullet$), Proposed-$\ell_1$ ($\blacktriangle$) and Proposed-exp ($\blacksquare$) in Example \ref{exa:5} (location change) 
                and Example \ref{exa:6} (scale change). Black curves correspond to $\beta = 0.4$ 
                and grey curves to $\beta = 0.6$.}
              \label{fig:combined_success_loc_scale_sparse}
            \end{figure}
In these examples, we have $\nu^2=0$, $\sigma_1^2=\sigma_2^2$
and $\liminf\limits_{d \rightarrow \infty}{{\cal E}}^{(d)}_{h,\psi}(F_1,F_2)=0$
for $\psi(t)=t$ and $\psi(t)=1-e^{-t}$ with $h(t)=t$. So, the conditions of Theorems 2 and 3 do not hold. But for $\beta=0.6$, we have $\|\muvec\|\gg v_d\asymp d^{1/2}$ (in Example \ref{exa:5}) and $|\text{trace}(\sigmat_0)-\text{trace}({\bf I}_d)|\ge v_d$ (in Example \ref{exa:6}). As a result, the success rate of Proposed-$\ell_2$ increased with 
$d$. Similarly,
for Proposed-$\ell_1$ and Proposed-exp, we have $v_{\psi,d} \asymp d^{1/2}$ and ${\cal E}_{\psi}(F_1,F_2)\asymp d^{\beta}$. So, for $\beta=0.6$, they also performed well in higher dimensions. But, for $\beta=0.4$, where the asymptotic order of the signal was lower than $d^{1/2}$, all three methods performed poorly.

\subsection{What happens if the sample size increases with dimension?}

If $n$ also increases with $d$, it is possible to have good performance by the proposed methods in Examples \ref{exa:5} and \ref{exa:6} even for $\beta=0.4$.  Figure \ref{fig:combined_success_loc_scale_HDHSS} shows the performance of the proposed methods
for $\beta=0.4$ and $d=2^i$ with $i=5,6,\ldots,9$ when ($i$) $n$ is fixed at $50$ and ($ii$) $n=50+d$ increases with $d$. In each case,
we repeated the experiment $500$ times and counted the proportion of times a change-point was detected and 
$|{\widehat \theta}_n - \theta|=|{\widehat T}_n-\tau|/n \le 0.01$ (or $|{\widehat \theta}^{h,\psi}_n - \theta|=|{\widehat T}^{h,\psi}_n-\tau|/n \le 0.01$). When $n$ was kept fixed, these proportions, which we call the modified success rates, dropped down slowly as $d$ increased. But when $n$ increased with $d$, they increased steadily. This intrigued us to investigate the behavior of the proposed methods when the sample size increases with the dimension.

\begin{figure}[h]
              \centering
              \includegraphics[width=\textwidth]{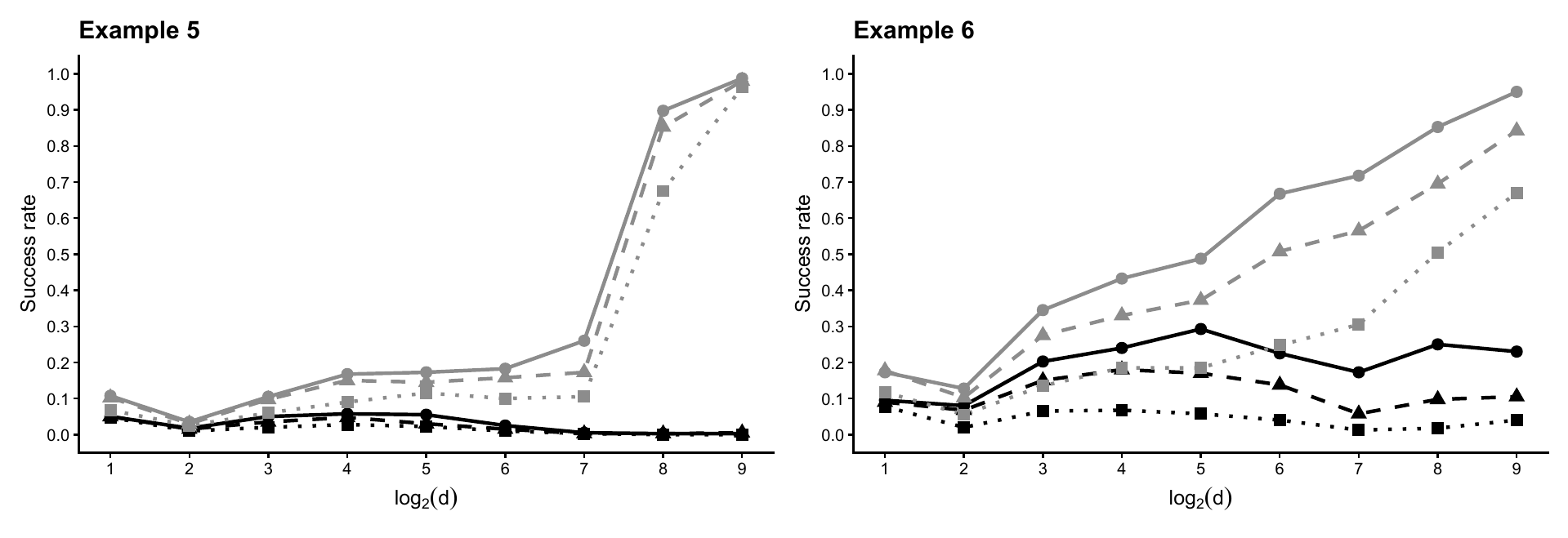}
              \caption{Modified success rates of Proposed-$\ell_2$ ($\bullet$), Proposed-$\ell_1$ ($\blacktriangle$) and Proposed-exp ($\blacksquare$) in Examples~5 and~6 with $\beta = 0.4$. Black curves correspond to the HDLSS setting ($n = 50$, fixed) and grey curves to the HDHSS setting ($n = d + 50$).}             \label{fig:combined_success_loc_scale_HDHSS}
\end{figure}
First we investigate the behavior of Proposed-$\ell_2$. Define $\mu_{11}\left(t\right)=\E T_{11}\left(t\right)$, $\mu_{12}\left(t\right)=\E T_{12}\left(t\right)$ and  $\mu_{22}\left(t\right)=\E T_{22}\left(t\right)$ ($T_{11}(t)$, $T_{12}(t)$ and $T_{22}(t)$ are as in Section 2). Also define $F(t)=\frac{t\left(n-t\right)}{n^{2}}\left[\left\{ \mu_{12}\left(t\right)-\mu_{11}\left(t\right)\right\} ^{2}+\left\{ \mu_{12}\left(t\right)-\mu_{22}\left(t\right)\right\} ^{2}\right].$

For { properly detecting} the change-point, we need the value of $F(t)$ to be higher in the neighborhood of $\tau$ compared to its values in other regions. For any $\delta>0$, one can measure this local difference in signal by $\Delta_{n,d}\left(\delta\right):= F\left(\tau\right)-\max_{t \in {\mathcal A}_n:\left|t-\tau\right|\ge\delta n}F\left(t\right)$. Higher values of $\Delta_{n,d}(\delta)$ facilitate proper detection of the change-point. In addition to this, we also define
$M_{\mu}=\left|\mu_{12}(\tau)-\mu_{11}(\tau)\right|+\left|\mu_{12}(\tau)-\mu_{22}(\tau)\right|,$ 
which can be viewed as a measure of separation between the distributions before and after the change-point. If the local signal difference $\Delta_{n,d}(\delta)$ is not negligible compared to $M_\mu$, we have high-dimensional consistency of Proposed-$\ell_2$ when the sample size also increases with the dimension. This is asserted by the following theorem.

\begin{thm}
\label{thm:HDHSS 1}
Suppose $\Zvec_{1},\dots,\Zvec_{\tau}\overset{iid}{\sim}\F_{1}$
and $\Zvec_{\tau+1},\dots,\Zvec_{n}\overset{iid}{\sim}\F_{2}$, where $\tau=\lfloor n \theta \rfloor$
for some $\theta\in\left(0,1\right)$. Let $\F_1$ and $\F_2$ be sub-exponential distributions, and suppose $n$ increases with $d$ in such a way that $n\delta_n/\log n\to\infty$ and, for any $\delta>0$, $\min\!\left\{\frac{\Delta_{n,d}(\delta)}{M_\mu},\,\sqrt{\Delta_{n,d}(\delta)}\right\}\gg \sqrt{\frac{d \log n }{n\delta_n}}.
$ Then,
\begin{enumerate}
    \item[$(a)$] ${\widehat{ \theta}}_n={\widehat{T}}_n/n$ converges to $\theta$ almost surely.
    \item[$(b)$] For any given level $\alpha ~(0<\alpha<1)$,
    the change-point $\tau$ is detected (by $\widehat T_n$) with probability tending to one.
\end{enumerate}
\end{thm}

{
In Example~\ref{exa:5}, the distributions before and after the change-point differ in $s=d^\beta$ coordinates. Since $\frac{1}{\sqrt d}\{\mu_{12}(\tau)-\mu_{11}(\tau)\}
=\frac{1}{\sqrt d}\{\mu_{12}(\tau)-\mu_{22}(\tau)\}
=\sqrt{2+s/d}-\sqrt2=\frac{1}{2\sqrt2}\frac{s}{d}+o(s/d)$, we have
$M_\mu:=|\eta_{12}-\eta_{11}|+|\eta_{12}-\eta_{22}|\asymp s/\sqrt d=d^{\beta-\frac12}$ and
$\Delta_{n,d}(\delta)\asymp d(s/d)^2=d^{2\beta-1}$ (hence
$\min\{\Delta_{n,d}(\delta)/M_\mu,\linebreak\sqrt{\Delta_{n,d}(\delta)}\}\asymp d^{\beta-\frac12}$). Since the
coordinates are independent Gaussian, the condition of Theorem~\ref{thm:HDHSS 1} can be relaxed, and in this case, we only need (see the Remark after proof of Theorem \ref{thm:HDHSS 1} in Appendix-B)
\[
\min\Big\{\frac{\Delta_{n,d}(\delta)}{M_\mu},\,\sqrt{\Delta_{n,d}(\delta)}\Big\}\gg \sqrt{\frac{ \log n }{n\delta_n}}\iff d^{2\beta-1}\gg \frac{\log n}{n\delta_n}.
\]
If $\delta_n\asymp n^{-1/2}$, we need $\sqrt n\,d^{2\beta-1}\gg \log n$. Note that there is a phase transition at $\beta=\tfrac12$. For $\beta>\tfrac12$ we have
consistency in the HDLSS asymptotic regime, but for $\beta<\tfrac12$ one must let $n$ grow with $d$. Note that for $\beta=0.4$, one needs $\sqrt n\,d^{-0.2}\gg \log n$, or equivalently,
$n\gg d^{0.4}(\log n)^2$. Similar explanations can be given for Example~6 as well.
}
However, if the distance function is bounded (e.g. $h(t)=t$ and $\psi(t)=1-e^{-t}$), the high-dimensional consistency of the proposed
method {follows without the sub-exponential or sub-gaussian assumption.} This result is stated below. 


	

\begin{thm}
\label{HDHSS 2}
Suppose $\Zvec_{1},\dots,\Zvec_{\tau}\overset{iid}{\sim}\F_{1}$
and $\Zvec_{\tau+1},\dots,\Zvec_{n}\overset{iid}{\sim}\F_{2}$, where
$\tau=\lfloor n\theta\rfloor$ for some $\theta\in(0,1)$. If the distance
function $\varphi_{h,\psi}$ is bounded, $n\delta_n/\log n\to\infty$,
and for any $\delta>0$,
$\Delta_{n,d}^{h,\psi}(\delta)\gg\sqrt{\log n/(n\delta_n)}$,
then as $n$ and $d$ grow to infinity,\\
(a) $\hat{\theta}_{n}^{h,\psi}=\hat{T}_{n}^{h,\psi}/n$
converges almost surely to $\theta$.\\
(b) For any given level $\alpha$ $(0<\alpha<1)$, the change-point $\tau$
is detected (by $\hat{T}_{n}^{h,\psi}$) with probability tending to $1$.
\end{thm}

\section{Results from the analysis of simulated data sets} \label{sec:Simulation}

Recall that in the location problem (Example 1),
while MJ had the best performance, followed by KCPA, 
our proposed methods performed well. 
However, in the scale problem 
(Example 2), these proposed methods outperformed their all competitors. We observed the same even when we carried out experiments with the same location and scale shifts, but with an auto-correlation structure among the coordinate variables. Even for the sparse versions of location and scale problems (Example 5 and 6), we observed the same. Therefore, to save space, those results are not reported here. Instead, here we analyse four other data sets, each involving 200-dimensional observations, to compare the empirical performance of our methods with MJ, KCPA, NNG and MST. For each example, we consider three choices of $\tau$ (10, 25 and 40), with $n=50$ in all cases. Each experiment is repeated 500 times, and we only count the cases where the change-points are finally detected. Table \ref{table:2} shows the frequency distribution of the difference between the true change-point and the estimated change-point for different methods. 
Brief descriptions of the four examples are given below.


\begin{example}
\label{exa:7}
$\Zvec_{1},\ldots,\Zvec_{\tau} \sim \mathrm{Unif}(\mathcal{C})$
and
$\Zvec_{\tau+1},\ldots,\Zvec_{n} \sim \mathrm{Unif}(\mathcal{S})$,
where $\mathcal{C}=[-1,1]^{200}$ and
$\mathcal{S}=\{x\in\R^{200}:\|x\|\le r\}$ with $r$ satisfying
$\Vol(\mathcal{C})=\Vol(\mathcal{S})$.
\end{example}
\begin{example}
\label{exa:8}
$\Zvec_{1},\ldots,\Zvec_{\tau} \sim \mathcal{N}(\mathbf{0}_{200},\mathbf{I}_{200})$
and
$\Zvec_{\tau+1},\ldots,\Zvec_{n} \sim \mathrm{MGSN}(d=200,p=0.2,\mathbf{0}_{200}, \mathbf{I}_{200})$, where $\mathrm{MGSN}$ refers to the Multivariate Geometric Skewed Normal Distribution proposed by \cite{Kundu2017}.
\end{example}
\begin{example}
\label{exa:9}
$\Zvec_{1},\ldots,\Zvec_{\tau} \sim \mathcal{N}(\mathbf{0}_{200},\Sigma_{1})$
and
$\Zvec_{\tau+1},\ldots,\Zvec_{n} \sim \mathcal{N}(\mathbf{0}_{200},\Sigma_{2})$, where
$\Sigma_{1}$ and 
$\Sigma_{2}$ are diagonal matrices. The first $d/2$ entries of $\Sigma_{1}$ (respectively, $\Sigma_{2}$) are  $1$
(respectively, 3) and the rest are 3 (respectively, 1).
\end{example}

\begin{example}
\label{exa:10}
$\Zvec_{1},\ldots,\Zvec_{\tau} \sim \mathcal{N}(\mathbf{0}_{200}, 3I_{200})$
and
$\Zvec_{\tau+1},\ldots,\Zvec_{n} \sim \mathcal{T}_{200}(4)$,
where $T_d(k)$ is
the $d$-dimensional distribution with i.i.d. co-ordinate variables, each following the
standard $t$ distribution with $k$ degrees of freedom.
\end{example}

In Example 7, our proposed methods outperformed all their competitors. For $\tau=10$, NNG had a somewhat competitive performance, but for other choices of $\tau$, it performed poorly. 
In Example 8, MJ,  Proposed-$\ell_1$ and Proposed-$\ell_2$ had almost similar performance, with the former one having an edge. But Proposed-exp outperformed all of them. 
KCPA had a competitive performance for $\tau=25$
and $40$, but for $\tau=10$, its performance was much inferior. 

\begin{table}[t]
\setlength{\tabcolsep}{2.5pt}
\renewcommand{\arraystretch}{0.6}
\centering
{\scriptsize
\begin{tabular}{|c|l||c|c|c|c|c|c||c|c|c|c|c|c||c|c|c|c|c|c|}
\hline
\multicolumn{2}{|c|}{} & \multicolumn{6}{c||}{$\tau=10$} & \multicolumn{6}{c||}{$\tau=25$} & \multicolumn{6}{c|}{$\tau=40$} \\ \hline
\multicolumn{2}{|c|}{$|\hat\tau-\tau|$} &  0 & 1 & 2 & 3 & $\geq$4 & Total &  0 & 1 & 2 & 3 & $\geq$4 & Total &  0 & 1 & 2 & 3 & $\geq$4 & Total \\ \hline
\parbox[t]{2mm}{\multirow{7}{*}{\rotatebox[origin=c]{90}{Example 7}}} &  Proposed-$\ell_2$ &  485 & 15 & 0 & 0 & 0 & 500 &  493 & 6 & 1 & 0 & 0 & 500 &  494 & 6 & 0 & 0 & 0 & 500 \\
 &  Proposed-$\ell_1$ &  498 & 2 & 0 & 0 & 0 & 500 &  497 & 2 & 1 & 0 & 0 & 500 &  500 & 0 & 0 & 0 & 0 & 500 \\
 &  Proposed-$\exp$ &  495 & 5 & 0 & 0 & 0 & 500 &  497 & 2 & 1 & 0 & 0 & 500 &  496 & 4 & 0 & 0 & 0 & 500 \\
 &  MJ &  11 & 21 & 14 & 12 & 57 & 115 &  23 & 25 & 20 & 16 & 87 & 171 &  3 & 9 & 9 & 3 & 59 & 83 \\
 &  KCPA &  189 & 88 & 59 & 41 & 123 & 500 &  119 & 83 & 50 & 27 & 221 & 500 &  4 & 8 & 13 & 7 & 468 & 500 \\
 &  NNG &  369 & 70 & 23 & 5 & 20 & 487 &  0 & 0 & 0 & 0 & 379 & 379 &  0 & 0 & 0 & 0 & 92 & 92 \\
 &  MST &  0 & 5 & 4 & 4 & 22 & 35 &  0 & 0 & 0 & 0 & 2 & 2 &  0 & 0 & 0 & 0 & 3 & 3 \\
\hline
\parbox[t]{2mm}{\multirow{7}{*}{\rotatebox[origin=c]{90}{Example 8}}} &  Proposed-$\ell_2$ &  262 & 85 & 47 & 31 & 68 & 493 &  267 & 103 & 45 & 33 & 52 & 500 &  293 & 108 & 45 & 23 & 31 & 500 \\
 &  Proposed-$\ell_1$ &  264 & 84 & 46 & 32 & 67 & 493 &  266 & 101 & 45 & 35 & 53 & 500 &  286 & 111 & 46 & 24 & 33 & 500 \\
 &  Proposed-$\exp$ &  357 & 78 & 26 & 21 & 18 & 500 &  365 & 67 & 35 & 17 & 16 & 500 &  381 & 73 & 30 & 7 & 9 & 500 \\
 &  MJ &  296 & 85 & 44 & 29 & 45 & 499 &  289 & 92 & 45 & 29 & 45 & 500 &  325 & 90 & 39 & 17 & 27 & 498 \\
 &  KCPA &  75 & 29 & 18 & 14 & 49 & 185 &  247 & 83 & 40 & 30 & 62 & 462 &  309 & 62 & 29 & 11 & 29 & 440 \\
 &  NNG &  0 & 0 & 0 & 0 & 44 & 44 &  0 & 1 & 1 & 4 & 80 & 86 &  7 & 7 & 5 & 7 & 54 & 80 \\
 &  MST &  0 & 0 & 0 & 0 & 19 & 19 &  0 & 0 & 0 & 0 & 16 & 16 &  1 & 2 & 0 & 0 & 29 & 32 \\
\hline
\parbox[t]{2mm}{\multirow{7}{*}{\rotatebox[origin=c]{90}{Example 9}}} &  Proposed-$\ell_2$ &  2 & 0 & 3 & 5 & 33 & 43 &  1 & 1 & 0 & 0 & 42 & 44 &  2 & 2 & 4 & 3 & 36 & 47 \\
 &  Proposed-$\ell_1$ &  381 & 52 & 11 & 9 & 28 & 481 &  485 & 10 & 2 & 0 & 3 & 500 &  385 & 31 & 13 & 7 & 33 & 469 \\
 &  Proposed-$\exp$ &  484 & 14 & 2 & 0 & 0 & 500 &  498 & 2 & 0 & 0 & 0 & 500 &  487 & 10 & 0 & 2 & 1 & 500 \\
 &  MJ &  3 & 5 & 1 & 3 & 50 & 62 &  1 & 1 & 1 & 1 & 74 & 78 &  0 & 5 & 8 & 4 & 47 & 64 \\
 &  KCPA &  11 & 19 & 15 & 23 & 430 & 498 &  6 & 6 & 11 & 14 & 462 & 499 &  11 & 19 & 25 & 12 & 433 & 500 \\
 &  NNG &  8 & 8 & 2 & 4 & 49 & 71 &  154 & 96 & 64 & 33 & 42 & 389 &  5 & 8 & 3 & 4 & 59 & 79 \\
 &  MST &  6 & 6 & 6 & 3 & 56 & 77 &  93 & 69 & 45 & 23 & 64 & 294 &  5 & 3 & 2 & 2 & 55 & 67 \\
\hline
\parbox[t]{2mm}{\multirow{7}{*}{\rotatebox[origin=c]{90}{Example 10}}} &  Proposed-$\ell_2$ &  3 & 7 & 2 & 3 & 45 & 60 &  16 & 12 & 11 & 5 & 76 & 120 &  56 & 35 & 24 & 19 & 55 & 189 \\
 &  Proposed-$\ell_1$ &  305 & 109 & 29 & 17 & 34 & 494 &  337 & 98 & 33 & 11 & 20 & 499 &  359 & 85 & 27 & 9 & 19 & 499 \\
 &  Proposed-$\exp$ &  475 & 22 & 2 & 0 & 1 & 500 &  486 & 13 & 1 & 0 & 0 & 500 &  478 & 21 & 1 & 0 & 0 & 500 \\
 &  MJ &  0 & 2 & 0 & 3 & 38 & 43 &  0 & 1 & 1 & 0 & 52 & 54 &  0 & 2 & 3 & 4 & 34 & 43 \\
 &  KCPA &  1 & 0 & 2 & 0 & 12 & 15 &  0 & 0 & 0 & 0 & 11 & 11 &  0 & 0 & 1 & 0 & 9 & 10 \\
 &  NNG &  0 & 0 & 0 & 0 & 32 & 32 &  0 & 0 & 2 & 10 & 24 & 36 &  0 & 0 & 0 & 0 & 39 & 39 \\
 &  MST &  0 & 0 & 0 & 0 & 26 & 26 &  0 & 0 & 0 & 0 & 2 & 2 &  0 & 0 & 0 & 0 & 15 & 15 \\
\hline
\end{tabular}}
\caption{Frequency distribution of $|\hat\tau-\tau|$ in Examples 7-10.
}
\label{table:2}
\end{table}

In Examples 9 and 10, two distributions have the same location (i.e., $\mu^2=0$) and the same average variance (i.e., $\sigma_1^2=\sigma_2^2$). So, all methods based on Euclidean distances
(MJ, KCPA, MST, KNN and Proposed-$\ell_2$) 
performed poorly. 
But Proposed-$\ell_1$ and Proposed-exp, particularly the latter one, had excellent performance. The reason 
has been discussed in Section 3.1.

\section{Multiple Change-Point Detection}
\label{subsec:multiple}

Following the idea of \cite{Matteson2014}, our methods can be easily generalised for multiple change-point detection. First, we look for the most potential change-point in the sequence $\Zvec_1,\ldots,\Zvec_n$. For this purpose, for any fixed $t$ and $s$ ( $t <s$), we compute ${\cal D}_n({\cal Z}_{1:t},{\cal Z}_{t+1:s})$ and ${\cal D}_0(t,s)=\frac{t(s-t)}{s^2}{\cal D}_n({\cal Z}_{1:t},{\cal Z}_{t+1:s})$, where ${\cal Z}_{1:t}=\{\Zvec_1,\ldots,\Zvec_t\}$ and ${\cal Z}_{t+1:s}=\{\Zvec_{t+1},\ldots,\Zvec_s\}$. This is done for all possible choices and $t$ and $s$, and the value of $t$ which leads to the highest value of 
${\cal D}_0(t,s)$ for some $s>t$, is considered as the most potential change-point. Next, we test for its statistical significance. A conditional test based on the permutation principle is used for this purpose. If this change-point turns out to be statistically insignificant, the algorithm
stops by suggesting no existence of change-points in the data. If it is statistically significant, the location corresponding to that change (say, $t$) is considered a detected change-point. In that case, we use the algorithm again on the two sub-sequences $\Zvec_1,\ldots,\Zvec_t$ and 
 $\Zvec_{t+1},\ldots,\Zvec_n$ separately to find other possible change-points in the data. 
    Thus, this recursive algorithm detects multiple change-points in a hierarchical
    manner. 
 
    



To evaluate the performance of the proposed methods, 
we analyse four examples each involving 200-dimensional observations, while $n$ is kept fixed at 60. 
Brief descriptions of these examples are given below.


\begin{example}
\label{exa:11}
$\Zvec_1,\Zvec_2,\ldots,\Zvec_{20} \stackrel{iid}{\sim}\mathcal{N}(\mathbf{0}_{200},\Sigma_{0})$,
$\Zvec_{21},\Zvec_{22},\ldots,\Zvec_{40} \stackrel{iid}{\sim}\mathcal{N}(\mathbf{1}_{200},\Sigma_{0})$,
and $\Zvec_{41}$ $,\Zvec_{42},\ldots,\Zvec_{60} \stackrel{iid}{\sim}\mathcal{N}\!\bigl(\tfrac12\mathbf{1}_{200},\Sigma_{0}\bigr)$, where $\Sigma_{0}$ is an auto-correlation matrix, whose ($i,j$)-th entry ($i,j=1,2,\ldots,200$) is given by $(0.9)^{|i-j|}$.
\end{example}

\begin{example}
\label{exa:12} Consider four Gaussian distributions $F_i=N({\bf 0}_{200},\big(\frac{1}{5}\big)^{i-1}{\bf I}_{200})$ $(i=1,2,3,4)$ differing in their scales. We generated $\Zvec_1,\ldots,\Zvec_{15}$ from $F_1$, $\Zvec_{16},\ldots,\Zvec_{30}$ from $F_2$, $\Zvec_{31},\ldots,\Zvec_{45}$ from $F_3$, and $\Zvec_{46},\ldots,\Zvec_{60}$ from $F_4$.

\end{example}

\begin{example}
\label{exa:13} Let 
$\F_{1}$,
$\F_{2}$,
and $\F_{3}$ be three uniform distributions on the hypersphere
$\mathcal{S}=\{x\in\R^{200}:\|x\|_{\ell_2}\le r_{2}\}$, hypercube $\mathcal{C}=\{x\in\R^{200}:|x_i|\le1\}$ and the $\ell_1$ ball $\mathcal{B}=\{x\in\R^{200}:\|x\|_{\ell_1}\le r_{1}\}$, respectively, where $r_{1}$ and $r_{2}$ are such that the volumes of $\mathcal{B},\mathcal{C}$ and $\mathcal{S}$ are equal. Here we have $\Zvec_{21},\Zvec_{22},\ldots,\Zvec_{40}\stackrel{iid}{\sim}F_1$, $\Zvec_{1},\Zvec_{2},\ldots,\Zvec_{20}\stackrel{iid}{\sim}F_2$ and $\Zvec_{41},\Zvec_{42},\ldots,\Zvec_{60}\stackrel{iid}{\sim}F_3$.
\end{example}

\begin{example}
\label{exa:14} $\Zvec_{1},\ldots,\Zvec_{20} \sim \mathrm{MGSN}(d=200,p=0.2,\mathbf{0}_{200}, \mathbf{I}_{200})$ \citep[see][]{Kundu2017},
$\Zvec_{21}$ $,\ldots,\Zvec_{40} \sim \mathcal{N}(\mathbf{0}_{200},\mathbf{I}_{200})$, and $\Zvec_{41},\ldots,\Zvec_{60} \sim \mathbb{T}_{200}(3)$, the $200$-dimensional standard $t$ distribution with $3$ degrees of freedom.
\end{example}

For each of these examples, we carried out our experiment 500 times. The bar diagrams in Figure \ref{fig:multi_ex13} show the distribution of the detected change-points by different methods. Note that NNG and MST, could not be used for multiple change-point detection. So, along with our proposed methods, here we report the results for MJ and KCPA only.

\begin{figure}[t]
  \centering
  \includegraphics[height=5.05in, width=\textwidth]{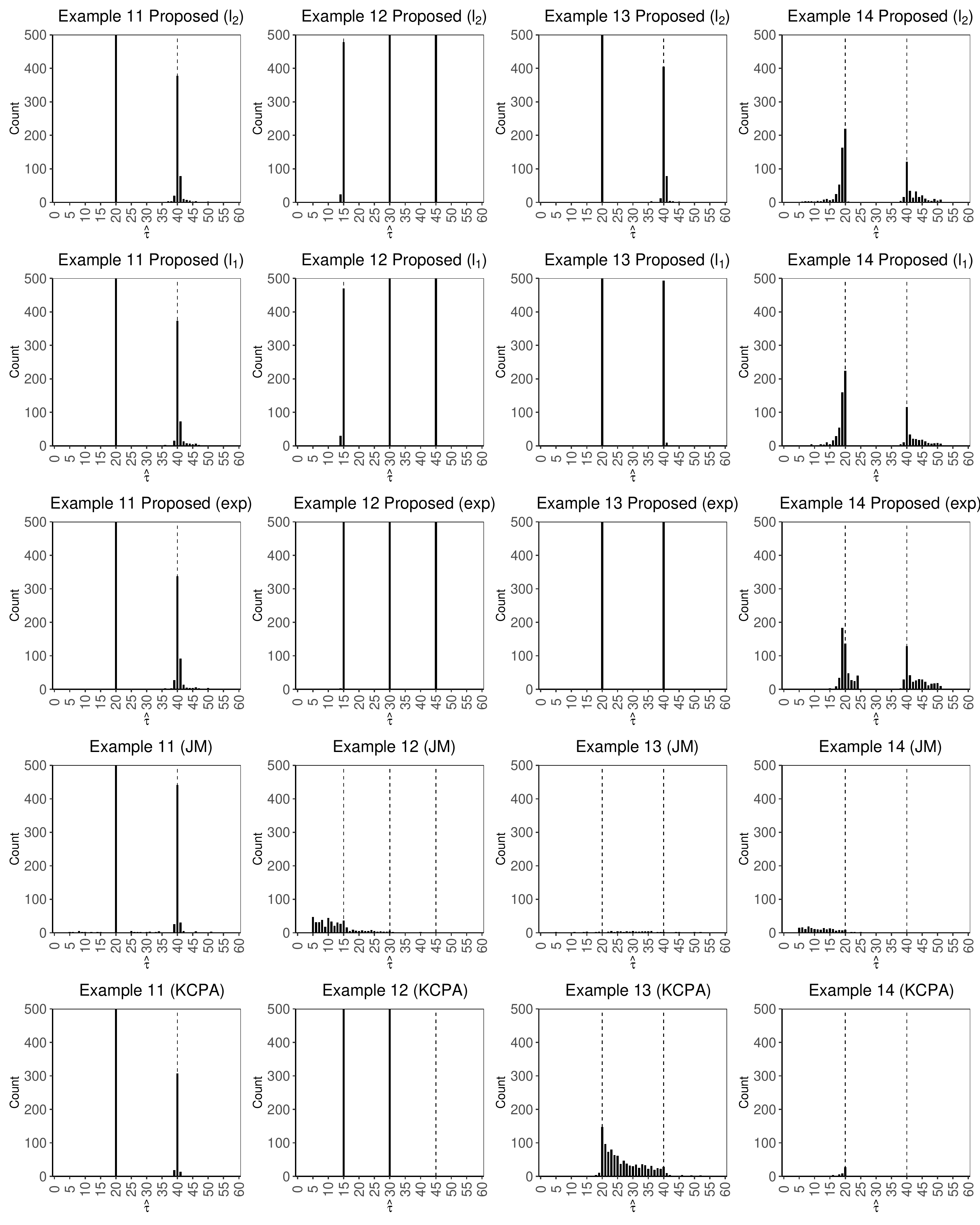}
  
    \caption{The performance of different methods in  Examples \ref{exa:11}-\ref{exa:14}}
  \label{fig:multi_ex13}
\end{figure}

Figure \ref{fig:multi_ex13} shows that MJ had the best performance in detecting location shifts in Example \ref{exa:11}. Here, all competing methods successfully detected the change-point at $\tau=20$ in almost all occasions. But MJ had a higher success rate in detecting the relatively smaller shift at $\tau=40$.

In Example \ref{exa:12}, our proposed methods outperformed MJ and KCPA in detecting the scale changes. 
KCPA could not detect the change at $\tau=45$ even on a single occasion. The performance of MJ was even poorer. 
In Example \ref{exa:13}, 
MJ and KCPA performed poorly, 
but our proposed methods had excellent performance. Among them, Proposed-$\ell_1$ and Proposed-exp had an edge.

The distance convergence results discussed in Section \ref{sec:hd-behaviour} do not hold for the multivariate $t$ distribution considered in Example \ref{exa:14}. We included this example to see how the proposed methods perform when the conditions of Theorems \ref{thm:HDLSS-consistency}-\ref{thm:sparse2} fail to hold. Here, all three proposed methods had similar performance, and they performed much better than MJ and KCPA.

\section{Analysis of stock price data}
\label{subsec:real}
We analyse a stock price data set for further evaluation of the proposed methods. We obtained daily stock prices for all S\&P (Standard \& Poor's) 500 constituents from \href{https://finance.yahoo.com/}{Yahoo Finance} 
and considered the weekly closing prices starting from 5 January 2007 (week 0) to 1 January 2010 (week 156). After excluding equities with missing data, we had 412 stocks observed over these 157 weeks. Following \cite{Hsu1977} and \cite{chen1997testing}, for each stock, we computed week-to-week returns as $R_{t}=({P_{t}}/{P_{t-1}})-1$, where $P_t$ is the closing stock price for week $t$ ($t=1,2,\ldots,156$). Changes in the original stock prices are expected to be reflected in week-to-week returns. During this period, there was a global financial crisis. Widespread defaults on subprime mortgages endangered banks, resulting in unpredictable volatility in global equity markets. The resulting chaos triggered sharp changes in stock prices of various companies. We use our proposed methods as well as other competing methods to  detect these change-points.  

As reported in \cite{financialcrisis2009} and \cite{guillen2009}, 
the 2008 financial crisis was marked by a sequence of key turning points:
(1) the collapse of the Bear Stearns hedge funds (20-Jul-2007, Week 27),  
(2) the official onset of the U.S. recession (28-Dec-2007, Week 50),  
(3) Bank of America’s acquisition of Countrywide (11-Jan-2008, Week 52),  
(4) the bankruptcy of Lehman Brothers (12-Sep-2008, Week 87),  
(5) the bailout of AIG and U.S. Treasury’s guarantee of money-market funds (19-Sep-2008, Week 88),  
(6) the seizure of Washington Mutual (26-Sep-2008, Week 89),  
(7) the Citigroup bailout and the launch of QE1 (28-Nov-2008, Week 98),  
(8) the passage of the American Recovery and Reinvestment Act (13-Feb-2009, Week 109),
(9) Public--Private Investment Program (PPIP) announced. G20 summit; looser accounting rules adopted. (03-Apr-2009, Week 116),
(10) the release of the bank stress test results under SCAP (15-May-2009, Week 122). If we consider events within a couple of weeks as a single combined event, we end up with seven distinct key turning points at weeks 27, 50-52, 87-89, 98, 116 and 122. Figure \ref{fig8:stock:returns} shows the median, first quartile and third quantiles of the returns at different time periods, where these seven turning points are marked using vertical lines or strips.

\begin{figure}[h]
  \centering
\includegraphics[height=2.250in,width=0.75\linewidth]{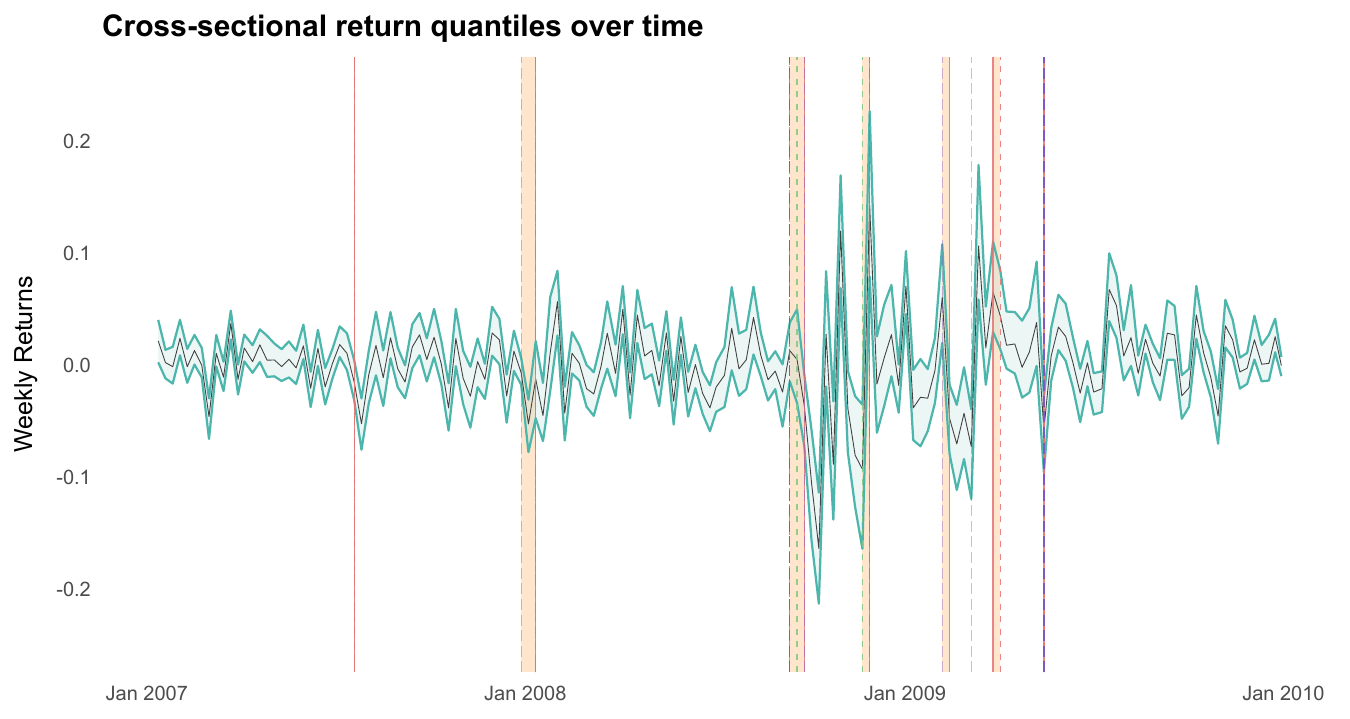}

  \caption{Three quantiles of stock returns over time with vertical lines/strips indicating key turning periods.}
  \label{fig8:stock:returns}
\end{figure}

In this example, while MJ detected only two of these change-points (Week 88 and Week 97), Proposed-$\ell_2$ detected change-points at weeks 27, 52, 87 and 122. Proposed-$\ell_1$ also detected four change-points (at weeks 27, 52, 87, 122), but Proposed-exp could detect all seven turning points at weeks 27, 52, 89, 98, 109, 115 and 122. The performance of KCPA critically depends on the associated penalty parameter $C$.
For small values of $C$, it detected changes at weeks 50, 89, 98, 108 and 122, but it also detected an additional change-point (at week 112) outside the above list; probably  it slightly missed the actual location at weeks 108-109. However, for larger values of $C$, it could detect only two change-points at weeks 87 and 122.

\section{Concluding remarks}
\label{sec:Concluding}
In this article, we proposed some methods for change-point detection 
based on averages of inter-point distances. 
Our methods can be conveniently used for high-dimensional data even when the dimension is much larger than the sample size. We established the high-dimensional consistency of these methods under appropriate regularity conditions, and amply demonstrated their usefulness by analysing several simulated and real data sets.
In high dimensions, while most of the existing methods based on Euclidean distances are unable to detect changes outside the first two moments, Proposed-$\ell_1$ and Proposed-exp can differentiate between two high-dimensional distributions differing in their one-dimensional marginals even when they have the same location and scale. Among these two, we recommend using Proposed-exp as it satisfies the bounded distance function condition of Theorem~\ref{HDHSS 2}. 
However, in some rare cases, where the two high-dimensional distributions have the same one-dimensional marginals but differ in higher-order marginals, like other existing methods, our methods may not work well. To take care of this problem,  we can use a block version of $\ell_1$ or exponential distances. We can make a partition of $d$-dimensional vector $\zvec$ into $b$ blocks $\zvec^{(1)}, \zvec^{(2)}, \ldots,\zvec^{(b)}$,  where the $i$-th block contains $d_i$ variables ($\sum_{i=1}^{b}d_i=d$). The block version of $\varphi_{h,\psi}$ between two observations $\xvec$ and $\yvec$ can be defined as $\varphi^B_{h,\psi}(\xvec,\yvec)=h\Big(\frac{1}{B} \sum_{i=1}^{b}\psi(\|\xvec^{(i)}-\yvec^{(i)}\|)\Big)$. If $b = d$ and $d_i = 1$ for all $i=1,\ldots,b$, it leads to the usual version $\varphi_{h,\psi}$ discussed before. The change-point detection methods  based on the block version of $\ell_1$ distance or exponential distance can discriminate between two high-dimensional distributions unless their marginal block distributions (joint distribution of the random variables in a block) are identical. As long as the block sizes are bounded (which implies $b \rightarrow \infty$ as $d \rightarrow \infty$), results similar to Theorems \ref{thm:HDLSS-gen-consistency} and \ref{thm:sparse2}, with univariate marginals replaced by marginal block distributions, can be proved under similar assumptions. 
For example, if we use blocks of size 2, we can discriminate between two distributions differing in their correlation structures. However, the optimal number of blocks, optimal block sizes and their configurations in a given problem are some open issues to be resolved. This can be investigated as a future research problem.

\bibliographystyle{apalike}
\bibliography{citation_corrected}

\section{Appendix}
\subsection{Appendix A: Some auxiliary mathematical results}
\begin{applem} \label{sub-exponential:conc}
	Let $\Xvec,\Yvec\in\mathbb{R}^{d}$ be independent random vectors whose
	one-dimensional projections are uniformly sub-exponential, i.e.,
	\[
	\sup_{\|v\|_{2}=1}\|\langle v,\Xvec\rangle\|_{\psi_{1}}\le K_{X},
	\qquad
	\sup_{\|v\|_{2}=1}\|\langle v,\Yvec\rangle\|_{\psi_{1}}\le K_{Y},
	\]
	where $\|U\|_{\psi_{1}}:=\inf\{c>0:\ \E\exp(|U|/c)\le2\}$ is the Orlicz norm of a real-valued random variable $U$ \citep[see Chapter 2 from][]{wainwright2019high}.
	Define $K=(K_{X}+K_{Y})$, $\Zvec=\Xvec-\Yvec$, $V=\|\Zvec\|_{2}$, and $\mu=\E V$. 
	Then there exists an absolute constant $c>0$ such that, for all $t\ge0$,
	\[
	\P\big(|V-\mu|\ge t\big)\ \le\
	2\exp\!\left[-\,c\,\min\!\left(\frac{t^{2}}{K^{2}d},\,\frac{t}{K\sqrt d}\right)\right].
	\]
\end{applem}

\begin{proof}
	If $S$ is a random variable with $\|S\|_{\psi_1}\le K$, i.e. $\E\exp(|S|/K)\le 2$, then for any $t>0$, we have
	$
	\P(|S|>t)\le 2e^{-t/K}
	$ \citep[see eq. (2.73) from][]{wainwright2019high}.
	Now, using integration by parts, for $m\ge 1$,  we get
	\[
	\E|S|^m
	= m\int_0^\infty t^{m-1}\P(|S|>t)\,dt
	\le 2m\int_0^\infty t^{m-1}e^{-t/K}\,dt
	=2m!\,K^m.
	\]
	So, we have $|\E S|\le \E|S|\le 2K$. Now, using this result and the triangle inequality, for $m\ge 2$, one gets
	\begin{align*}
		& \E|S-\E S|^{m}\le2^{m-1}\big(\E|S|^{m}+|\E S|^{m}\big)
		& \le2^{m-1}\big(2m!K^{m}+(2K)^{m}\big)\le2^{m+1}m!K^{m}.
	\end{align*}
	So, 
	for $|\lambda|\le 1/(4K)$ (i.e., $1-2K|\lambda|\ge 1/2$), we get
	\begin{align*}
		\E e^{\lambda(S-\E S)}
		&=1+\sum_{m\ge 2}\frac{\lambda^m\E(S-\E S)^m}{m!}
		\le 1+\sum_{m\ge 2}\frac{|\lambda|^m\E|S-\E S|^m}{m!}
		\\
		&\le 1+\sum_{m\ge 2}2^{m+1}(K|\lambda|)^m
		=1+\frac{8(K|\lambda|)^2}{1-2K|\lambda|}\\
		&\le 1 +16(K|\lambda|)^2 \le \exp\big(16K^2\lambda^2\big),
	\end{align*}
	Thus there exist
	$c_1,C_0>0$ (here $c_1=1/4$, $C_0=16$) such that
	\[
	\E e^{\lambda(S-\E S)}\le \exp(C_0K^2\lambda^2)\qquad\text{for }|\lambda|\le c_1/K.
	\]
	As a consequence of the standard Chernoff optimization \citep[see][]{foss2013introduction},
	there exist absolute constants $c,C>0$ such that for all $t\ge 0$,
	\begin{equation}\label{eq:scalar-subexp-tail}
		\P\big(|S-\E S|\ge t\big)\le 2\exp\!\left[-c\min\!\left(\frac{t^2}{K^2},\frac{t}{K}\right)\right].
	\end{equation}
	
	Let $\Zvec=\Xvec-\Yvec$ and $V=\|\Zvec\|_2$. By the triangle inequality for Orlicz norms \citep[see][]{orlicz1991},
	for every unit vector $v$, we have
	\[
	\|\langle v,\Zvec\rangle\|_{\psi_1}
	=\|\langle v,\Xvec\rangle-\langle v,\Yvec\rangle\|_{\psi_1}
	\le \|\langle v,\Xvec\rangle\|_{\psi_1}+\|\langle v,\Yvec\rangle\|_{\psi_1}
	\le K_X+K_Y.
	\]
	So, for $K=(K_X+K_Y)$, we have
	\begin{equation}\label{eq:projZ}
		\sup_{\|v\|_2=1}\|\langle v,\Zvec\rangle\|_{\psi_1}\le K.
	\end{equation}
	In particular, for the coordinate projections $Z_i:=\langle e_i,\Zvec\rangle$, we have
	$\|Z_i\|_{\psi_1}\le K$ for all $i\in\{1,\dots,d\}$, and hence
	\begin{equation}\label{eq:Zi-mom}
		\E|Z_i|^m\le 2m!\,K^m,\qquad m\ge 1.
	\end{equation}
	
	Now, for any integer $m\ge 2$. we get  $\|z\|_2\le d^{1/2-1/m}\|z\|_m$ (follows from H{\"o}lder inequality), or equivalently,
	$\|z\|_2^m\le d^{m/2-1}\sum_{i=1}^d |z_i|^m$, and hence
	\[
	V^m=\|\Zvec\|_2^m\le d^{m/2-1}\sum_{i=1}^d |Z_i|^m.
	\]
	Taking expectations and applying \eqref{eq:Zi-mom}, we get
	\[
	\E V^m\le d^{m/2-1}\sum_{i=1}^d \E|Z_i|^m
	\le d^{m/2}\,2m!\,K^m.
	\]
	Now, using the above bound and the power series, for any $c_0\ge 3$, we have 
	\begin{align*}
		\E\exp\Big(\frac{V}{c_0K\sqrt d}\Big)
		&=\sum_{m\ge 0}\frac{\E V^m}{m!(c_0K\sqrt d)^m}
		\le 1+\sum_{m\ge 1}\frac{2m!K^m d^{m/2}}{m!(c_0K\sqrt d)^m}
		\\
		&=1+2\sum_{m\ge 1}c_0^{-m}
		=1+\frac{2}{c_0-1}
		\le 2.
	\end{align*}
	Therefore, from the definition of Orlicz norm, we have $\|V\|_{\psi_1}\le c_0K\sqrt d$.
	and hence, \eqref{eq:scalar-subexp-tail} gives
	\[
	\P\big(|V-\mu|\ge t\big)=\P\big(|V-\E V|\ge t\big)\le2\exp\!\left[-c\min\!\left(\frac{t^{2}}{K^{2}d},\frac{t}{K\sqrt{d}}\right)\right],
	\] for all $t\ge 0$. 
	Now for $V_d:=V/\sqrt d$ and $\mu_d:=\E V_d$, it is easy to check that
	\[
	\P\big(|V_d-\mu_d|\ge t\big)
	=\P\big(|V-\mu|\ge t\sqrt d\big)
	\le 2\exp\!\left[-c\min\!\left(\frac{t^2}{K^2},\frac{t}{K}\right)\right].
	\]
	This completes the proof.
\end{proof}
\begin{applem}\label{lem:block+perm-U}
	Let $\Zvec_1,\ldots,\Zvec_n$ be independent random vectors in $\mathbb R^d$ such that
	$\sup_{\|v\|_2=1}\ \|\langle v,\Zvec_i-\Zvec_j\rangle\|_{\psi_1}\ \le\ K$ for any $i\neq j$. For $b\in\{2,\ldots,n\}$ define $U_b=\binom{b}{2}^{-1}\sum_{1\le i<j\le b}\|\Zvec_i-\Zvec_j\|$ and $\mu_b=\E U_b$. For $k=\lfloor b/2\rfloor$, there exists an absolute constant $c>0$ such that for all $\varepsilon>0$,
	\begin{align}
		\P\big(|U_b-\mu_b|\ge \varepsilon\big)\ \le\
		2\exp\!\left\{-\,c\,k\,\min\!\Big(\frac{\varepsilon^2}{K^2d},\frac{\varepsilon}{K\sqrt{d}}\Big)\right\},
		\label{sub-expo-conc}
	\end{align}
\end{applem}

\begin{proof}
	Let $g(x,y)=\|x-y\|_2$ and define the centered kernel
	\[
	\nu(\Zvec,\Zvec'):=g(\Zvec,\Zvec')-\E g(\Zvec,\Zvec'), 
	\]
	where $\Zvec,\Zvec'$ are independent random vectors and $g(\Zvec,\Zvec')=\|\Zvec-\Zvec'\|$. From Lemma A.~\ref{sub-exponential:conc} shows that
	there exists $c_0>0$ such that
	\[
	\P\big(|\nu(\Zvec,\Zvec')|\ge x\big)
	\le 2\exp\!\left[-\,c_0\min\!\left(\frac{x^2}{K^2 d},\frac{x}{K\sqrt d}\right)\right],
	\qquad x\ge 0.
	\]
	In particular, we have
	$L_1=2\max\{\exp(c_0/4),1\}$ and $
	L_2:=\frac{c_0}{K\sqrt d}$ such that
	\begin{equation}\label{eq:nu-tail-L1L2}
		\P\big(|\nu(\Zvec,\Zvec')|\ge x\big)
		\le L_1\exp(-L_2 x),\qquad x\ge 0,
	\end{equation}
	
	Note that $U_b$ is a second-order U-statistic with mean $\mu_b$ and 
	\[
	U_b-\mu_b
	=\binom{b}{2}^{-1}\sum_{1\le i<j\le b}\nu(\Zvec_i,\Zvec_j),
	\]
	with $\E(\nu(\Zvec_i,\Zvec_j))=0$ for all $i\neq j$. So, using \eqref{eq:nu-tail-L1L2} and Lemma~A.3 of \cite{Subgaussian_Concentration_Ning2017}, for $k=\lfloor{b}/{2}\rfloor$, we get a constant 
	$c>0$ such that for all $\varepsilon>0$,
	\[
	\P\big(|U_b-\mu_b|\ge \varepsilon\big)
	\le 2\exp\!\left\{-\,c\,k\,
	\min\!\left(\frac{L_2^2\varepsilon^2}{L_1^2},\frac{L_2\varepsilon}{L_1}\right)\right\}.
	\]
	Substituting the explicit values of $L_1$ and $L_2$ yields
	\[
	\P\big(|U_b-\mu_b|\ge \varepsilon\big)
	\le 2\exp\!\left\{-\,c\,k\,
	\min\!\left(\frac{\varepsilon^2}{K^2 d},\frac{\varepsilon}{K\sqrt d}\right)\right\}.
	\]
\end{proof}

\begin{applem}
	\label{lem:perm_stat_vanish_BG}
	Let $\Zvec_{1},\ldots,\Zvec_{\tau}\stackrel{iid}{\sim}\F_{1}$ and 
	$\Zvec_{\tau+1},\ldots,\Zvec_{n}\stackrel{iid}{\sim}\F_{2}$, where
	$\tau=\lfloor n\theta\rfloor,\ \theta\in(0,1)$.
	Let $\Pi$ be a uniform random permutation of $\{1,\ldots,n\}$, independent of
	$\{\Zvec_i\}_{i=1}^n$. For $t\in\{2,\ldots,n-2\}$, define $T^{(\Pi)}_{rs}(t)$ as the permuted version of the statistics $T_{rs}(t)$ in \ref{eq:estimator},
	$w(t):={t(n-t)}/{n^{2}}$ and 
	\[
	D^{(\Pi)}(t):=\{T_{12}^{(\Pi)}(t)-T_{11}^{(\Pi)}(t)\}^{2}+\{T_{12}^{(\Pi)}(t)-T_{22}^{(\Pi)}(t)\}^{2}.
	\]
	Let $\mathcal{A}_{n}:=\{\lfloor n\delta_{n}\rfloor,\ldots,\lceil(1-\delta_{n})n\rceil\}$, where
	$n\delta_n\to\infty$ and ${n\delta_n}/{d} \log n
	\to \infty$ as $n \to \infty$.
	Assume that $\F_1$, $\F_2$ are sub-exponential distributions mentioned in Lemma A.\ref{sub-exponential:conc}. Then, as $n\to\infty$,
	$
	\widehat{S}_n^{(\Pi)}:=\max_{t\in\mathcal{A}_n} w(t)\,D^{(\Pi)}(t)\xrightarrow{\P}0.
	$ Consequently, for any fixed $\alpha\in(0,1)$, the permutation $(1-\alpha)$-quantile cutoff
	$\widehat c_{1-\alpha}$ based on $\{\widehat{S}_n^{(\Pi)}\}$ satisfies $\widehat c_{1-\alpha}\xrightarrow{\P}0.$
\end{applem}

\begin{proof}
	Recall that for $t\in\{2,\dots,n-2\}$, we have 
	\begin{equation}
		T_{11}^{(\Pi)}(t):=\binom{t}{2}^{-1}\sum_{1\le i<j\le t}\norm{\Zvec_{\Pi(i)}-\Zvec_{\Pi(j)}},\label{eq:T11Pi}
	\end{equation}
	\begin{equation}
		T_{22}^{(\Pi)}(t):=\binom{n-t}{2}^{-1}\sum_{t+1\le i<j\le n}\norm{\Zvec_{\Pi(i)}-\Zvec_{\Pi(j)}},\label{eq:T22Pi}
	\end{equation}
	\begin{equation}
		T_{12}^{(\Pi)}(t):=\frac{1}{t(n-t)}\sum_{i=1}^{t}\sum_{j=t+1}^{n}\norm{\Zvec_{\Pi(i)}-\Zvec_{\Pi(j)}}.\label{eq:T12Pi}
	\end{equation}
	\medskip{}
	For independent random vectors $\Xvec,\Xvec^{\prime}\stackrel{iid}{\sim}F_{1}$, $\Yvec,\Yvec^{\prime}\stackrel{iid}{\sim}F_{2}$, let us define
	\begin{equation}
		\eta_{11}:=\E\norm{\Xvec-\Xvec^{\prime}},\quad\eta_{22}:=\E\norm{\Yvec-\Yvec^{\prime}},\quad\eta_{12}:=\E\norm{\Xvec-\Yvec}.\label{eq:mu11mu22mu12}
	\end{equation}
	Define the complete-sample U-statistic and its
	mean as
	\begin{equation}
		U_{n}:=\binom{n}{2}^{-1}\sum_{1\le r<s\le n}\norm{\Zvec_{r}-\Zvec_{s}},\quad\text{and}\quad\bar{\mu}_{n}:=\E U_{n}.\label{eq:mu_n}
	\end{equation}
	A direct decomposition over within- and between-sample pairs yields
	\begin{equation}
		\bar{\mu}_{n}=\binom{n}{2}^{-1}\Big\{\binom{\tau}{2}\eta_{11}+\binom{n-\tau}{2}\eta_{22}+\tau(n-\tau)\eta_{12}\Big\}.\label{eq:bar_mu_n_explicit}
	\end{equation}
	
	\medskip{}
	\noindent Let $K_t$ be the number of observations from $\F_1$ in the set of first $t$ observations in the permuted sample. This random variable is defined as
	\begin{equation}
		K_{t}:=\#\{\Pi(1),\dots,\Pi(t)\}\cap\{1,\dots,\tau\},\label{eq:Kt_def}
	\end{equation}
	Clearly, $K_{t}$ follows a hypergeometric distribution with $\E K_{t}=t\,\frac{\tau}{n}.\label{eq:EKt}$ Define
	\begin{equation}
		a:=K_{t},~b:=t-K_{t},~c:=\tau-K_{t},~d:=(n-\tau)-b=n-t-\tau+K_{t}.\label{eq:abcd_def}
	\end{equation}
	Conditioning on $\Pi$ (or equivalently, on $K_t$), we have 
	\begin{equation}
		\E\big[T_{11}^{(\Pi)}(t)\mid\Pi\big]=G_{11,t}(K_{t}):=\frac{\binom{a}{2}\eta_{11}+\binom{b}{2}\eta_{22}+ab\,\eta_{12}}{\binom{t}{2}},\label{eq:cond_mean_T11}
	\end{equation}
	\begin{equation}
		\E\big[T_{22}^{(\Pi)}(t)\mid\Pi\big]=G_{22,t}(K_{t}):=\frac{\binom{c}{2}\eta_{11}+\binom{d}{2}\eta_{22}+cd\,\eta_{12}}{\binom{n-t}{2}},\label{eq:cond_mean_T22}
	\end{equation}
	\begin{equation}
		\E\big[T_{12}^{(\Pi)}(t)\mid\Pi\big]=G_{12,t}(K_{t}):=\frac{ac\,\eta_{11}+bd\,\eta_{22}+(ad+bc)\,\eta_{12}}{t(n-t)}.\label{eq:cond_mean_T12}
	\end{equation}
	\noindent Taking expectation w.r.t. $K_t$, from \eqref{eq:cond_mean_T11}--\eqref{eq:cond_mean_T12}, we have
	\begin{equation}
		\E T_{11}^{(\Pi)}(t)=\E T_{22}^{(\Pi)}(t)=\E T_{12}^{(\Pi)}(t)=\bar{\mu}_{n}.\label{eq:uncond_means_equal}
	\end{equation}
	
	\medskip{}
	\noindent Since $F_1$ and $F_2$ are sub-exponential, for any independent $\mathbf{U},\mathbf{V}$ with $\mathbf{U}\sim\F_{a}$ and $\mathbf{V}\sim\F_{b}$ for $a,b\in\{1,2\}$,  from Lemma A.~\ref{sub-exponential:conc}, for $x\ge 0$, we have 
	\begin{equation}
		\P\Big(\left|\norm{\mathbf{U}-\mathbf{V}}-\E\norm{\mathbf{U}-\mathbf{V}}\right|\ge x\Big)\le2\exp\Big(\hspace{-0.05in}-c_{0}\min\Big\{ \frac{x^{2}}{K^{2}d},\frac{x}{K\sqrt{d}}\Big\} \Big).\label{eq:pair_tail}
	\end{equation}
	For any fixed $t$, conditioned on $\Pi$ $\{\Zvec_{\Pi(i)}\}_{i=1}^{t}$
	are independent (not necessarily identically distributed), and $T_{11}^{(\Pi)}(t)$
	is a second-order U-statistic. Define 
	\begin{equation}
		v_{ij}:=\norm{\Zvec_{\Pi(i)}-\Zvec_{\Pi(j)}}-\E\norm{\Zvec_{\Pi(i)}-\Zvec_{\Pi(j)}},\qquad1\le i<j\le t.\label{eq:uij_def}
	\end{equation}
	Note that \eqref{eq:pair_tail} holds for each $\nu_{ij}$ uniformly in $i,j$. Note that
	\begin{equation}
		T_{11}^{(\Pi)}(t)-\E[T_{11}^{(\Pi)}(t)\mid\Pi]=\binom{t}{2}^{-1}\sum_{1\le i<j\le t}\nu_{ij}.\label{eq:T11_minus_condmean}
	\end{equation}
	So, it follows from Lemma~A.3 of \citet{Subgaussian_Concentration_Ning2017} that
	there exists a constant $c>0$ such that for all $\varepsilon>0$,
	\begin{equation}
		\P\Big(\big|T_{11}^{(\Pi)}(t)-\E[T_{11}^{(\Pi)}(t)\mid\Pi]\big|\ge\varepsilon\ \Big|\ \Pi\Big)\le2\exp\Big(-c\,\Big\lfloor\frac{t}{2}\Big\rfloor\,\min\Big(\frac{\varepsilon^{2}}{K^{2}d},\frac{\varepsilon}{K\sqrt{d}}\Big)\Big).\label{eq:conc_T11_cond}
	\end{equation}
	Since the bounds do not depend on $\Pi$, removing the conditioning on $\Pi$ gives
	\begin{equation}
		\P\Big(\big|T_{11}^{(\Pi)}(t)-G_{11,t}(K_{t})\big|\ge\varepsilon\Big)\le2\exp\Big(-c\,\Big\lfloor\frac{t}{2}\Big\rfloor\,\min\Big(\frac{\varepsilon^{2}}{K^{2}d},\frac{\varepsilon}{K\sqrt{d}}\Big)\Big).\label{eq:conc_T11_uncond}
	\end{equation}
	The same argument gives
	\begin{equation}
		\P\Big(\big|T_{22}^{(\Pi)}(t)-G_{22,t}(K_{t})\big|\ge\varepsilon\Big)\le2\exp\Big(-c\,\Big\lfloor\frac{n-t}{2}\Big\rfloor\,\min\Big(\frac{\varepsilon^{2}}{K^{2}d},\frac{\varepsilon}{K\sqrt{d}}\Big)\Big).\label{eq:conc_T22_cond}
	\end{equation}
	\begin{equation}
		\P\Big(\big|T_{12}^{(\Pi)}(t)-G_{12,t}(K_{t})\big|\ge\varepsilon\Big)\le2\exp\Big(-c\,m_{t}\,\min\Big(\frac{\varepsilon^{2}}{K^{2}d},\frac{\varepsilon}{K\sqrt{d}}\Big)\Big),\label{eq:conc_T12_cond}
	\end{equation}
	where $m_{t}:=\min\{t,n-t\}$. 
	Define $N(k):=\binom{k}{2}\eta_{11}+\binom{t-k}{2}\eta_{22}+k(t-k)\eta_{12}$. Note that
	\begin{align}
		N(k+1)-N(k) & =k\eta_{11}-(t-k-1)\eta_{22}+(t-2k-1)\eta_{12}\nonumber \\
		& =k(\eta_{11}-\eta_{12})+(t-k-1)(\eta_{12}-\eta_{22}).\label{eq:diff_Nk}
	\end{align}
	Since $G_{11,t}(k)=N(k)/\binom{t}{2}$, for $t\ge2$, we have 
	\begin{align}
		|G_{11,t}(k+1)-G_{11,t}(k)| & =\frac{|N(k+1)-N(k)|}{\binom{t}{2}}\nonumber \\
		& \le\frac{k|\eta_{11}-\eta_{12}|+(t-k-1)|\eta_{12}-\eta_{22}|}{\binom{t}{2}}\nonumber \\
		& \le\frac{tM_{\mu}}{\binom{t}{2}}\le\frac{4M_{\mu}}{t}=:C\,\frac{M_{\mu}}{t},\label{eq:Lipschitz_G11}
	\end{align}
	where $M_{\mu}:=|\eta_{12}-\eta_{11}|+|\eta_{12}-\eta_{22}|$. 
	Now, we first extend $G_{11,t}(\cdot)$ from $\{0,1,\ldots,t\}$ to $[0,t]$ by linear
	interpolation. For any $0\le x \le t$ of the form $x=k+\delta$ (where $k$ is an integer and $0\le \delta<1$), define $G_{11,t}(x)=(x-k)G_{11,t}(k)+(k+1-x)G_{11,t}(k+1)$. Now,  one can prove the  Lipschitz bound
	$|G_{11,t}(x)-G_{11,t}(y)|\le C\frac{M_{\mu}}{t}|x-y|, \text{ for all } x,y\in[0,t]$, and hence
	\begin{equation}
		|G_{11,t}(K_{t})-G_{11,t}(t\tau/n)|\le C\frac{M_{\mu}}{t}|K_{t}-t\tau/n|=CM_{\mu}\Big|\frac{K_{t}}{t}-\frac{\tau}{n}\Big|.\label{eq:G11_Lip_to_K}
	\end{equation}
	Since $K_t$ is hypergeometric using Hoeffding-type concentration for sampling without replacement \citep[see, e.g.,][]{Hoeffding1963} for all $u>0$, we get
	\begin{equation}
		\P\Big(\Big|\frac{K_{t}}{t}-\frac{\tau}{n}\Big|>u\Big)\le2\exp(-2tu^{2}).
		\label{eq:hypergeom_Hoeffding}
	\end{equation}
	
	Combining \eqref{eq:G11_Lip_to_K} and \eqref{eq:hypergeom_Hoeffding},
	for all $\varepsilon>0$, we have
	\begin{align}
		P\big(|G_{11,t}(K_{t})- &G_{11,t}(t\tau/n)|>\varepsilon\big) \le P\Big(\Big|\frac{K_{t}}{t}-\frac{\tau}{n}\Big|>\frac{\varepsilon}{CM_{\mu}}\Big)\nonumber \\
		& \le2\exp\Big(-2t\,\frac{\varepsilon^{2}}{C^{2}M_{\mu}^{2}}\Big)\le2\exp\Big(-c\,t\,\frac{\varepsilon^{2}}{M_{\mu}^{2}}\Big),\label{eq:G11_fluct}
	\end{align}
	for an absolute constant $c>0$. 
	
	Using the same argument as above, for all $\varepsilon>0$, we have
	\begin{equation}
		P\big(|G_{22,t}(K_{t})-G_{22,t}(t\tau/n)|>\varepsilon\big)\le2\exp\Big(-c\,(n-t)\,\frac{\varepsilon^{2}}{M_{\mu}^{2}}\Big),\label{eq:G22_fluct}
	\end{equation}
	\begin{equation}
		P\big(|G_{12,t}(K_{t})-G_{12,t}(t\tau/n)|>\varepsilon\big)\le2\exp\Big(-c\,m_{t}\,\frac{\varepsilon^{2}}{M_{\mu}^{2}}\Big).\label{eq:G12_fluct}
	\end{equation}
	
	\noindent Next, consider the corresponding bias terms 
	$
	b_{11}(t):=|G_{11,t}(t\tau/n)-\bar{\mu}_{n}|$,  $b_{22}(t):=|G_{22,t}(t\tau/n)-\bar{\mu}_{n}|$ and $b_{12}(t):=|G_{12,t}(t\tau/n)-\bar{\mu}_{n}|.
	$ Define
	\noindent  
	\noindent
	\begin{align*}
		&q_{11}:={\binom{\tau}{2}}\Big/{\binom{n}{2}},~ q_{22}:={\binom{n-\tau}{2}}\Big/{\binom{n}{2}},~ q_{12}:={\tau(n-\tau)}\Big/{\binom{n}{2}}, \text{ and  }\\
		&p_{11}(t):={\binom{t\theta}{2}}\Big/{\binom{t}{2}},~ p_{22}(t):={\binom{t(1-\theta)}{2}}\Big/{\binom{t}{2}},~ p_{12}(t):={t^2\theta(1-\theta)}\Big/{\binom{t}{2}},
	\end{align*}
	for $\theta:=\tau/n\in(0,1)$.
	Note that
	\noindent
	\begin{align*}
		G_{11,t}(t\alpha)
		& =p_{11}(t)\eta_{11}+p_{22}(t)\eta_{22}+p_{12}(t)\eta_{12}\\
		& =\eta_{12}+p_{11}(t)(\eta_{11}-\eta_{12})+p_{22}(t)(\eta_{22}-\eta_{12}).
	\end{align*}
	
	\noindent Similarly, we have 
	$\bar{\mu}_{n}=\eta_{12}+q_{11}(\eta_{11}-\eta_{12})+q_{22}(\eta_{22}-\eta_{12})$.
	\noindent Therefore,
	
	\noindent
	\[
	b_{11}(t)=|G_{11,t}(t\alpha)-\bar{\mu}_{n}|\le\big(|p_{11}(t)-q_{11}|+|p_{22}(t)-q_{22}|\big)\,M_{\mu}.
	\]
	
	
	
	\noindent Since $|\theta^{2}-\theta|=\theta(1-\theta)\le1/4$, one can show that
	\noindent
	\[
	|p_{11}(t)-q_{11}|=\Big|\frac{\theta^{2}-\theta}{t-1}-\frac{\theta^{2}-\theta}{n-1}\Big|\le C_0\Big(\frac{1}{t}+\frac{1}{n}\Big),
	\]
	for some $C_0>0$. 
	\noindent Similarly, we have $|p_{22}(t)-q_{22}|\le C_0\big(\frac{1}{t}+\frac{1}{n}\big)$.
	\noindent Consequently,
	
	\noindent
	\[
	b_{11}(t)\le C_0\,M_{\mu}\Big(\frac{1}{t}+\frac{1}{n}\Big).
	\]
	
	\noindent Similar calculations lead to 
	\noindent 
	\noindent
	\[
	b_{22}(t)\le C_0\,M_{\mu}\Big(\frac{1}{n-t}+\frac{1}{n}\Big) \text{  and } b_{12}(t)\le C_0\,\frac{M_{\mu}}{n}.
	\]
	
	
	
	
	
	
	
	\noindent Combining these three, we get the uniform bound
	\noindent
	\begin{align}
		\label{maximum-bias-bound}
		\max\{b_{11}(t),b_{22}(t),b_{12}(t)\}\le C_0\,M_{\mu}\Big(\frac{1}{t}+\frac{1}{n-t}+\frac{1}{n}\Big),    
	\end{align}
	
	\noindent If $t\in\mathcal{A}_n = \{\lfloor n\delta_n\rfloor, \ldots, \lceil n(1-\delta_n)\rceil\}$,
	then $t\ge n\delta_{n}$, $n-t\ge n\delta_{n}$ and
	\noindent
	\[
	\frac{1}{t}+\frac{1}{n-t}+\frac{1}{n}\le\frac{C_0}{n\delta_{n}}.
	\]
	\noindent Hence, taking supremum over the set $\mathcal{A}_{n}$, one gets
	\noindent
	\[
	\sup_{t\in\mathcal{A}_{n}}\max\{b_{11}(t),b_{22}(t),b_{12}(t)\}\le\frac{C_0 M_{\mu}}{n\delta_{n}}=:b_{n}.
	\]
	For any $\varepsilon>0$, take $n$ large so that $b_{n}\le\varepsilon/3$.
	Using the decomposition 
	\begin{equation*}
		|T_{11}^{(\Pi)}(t)-\bar{\mu}_{n}|\le|T_{11}^{(\Pi)}(t)-G_{11,t}(K_{t})|+|G_{11,t}(K_{t})-G_{11,t}(t\tau/n)|+b_{11}(t),\label{eq:decomp_T11}
	\end{equation*}
	and applying \eqref{eq:conc_T11_uncond} and \eqref{eq:G11_fluct}, for all $t\in\mathcal{A}_{n}$,
	we obtain  
	\begin{align}
		& \P\big(|T_{11}^{(\Pi)}(t)-\bar{\mu}_{n}|>\varepsilon\big)\le\label{eq:T11_final_conc} \nonumber \\
		& 2\exp\Big(-c\,\Big\lfloor\frac{t}{2}\Big\rfloor\,\min\Big(\frac{\varepsilon^{2}}{9K^{2}d},\frac{\varepsilon}{3K\sqrt{d}}\Big)\Big)+2\exp\Big(-c\,t\,\frac{\varepsilon^{2}}{M_{\mu}^{2}}\Big).
	\end{align}
	Similarly, we have
	\begin{align}
		& \P\big(|T_{22}^{(\Pi)}(t)-\bar{\mu}_{n}|>\varepsilon\big)\le\label{eq:T22_final_conc} \nonumber \\
		& 2\exp\Big(\hspace{-0.08in}-c\,\Big\lfloor\frac{n-t}{2}\Big\rfloor\min\Big(\frac{\varepsilon^{2}}{9K^{2}d},\frac{\varepsilon}{3K\sqrt{d}}\Big)\Big)+2\exp\Big(\hspace{-0.08in}-c\,(n-t)\,\frac{\varepsilon^{2}}{M_{\mu}^{2}}\Big),\\
		& \P\big(|T_{12}^{(\Pi)}(t)-\bar{\mu}_{n}|>\varepsilon\big)\le \nonumber \\
		& 2\exp\Big(-c\,m_{t}\,\min\Big(\frac{\varepsilon^{2}}{9K^{2}d},\frac{\varepsilon}{3K\sqrt{d}}\Big)\Big)+2\exp\Big(-c\,m_{t}\,\frac{\varepsilon^{2}}{M_{\mu}^{2}}\Big).
	\end{align}
	Now, for $r,s=1,2$, define $\Delta_{rs}(t):=T_{rs}^{(\Pi)}(t)-\bar{\mu}_{n}$.
	Then, we have 
	\begin{equation*}
		T_{12}^{(\Pi)}(t)-T_{11}^{(\Pi)}(t)=\Delta_{12}(t)-\Delta_{11}(t), T_{12}^{(\Pi)}(t)-T_{22}^{(\Pi)}(t)=\Delta_{12}(t)-\Delta_{22}(t).
	\end{equation*}
	This implies 
	$D^{(\Pi)}(t)\le4\big(\Delta_{11}(t)^{2}+\Delta_{12}(t)^{2}+\Delta_{22}(t)^{2}\big). \label{eq:D_upper_by_deltas}
	$
	(follows from the fact that $(a-b)^{2}\le2a^{2}+2b^{2}$). Hence for $\varepsilon>0$ with $\varepsilon_{*}:=\sqrt{\varepsilon/12}$,
	\begin{equation}
		\{D^{(\Pi)}(t)>\varepsilon\}\subseteq\bigcup_{rs\in\{11,12,22\}}\big\{|\Delta_{rs}(t)|>\varepsilon_{*}\big\}.\label{eq:D_event_subset}
	\end{equation}
	Since $m_t=\min\{t,n-t\}$, combining \eqref{eq:D_event_subset} with \eqref{eq:T11_final_conc}--\eqref{eq:T22_final_conc}
	for all $t\in\mathcal{A}_{n}$ and all large $n$, we have 
	\begin{align}
		&\P\big(D^{(\Pi)}(t)>\varepsilon\big) \nonumber\\
		&~~~~\le C\exp\!\Big(-c\,m_{t}\,\min\Big(\frac{\varepsilon_{*}^{2}}{9K^{2}d},\frac{\varepsilon_{*}}{3K\sqrt{d}}\Big)\Big)
		+
		C\exp\!\Big(-c\,m_{t}\,\frac{\varepsilon_{*}^{2}}{M_{\mu}^{2}}\Big),\label{eq:D_tail_final}
	\end{align}
	
	Since $w(t)\le1/4$, applying union-bound, we have 
	\begin{equation}
		\P\Big(\max_{t\in\mathcal{A}_{n}}w(t)D^{(\Pi)}(t)>\varepsilon\Big)
		\le\sum_{t\in\mathcal{A}_{n}}\P\big(D^{(\Pi)}(t)>4\varepsilon\big).\label{eq:max_union_bound}
	\end{equation}
	
	Now set $K_{*}:=3\sqrt{3}\,K.$ and from \eqref{eq:D_tail_final} with $\varepsilon$ replaced by $4\varepsilon$ (note that
	$(4\varepsilon)_{*}=\sqrt{(4\varepsilon)/12}=\sqrt{\varepsilon/3}$), we obtain
	\begin{equation*}
		\P\big(D^{(\Pi)}(t)>4\varepsilon\big)\le
		C\exp\!\Big(-c\,m_{t}\,\min\Big(\frac{\varepsilon}{K_{*}^{2}d},\frac{\sqrt{\varepsilon}}{K_{*}\sqrt{d}}\Big)\Big)
		+
		C\exp\!\Big(-c\,m_{t}\,\frac{\varepsilon}{3M_{\mu}^{2}}\Big).
	\end{equation*}
	Let $m_{*}:=\min_{t\in\mathcal{A}_{n}}m_{t}\asymp n\delta_{n}$ and
	$|\mathcal{A}_{n}|\asymp n$ 
	(here $a_n \asymp b_n$ means $a_n$ and $b_n$ are of the same asymptotic order). Then, for all large $n$,
	\begin{align}
		\P\Big(\max_{t\in\mathcal{A}_{n}}w(t)D^{(\Pi)}(t)>\varepsilon\Big)
		&\le
		C|\mathcal{A}_{n}|\exp\!\Big(-c\,m_{*}\,\min\Big(\frac{\varepsilon}{K_{*}^{2}d},\frac{\sqrt{\varepsilon}}{K_{*}\sqrt{d}}\Big)\Big)\nonumber \\
		&\quad+
		C|\mathcal{A}_{n}|\exp\!\Big(-c\,m_{*}\,\frac{\varepsilon}{3M_{\mu}^{2}}\Big).\label{eq:max_bound_final}
	\end{align}
	
	Since 
	$n\delta_{n}\to\infty$ and $m_{*}\,\min\Big(\frac{\varepsilon}{K_{*}^{2}d},\frac{\sqrt{\varepsilon}}{K_{*}\sqrt{d}}\Big)/\log n \asymp n\delta_n/\log n\rightarrow \infty$ as $n \rightarrow \infty$, for every fixed $\varepsilon>0$, the right-hand side of \eqref{eq:max_bound_final}
	tends to $0$ as $n$ diverges to infinity. Therefore, 
	$S_{n}^{(\Pi)}:=\max_{t\in\mathcal{A}_{n}}w(t)D^{(\Pi)}(t)\xrightarrow{\P}0.\label{eq:SnPi_to0}
	$
	Consequently, for any fixed $\alpha\in(0,1)$, the permutation $(1-\alpha)$-quantile
	$\widehat c_{1-\alpha}$ converges to $0$ in probability.
\end{proof}

\begin{applem}
	\label{lem:perm_conc_BG_gen}
	Let $\Zvec_{1},\ldots,\Zvec_{\tau}\stackrel{iid}{\sim}\F_{1}$ and 
	$\Zvec_{\tau+1},\ldots,\Zvec_{n}\stackrel{iid}{\sim}\F_{2}$, where
	$\tau=\lfloor n\theta\rfloor,\ \theta\in(0,1)$. Let $\Pi$ be a uniform random
	permutation of $\{1,\ldots,n\}$, independent of $\{\Zvec_i\}_{i=1}^n$. For
	$t\in\{2,\ldots,n-2\}$, let $T^{h,\psi,(\Pi)}_{rs}(t)$ denote the permuted
	version of $T^{h,\psi}_{rs}(t)$ defined in sub-section \ref{sec:3.1} with bounded
	distance function $\varphi_{h,\psi}$ satisfying
	$\|\varphi_{h,\psi}\|_\infty\le L$, $w(t):=t(n-t)/n^2$, and
	\[
	D^{h,\psi,(\Pi)}(t):=\{T^{h,\psi,(\Pi)}_{12}(t)-T^{h,\psi,(\Pi)}_{11}(t)\}^{2}
	+\{T^{h,\psi,(\Pi)}_{12}(t)-T^{h,\psi,(\Pi)}_{22}(t)\}^{2}.
	\]
	Since $\mathcal{A}_n = \{\lfloor n\delta_n\rfloor, \ldots, \lceil n(1-\delta_n)\rceil\}$
	where $n\delta_n/\log n\to\infty$. Then, as $n\to\infty$,
	$\widehat{S}_n^{h,\psi,(\Pi)}:=\max_{t\in\mathcal{A}_n}w(t)D^{h,\psi,(\Pi)}(t)
	\xrightarrow{\P}0$, and the corresponding permutation cutoff
	$\hat c_{1-\alpha}(n)\xrightarrow{\P}0$.
\end{applem}

\begin{proof}
	The structure of the argument is identical to the proof of
	Lemma~A.\ref{lem:perm_stat_vanish_BG}; we only indicate the difference. Define $\eta_{rs}^{h,\psi}$, $\bar\mu_n^{h,\psi}$, $K_t$, $a,b,c,d$,
	and the conditional means $G^{h,\psi}_{rs,t}(K_t)$ exactly as in
	\eqref{eq:mu11mu22mu12}--\eqref{eq:cond_mean_T12} with
	$\|\Zvec_i-\Zvec_j\|$ replaced by $\varphi_{h,\psi}(\Zvec_i,\Zvec_j)$;
	\eqref{eq:uncond_means_equal} continues to hold by exchangeability.
	
	The only step in Lemma~A.\ref{lem:perm_stat_vanish_BG} that uses the
	sub-exponential tail \eqref{eq:pair_tail} is the concentration of
	$T^{(\Pi)}_{rs}(t)$ around $G_{rs,t}(K_t)$, and this is the only step where boundedness of $\varphi_{h,\psi}$ enters. 
	
	Conditional on $\Pi$, each of $T^{h,\psi,(\Pi)}_{11}(t)$, $T^{h,\psi,(\Pi)}_{22}(t)$, $T^{h,\psi,(\Pi)}_{12}(t)$ is a U-statistic with bounded kernel $|\varphi_{h,\psi}|\le L$. Since, $T^{h,\psi,(\Pi)}_{11}(t)$ averages over $\binom{t}{2}$ pairs from $\Zvec_{\Pi(1)},\ldots,\Zvec_{\Pi(t)}$; replacing $\Zvec_{\Pi(k)}$ ($k\le t$) by an independent copy perturbs exactly $t-1$ kernel terms each by at most $L/\binom{t}{2}$. So, the total deviation in $T^{h,\psi,(\Pi)}_{11}(t)$ remains bounded by $c^{(11)}_{k,t}=2L/t$. Note that this replacement of $\Zvec_{\Pi(t)}$ does not disturb $T^{h,\psi,(\Pi)}_{22}(t)$ (take $c^{(22)}_{t,k}=0)$, but $T^{h,\psi,(\Pi)}_{12}(t)$, which is an average over $t(n-t)$ cross-pairs, can deviate at most by $c_{t,k}^{(12)}=L/t$. Similarly, for $k>t$ replacement of $\Zvec_{\Pi(k)}$ by an independent term, the deviations in  $T^{h,\psi,(\Pi)}_{11}(t)$, $T^{h,\psi,(\Pi)}_{22}(t)$ and $T^{h,\psi,(\Pi)}_{12}(t)$ remain bounded by $c_{t,k}^{(11)}=0,c_{t,k}^{(22)}=2L/(n-t)$ and $c_{t,k}^{(12)}=L/(n-t)$, respectively.
	So, here we have $\sum_{k=1}^{n}(c^{(11)}_{k,t})^2\le 4L^2/m_t$, $\sum_{k=1}^{n}(c^{(22)}_{k,t})^2\le 4L^2/m_t$ and $\sum_{k=1}^{n}\!\bigl(c^{(12)}_{k,t}\bigr)^2
	\;\le\;\frac{2L^2}{m_t}$, where $m_t=\min\{t,n-t\}$. Therefore, using McDiarmid's inequality, conditional on $\Pi$, we have
	\begin{align}
		\P\!\bigl(|T^{h,\psi,(\Pi)}_{rs}(t)-G^{h,\psi}_{rs,t}(K_t)|>\varepsilon\mid\Pi\bigr)
		&\;\le\; 2\exp\!\bigl(-c\,m_t\,\varepsilon^2/L^2\bigr),\label{eq:conc_Trs_BG}
	\end{align}
	Since this bound is free of $\Pi$, it holds unconditionally as well.
	Since $|\eta_{rs}^{h,\psi}|\le L$, repeating the calculations
	in \eqref{eq:Lipschitz_G11}--\eqref{eq:G11_fluct} with
	$M_\mu$ replaced by $M_\mu^{h,\psi}:=|\mu_{12}^{h,\psi}-\mu_{11}^{h,\psi}|
	+|\mu_{12}^{h,\psi}-\mu_{22}^{h,\psi}|$, one gets
	\begin{equation}\label{eq:G_fluct_BG}
		\P\!\big(|G^{h,\psi}_{rs,t}(K_t)-G^{h,\psi}_{rs,t}(t\tau/n)|>\varepsilon\big)
		\;\le\; 2\exp\!\big(-c\,m_t\,\varepsilon^2/(M_\mu^{h,\psi})^2\big).
	\end{equation}
	Repeating the same argument as in \eqref{maximum-bias-bound}, for $r,s\in\{1,2\}$, we have the following bound for the bias $b_{rs}^{h,\psi}(t):=|G_{rs,t}^{h,\psi}(t\tau/n)-\bar{\mu}_{n}^{h,\psi}|$
	\[
	b_{rs}^{h,\psi}(t)
	\;\le\;
	C M_\mu^{h,\psi}\!\left(\frac{1}{t}+\frac{1}{n-t}+\frac{1}{n}\right),
	\]
	where $C$ is a constant. For $t\in \mathcal{A}_n$, since $t\ge n\delta_n$ and $n-t\ge n\delta_n$, we have
	\begin{align}
		\sup_{t\in\mathcal{A}_n} b_{rs}^{h,\psi}(t)
		\;=\;
		O\!\left({M_\mu^{h,\psi}}/{n\delta_n}\right)
		\;=\; o(1). \label{bias_negligence}    
	\end{align}
	Now, for $\Delta_{rs}^{h,\psi}(t):=T^{h,\psi,(\Pi)}_{rs}(t)-\bar{\mu}_{n}^{h,\psi}$, the triangle inequality gives
	\[
	\bigl|\Delta_{rs}^{h,\psi}(t)\bigr|
	\;\le\;
	\bigl|T^{h,\psi,(\Pi)}_{rs}(t)-G^{h,\psi}_{rs,t}(K_t)\bigr|
	+
	\bigl|G^{h,\psi}_{rs,t}(K_t)-G^{h,\psi}_{rs,t}(t\tau/n)\bigr|
	+
	b_{rs}^{h,\psi}(t).
	\]
	From \eqref{eq:conc_Trs_BG}, it follows that  
	$P(\bigl|T^{h,\psi,(\Pi)}_{rs}(t)-G^{h,\psi}_{rs,t}(K_t)\bigr|>\varepsilon)$ is bounded by
	$2\exp(-c\,m_t\,\varepsilon^2/L^2)$. From 
	\eqref{eq:G_fluct_BG}, we get
	$P(\bigl|G^{h,\psi}_{rs,t}(K_t)-G^{h,\psi}_{rs,t}(t\tau/n)\bigr|>\varepsilon)<2\exp\!\big(-c\,m_t\,\varepsilon^2/(M_\mu^{h,\psi})^2\big)$. From \eqref{bias_negligence}, we get $b_{rs}^{h,\psi}(t)=o(1)$ uniformly on $\mathcal{A}_n$.
	Using these results, for every $\varepsilon>0$ and all large $n$, we have
	$$  \P\!\Bigl(|\Delta_{rs}^{h,\psi}(t)|>\varepsilon\Bigr)
	\;\le\;
	C_1\exp\!\Bigl(-c\,m_t\,\varepsilon^2/L^2\Bigr)
	+C_1\exp\!\Bigl(-c\,m_t\,\varepsilon^2/(M_\mu^{h,\psi})^2\Bigr), 
	$$ where $C_1>0$ is a constant. Since $T^{h,\psi,(\Pi)}_{rs}(t)=\Delta^{h,\psi}_{rs}(t)+\bar{\mu}_n^{h,\psi}$,
	we have $T^{h,\psi,(\Pi)}_{\ell\ell}(t)-T^{h,\psi,(\Pi)}_{12}(t)
	\;=\;
	\Delta^{h,\psi}_{\ell\ell}(t)-\Delta^{h,\psi}_{12}(t),$ for $\ell=1,2$. 
	Now,
	\begin{align}
		& D^{h,\psi,(\Pi)}(t)
		\;=\;
		\bigl(\Delta^{h,\psi}_{11}(t)-\Delta^{h,\psi}_{12}(t)\bigr)^2
		+\bigl(\Delta^{h,\psi}_{22}(t)-\Delta^{h,\psi}_{12}(t)\bigr)^2 \nonumber\\
		&\;\le\;
		2\bigl(\Delta^{h,\psi}_{11}(t)\bigr)^2+2\bigl(\Delta^{h,\psi}_{12}(t)\bigr)^2
		+2\bigl(\Delta^{h,\psi}_{22}(t)\bigr)^2+2\bigl(\Delta^{h,\psi}_{12}(t)\bigr)^2 \nonumber\\
		&\;=\;
		2\bigl(\Delta^{h,\psi}_{11}(t)\bigr)^2
		+4\bigl(\Delta^{h,\psi}_{12}(t)\bigr)^2
		+2\bigl(\Delta^{h,\psi}_{22}(t)\bigr)^2
		\;\le\;
		4\sum_{r,s\in\{1,2\}}\bigl(\Delta^{h,\psi}_{rs}(t)\bigr)^2. \nonumber
	\end{align}
	Since $w(t)\le 1/4$, we have $w(t)\,D^{h,\psi,(\Pi)}(t)\le \sum_{r,s\in\{1,2\}}(\Delta^{h,\psi}_{rs}(t))^2$,
	and hence $\{w(t)\,D^{h,\psi,(\Pi)}(t)>\varepsilon\}$ implies
	$\{|\Delta^{h,\psi}_{rs}(t)|>\varepsilon^{1/2}/2\}$ for some $(r,s)$.
	A union bound over $r,s\in\{1,2\}$ therefore gives
	\begin{align}
		\P\!\Bigl(w(t)\,D^{h,\psi,(\Pi)}(t)>\varepsilon\Bigr)
		&\;\le\;
		\sum_{r,s\in\{1,2\}}\P\!\Bigl(|\Delta_{rs}^{h,\psi}(t)|>\varepsilon^{1/2}/2\Bigr) \nonumber\\
		& \hspace{-1in}\;\le\;
		C_1\exp\!\Bigl(-c_0\,m_t\,\varepsilon/L^2\Bigr)
		+C_1\exp\!\Bigl(-c_0\,m_t\,\varepsilon/(M_\mu^{h,\psi})^2\Bigr), \label{eq:Delta_sq_bound}
	\end{align}
	for some constants $C_1$ and $c_0$. 
	Again, applying a union bound over all $t\in\mathcal{A}_n$ in \eqref{eq:Delta_sq_bound}
	and using the fact $m_t\ge m_*:=\min_{t\in\mathcal{A}_n}m_t\asymp n\delta_n$, we get
	\begin{align}
		&\P\!\left(\max_{t\in\mathcal{A}_n}w(t)\,D^{h,\psi,(\Pi)}(t)>\varepsilon\right)
		\;\le\; \nonumber\\ &
		C_1|\mathcal{A}_n|\exp\!\left(-c_0\,m_*\,\frac{\varepsilon}{L^2}\right)
		+C_1|\mathcal{A}_n|\exp\!\left(-c_0\,m_*\,\frac{\varepsilon}{(M_\mu^{h,\psi})^2}\right). \label{eq:max_bound_BG}
	\end{align}
	Since $|\mathcal{A}_n|\asymp n$ and $m_*\asymp n\delta_n$ with $n\delta_n/\log n\to\infty$,
	both terms on the right of \eqref{eq:max_bound_BG} tend to $0$ for every fixed $\varepsilon>0$.
	Hence $\widehat{S}_n^{h,\psi,(\Pi)}=\max_{t\in\mathcal{A}_n}w(t)\,D^{h,\psi,(\Pi)}(t)\xrightarrow{\P}0$, and the standard Markov argument
	as in Lemma~\ref{lem:perm_stat_vanish_BG} gives $\hat{c}_{1-\alpha}(n)\xrightarrow{\P}0$.
\end{proof}

\begin{applem}
	\label{Uniform Convergence LST}
	Suppose $\Zvec_1,\ldots, \Zvec_{\lfloor n\theta \rfloor}\stackrel{iid}{\sim}F_1$ and
	$\Zvec_{\lfloor n\theta \rfloor +1},\ldots, \Zvec_{n}\stackrel{iid}{\sim}F_2$ where
	$\theta \in(0,1)$. For $\Zvec\sim\F_{r}$ and $\Zvec_{*}\sim\F_{s}$ ($r,s=1,2$), define
	$\mu_{rs}=\E\|\Zvec-\Zvec_{*}\|$. Define the U-statistic $U_n:=\binom{n}{2}^{-1}\sum_{1\le i<j\le n} d^{-1/2}\,\|\Zvec_i-\Zvec_j\|.$ If $\delta_n \rightarrow 0$ and $n \delta_n \rightarrow \infty$ as $n \rightarrow \infty$, then
	\begin{equation}
		\sup_{\theta\in[\delta_n,\,1-\delta_n]}
		\Big|\,U_n-\big[\theta^{2}\eta_{11}+(1-\theta)^{2}\eta_{22}+2\theta(1-\theta)\eta_{12}\big]\Big|
		\ \xrightarrow{a.s.}\ 0.
		\label{eq:Uniform-conv}
	\end{equation}
\end{applem}

\begin{proof} 
	Note that $\sum_{1\leq i<j\leq n}d^{-1/2}\norm{\Zvec_{i}-\Zvec_{j}}$ can be written 
	as 
	\begin{align*}
		& \sum_{1 \le i<j\leq n}d^{-1/2}\norm{\Zvec_{i}-\Zvec_{j}}\\
		& =\left[\sum_{1\le i<j\leq \lfloor n\theta\rfloor}\hspace{-0.15in}d^{-1/2}\norm{\Zvec_{i}-\Zvec_{j}}+\hspace{-0.15in}\sum_{\lfloor n\theta\rfloor<i<j\le n}\hspace{-0.15in}d^{-1/2}\norm{\Zvec_{i}-\Zvec_{j}}+\hspace{-0.15in}\sum_{i\le \lfloor n\theta \rfloor <j}\hspace{-0.15in}d^{-1/2}\norm{\Zvec_{i}-\Zvec_{j}}\right].
	\end{align*}
	Since $\theta>\delta_n$ and $n\delta_n \rightarrow \infty$ as $n \rightarrow \infty$, we  have $\lfloor n \theta\rfloor \rightarrow \infty$ as $n \rightarrow \infty$. Now, using almost sure convergence of U-statistics \citep{VladimirUstat}, $\forall\,\,\epsilon,\delta>0$,
	we get  $N\left(\epsilon,\delta\right)\in\mathbb{N}\,$ such that
	\[
	\mathbb{P}\left[\left|\binom{\lfloor n \theta \rfloor}{2}^{-1}\hspace{-0.2in}\sum_{1\leq i<j\leq \lfloor n \theta \rfloor} \hspace{-0.15in}d^{-1/2}\norm{\Zvec_{i}-\Zvec_{j}}-\eta_{11}\right|<\epsilon\,\,\,\forall\,n\geq N\left(\epsilon,\delta\right)\right]\geq1-\delta.
	\]
	This implies 
	\[
	\left|\binom{n}{2}^{-1}\hspace{-0.2in}\sum_{1\leq i<j\leq\floor{n\theta}}d^{-1/2}\norm{\Zvec_{i}-\Zvec_{j}}-{\floor{n\theta} \choose 2}\binom{n}{2}^{-1}\eta_{11}\right|<{\floor{n\theta} \choose 2}\binom{n}{2}^{-1}\epsilon
	\]
	holds for all $n\geq N\left(\epsilon,\delta\right)$ with probability at least $1-\delta$. This further implies
	\begin{align*}
		& \left|\binom{n}{2}^{-1}\hspace{-0.15in}\sum_{1\leq i<j\leq\floor{n\theta}}d^{-1/2}\norm{\Zvec_{i}-\Zvec_{j}}-\theta^{2}\eta_{11}\right|\\
		& \le \left|\binom{n}{2}^{-1}\hspace{-0.2in}\sum_{1\leq i<j\leq\floor{n\theta}}\hspace{-0.15in}d^{-1/2}\norm{\Zvec_{i}-\Zvec_{j}}-{\floor{n\theta} \choose 2}\binom{n}{2}^{-1}\hspace{-0.15in}\eta_{11}\right|+\left|{\floor{n\theta} \choose 2}\binom{n}{2}^{-1}\hspace{-0.15in}-\theta^{2}\right|\eta_{11}\\
		& <{\floor{n\theta} \choose 2}\binom{n}{2}^{-1}\epsilon+\left|{\floor{n\theta} \choose 2}\binom{n}{2}^{-1}-\theta^{2}\right|\eta_{11}\,\,\,\forall\,{n}\geq N\left(\epsilon,\delta\right)
	\end{align*}
	with probability at least $1-\delta$. 
	Therefore, we have 
	\begin{align*}
		& \sup_{\theta\in\left[\delta_{n},1-\delta_{n}\right]}\left|\binom{n}{2}^{-1}\sum_{1\leq i<j\leq\floor{n\theta}}d^{-1/2}\norm{\Zvec_{i}-\Zvec_{j}}-\theta^{2}\eta_{11}\right|\\
		& <\sup_{\theta\in\left[\delta_{n},1-\delta_{n}\right]}\left[\theta^{2}\epsilon+\left\{ {\floor{n\theta} \choose 2}\binom{n}{2}^{-1}-\theta^{2}\right\} \epsilon+\left|{\floor{n\theta} \choose 2}\binom{n}{2}^{-1}-\theta^{2}\right|\eta_{11}\right]\\
		& <\epsilon+{\sup_{\theta\in\left[\delta_{n},1-\delta_{n}\right]}\left|{\floor{n\theta} \choose 2}\binom{n}{2}^{-1}-\theta^{2}\right|(\epsilon+\eta_{11})}
	\end{align*}
	for all $n>N\left(\epsilon,\delta\right)$ with probability at least $1-\delta$. It is easy to show that \\ 
	$\sup\limits_{\theta\in\left[\delta_{n},1-\delta_{n}\right]}\left| {\floor{n\theta} \choose 2}\binom{n}{2}^{-1}-\theta^{2}\right| =o(1)$. 
	Hence, we have  
	\[
	\sup_{\theta\in\left[\delta_{n},1-\delta_{n}\right]}\left|\binom{n}{2}^{-1}\sum_{1\leq i<j\leq\floor{n\theta}}d^{-1/2}\norm{\Zvec_{i}-\Zvec_{j}}-\theta^{2}\eta_{11}\right|\overset{a.s.}{\longrightarrow}0.
	\]
	
	Similarly, we can show that  
	\[
	\sup_{\theta\in\left[\delta_{n},1-\delta_{n}\right]}\left|\binom{n}{2}^{-1}\sum_{\floor{n\theta}<i<j\le n}d^{-1/2}\norm{\Zvec_{i}-\Zvec_{j}}-(1-\theta)^{2}\eta_{22}\right|\overset{a.s.}{\longrightarrow}0 \text{ as } n \rightarrow \infty.
	\]
	Now, using the result on almost sure convergence for generalised (two-sample) U-statistic  \citep[see e.g.][]{VladimirUstat}  and following similar arguments as above, we get 
	\[
	\sup_{\theta\in\left[\delta_{n},1-\delta_{n}\right]}\left|\binom{n}{2}^{-1}\sum_{i\le\floor{n\theta}<j}d^{-1/2}\norm{\Zvec_{i}-\Zvec_{j}}-2\theta(1-\theta)\eta_{12}\right|\overset{a.s.}{\longrightarrow}0 \text{ as } n \rightarrow \infty.
	\]
	Combining  the results on the three components of $U_n$, the proof is obtained.
\end{proof}


\begin{applem}
	\label{Uniform Convergence LST-2} Under the  assumptions stated in
	Lemma A.\ref{Uniform Convergence LST},
	\[
	\sup_{\tilde{\theta}\in\left[\delta_{n},1-\delta_{n}\right]}\left|\frac{\floor{n\tilde{\theta}}\left(n-\floor{n\tilde{\theta}}\right)}{n^{2}}\mathcal{D}_{n}\left(\mathcal{X}_{\floor{n\tilde{\theta}}},\mathcal{Y}_{\floor{n\tilde{\theta}}}\right)-\varrho\left(\theta,\tilde{\theta}\right)\right|\stackrel{a.s.}{\rightarrow}0
	\]
	as $n$ tends to infinity, where 
	$\varrho\left(\theta,\tilde{\theta}\right)  =\tilde{\theta}\left(1-\tilde{\theta}\right) \big\{A_{\theta}({\tilde \theta}) ~\mathbb{I}_{\tilde{\theta}\leq\theta} + B_{\theta}({\tilde \theta}) ~\mathbb{I}_{\tilde{\theta}>\theta}\big\},$ with ${\mathbb I}_{(\cdot)}$  being the indicator function, and $A_{\theta}({\tilde \theta})$ and $B_{\theta}({\tilde \theta})$ given by
	\begin{align*}
		A_{\theta}({\tilde{\theta}})
		& =\left(\frac{1-\theta}{1-\tilde{\theta}}\right)^{2}\left\{ \left(\eta_{12}-\eta_{11}\right)^{2}+\left(\frac{\theta-\tilde{\theta}}{1-\tilde{\theta}}\left(\eta_{11}-\eta_{12}\right)+\frac{1-\theta}{1-\tilde{\theta}}\left(\eta_{12}-\eta_{22}\right)\right)^{2}\right\},\\
		B_{\theta}({\tilde{\theta}})
		&=\left(\frac{\theta^{2}}{\tilde{\theta}^{2}}\right)\left\{ \left(\left(1-\frac{\theta}{\tilde{\theta}}\right)\left(\eta_{22}-\eta_{12}\right)+\frac{\theta}{\tilde{\theta}}\left(\eta_{12}-\eta_{11}\right)\right)^{2}+\left(\eta_{12}-\eta_{22}\right)^{2}\right\} .
	\end{align*}
\end{applem}

\begin{proof} From Lemma A.~\ref{lem:perm_conc_BG_gen} and its proof, it follows that for ${\tilde{\theta}}\le\theta$,
	as $n\rightarrow\infty$, 
	\begin{align*}
		& \sup_{\delta_{n}\leq\tilde{\theta}\leq\theta}\Big|T_{11}\Big(\floor{n\tilde{\theta}}\Big)-\eta_{11}\Big|\overset{a.s.}{\longrightarrow}0,\\
		& \sup_{\delta_{n}\leq\tilde{\theta}\leq\theta}\bigg|T_{12}\Big(\floor{n\tilde{\theta}}\Big)-\frac{\theta-\tilde{\theta}}{1-\tilde{\theta}}\eta_{11}-\frac{1-\theta}{1-\tilde{\theta}}\eta_{12}\bigg|\overset{a.s.}{\longrightarrow}0,\,\,\text{and}\\
		& \sup_{\delta_{n}\leq\tilde{\theta}\leq\theta}\bigg|T_{22}\Big(\floor{n\tilde{\theta}}\Big)-\left(\frac{\theta-\tilde{\theta}}{1-\tilde{\theta}}\right)^{2}\eta_{11}-\left(\frac{1-\theta}{1-\tilde{\theta}}\right)^{2}\eta_{22}-2\frac{\left(\theta-\tilde{\theta}\right)\left(1-\theta\right)}{\left(1-\tilde{\theta}\right)^{2}}\eta_{12}\bigg|\overset{a.s.}{\longrightarrow}0.
	\end{align*}
	Similarly, for ${\tilde{\theta}}>\theta$, as $n$ tends to infinity,
	we have
	\begin{align*}
		& \sup_{\theta<\tilde{\theta}\leq1-\delta_{n}}\bigg|T_{22}\Big(\floor{n\tilde{\theta}}\Big)-\eta_{22}\bigg|\overset{a.s.}{\longrightarrow}0,\\ &
		\sup_{\theta<\tilde{\theta}\leq1-\delta_{n}}\bigg|T_{12}\Big(\floor{n\tilde{\theta}}\Big)-\frac{\theta-\tilde{\theta}}{1-\tilde{\theta}}\eta_{11}-\frac{1-\theta}{1-\tilde{\theta}}\eta_{12}\bigg|\overset{a.s.}{\longrightarrow}0, \,\,\text{and} \\
		& \sup_{\theta<\tilde{\theta}\leq1-\delta_{n}}\Big|T_{11}\Big(\floor{n\tilde{\theta}}\Big)-\frac{\theta^{2}}{\tilde{\theta}^{2}}\eta_{11}-\left(1-\frac{\theta}{\tilde{\theta}}\right)^{2}\eta_{22}-2\frac{\theta}{\tilde{\theta}}\left(1-\frac{\theta}{\tilde{\theta}}\right)\eta_{12}\Big|\overset{a.s.}{\longrightarrow}0.
	\end{align*}
	Now define $T_{1}\left(\tilde{\theta}\right)=T_{12}\left(\floor{n\tilde{\theta}}\right)-T_{11}\left(\floor{n\tilde{\theta}}\right)$
	and $T_{2}\left(\tilde{\theta}\right)=T_{12}\left(\floor{n\tilde{\theta}}\right)-T_{22}\left(\floor{n\tilde{\theta}}\right).$
	Combining the above results, we get
	\begin{align*}
		& \sup_{\tilde{\theta}\in\left[\delta_{n},1-\delta_{n}\right]}\left|T_{1}\left(\tilde{\theta}\right)-f_{1}\left(\tilde{\theta};\theta\right)\right|\overset{a.s.}{\longrightarrow}0\\
		& \sup_{\tilde{\theta}\in\left[\delta_{n},1-\delta_{n}\right]}\left|T_{2}\left(\tilde{\theta}\right)-f_{2}\left(\tilde{\theta};\theta\right)\right|\overset{a.s.}{\longrightarrow}0
	\end{align*}
	\begin{align*}
		\text{where } f_{1}\left(\tilde{\theta};\theta\right)&=\frac{1-\theta}{1-\tilde{\theta}}\left(\eta_{12}-\eta_{11}\right)\mathbb{I}_{\tilde{\theta}\leq\theta}\\ &+\frac{\theta}{\tilde{\theta}}\left[\left(1-\frac{\theta}{\tilde{\theta}}\right)\left(\eta_{22}-\eta_{12}\right)+\frac{\theta}{\tilde{\theta}}\left(\eta_{12}-\eta_{11}\right)\right]\mathbb{I}_{\tilde{\theta}>\theta}
	\end{align*}
	\begin{align*}
		\text{ and }f_{2}\left(\tilde{\theta};\theta\right) & =\frac{1-\theta}{1-\tilde{\theta}}\left[\frac{\theta-\tilde{\theta}}{1-\tilde{\theta}}\left(\eta_{11}-\eta_{12}\right)+\frac{1-\theta}{1-\tilde{\theta}}\left(\eta_{12}-\eta_{22}\right)\right]\mathbb{I}_{\tilde{\theta}\leq\theta}\\
		& +\frac{\theta}{\tilde{\theta}}\left(\eta_{12}-\eta_{22}\right)\mathbb{I}_{\tilde{\theta}>\theta}.
	\end{align*}
	Next, note that
	\begin{align*}
		& \sup_{\tilde{\theta}\in\left[\delta_{n},1-\delta_{n}\right]}\left|\left[T_{1}\left(\tilde{\theta}\right)\right]^{2}-\left[f_{1}\left(\tilde{\theta};\theta\right)\right]^{2}\right|\\
		& =\sup_{\tilde{\theta}\in\left[\delta_{n},1-\delta_{n}\right]}\left|T_{1}\left(\tilde{\theta}\right)-f_{1}\left(\tilde{\theta};\theta\right)\right|\left|T_{1}\left(\tilde{\theta}\right)+f_{1}\left(\tilde{\theta};\theta\right)\right|.
	\end{align*}
	Since $\sup\limits_{\tilde{\theta}\in\left[\delta_{n},1-\delta_{n}
		\right]}\left|T_{1}\left(\tilde{\theta}\right)-f_{1}\left(\tilde{
		\theta};\theta\right)\right|\overset{a.s.}{\longrightarrow}0$ and 
	$\sup\limits_{\tilde{\theta}\in\left[\delta_{n},1-\delta_{n}\right]}
	\left|T_{1}\left(\tilde{\theta}\right)+f_{1}\left(\tilde{\theta};
	\theta\right)\right|$ is almost surely bounded, (as the function $\left|T_{1}\left(\tilde{\theta}\right)+f_{1}\left(\tilde{\theta};\theta\right)\right|$
	is finite valued for any $\tilde{\theta}\in\left[\delta_{n},1-\delta_{n}\right]$), we have
	\[
	\sup_{\tilde{\theta}\in\left[\delta_{n},1-\delta_{n}\right]}\left|\left[T_{1}\left(\tilde{\theta}\right)\right]^{2}-\left[f_{1}\left(\tilde{\theta};\theta\right)\right]^{2}\right|\overset{a.s.}{\longrightarrow}0.
	\]
	Similarly, we can show almost sure uniform convergence of $\left[T_{2}\left(\tilde{\theta}\right)\right]^2$ to $\left[f_{1}\left(\tilde{\theta};\theta\right)\right]^{2}$. One can check that $\left[f_{1}\left(\tilde{\theta};\theta\right)\right]^{2}+\left[f_{2}\left(\tilde{\theta};\theta\right)\right]^{2}=A_{\theta}({\tilde \theta}) ~\mathbb{I}_{\tilde{\theta}\leq\theta} + B_{\theta}({\tilde \theta})~\mathbb{I}_{\tilde{\theta}>\theta}.$ Therefore, 
	\begin{align*}
		\frac{\floor{n\tilde{\theta}}\left(n-\floor{n\tilde{\theta}}\right)}{n^{2}}\mathcal{D}_{n}\left(\mathcal{X}_{\floor{n\tilde{\theta}}},\mathcal{Y}_{\floor{n\tilde{\theta}}}\right) \overset{a.s.}{\longrightarrow} & \tilde{\theta}\left(1-\tilde{\theta}\right)\Big[A_{\theta}({\tilde \theta}) ~\mathbb{I}_{\tilde{\theta}\leq\theta} + B_{\theta}({\tilde \theta})~\mathbb{I}_{\tilde{\theta}>\theta}\Big]\\
		& =\varrho\left(\theta,\tilde{\theta}\right). 
	\end{align*}
	If we define $\Delta_{1}=\eta_{12}-\eta_{11}$
	and $\Delta_{2}=\eta_{12}-\eta_{22}$, $\varrho\left(\theta,\tilde{\theta}\right)$ can be expressed as
	\[
	\varrho\left(\theta,\tilde{\theta}\right)=\begin{cases}
		\frac{\tilde{\theta}\left(1-\theta\right)^{2}}{1-\tilde{\theta}}\Bigg[\left\{ \left(\Delta_{1}\right)^{2}+\left(\frac{\theta-\tilde{\theta}}{1-\tilde{\theta}}\left(-\Delta_{1}\right)+\frac{1-\theta}{1-\tilde{\theta}}\left(\Delta_{2}\right)\right)^{2}\right\}  & \tilde{\theta}<\theta\\
		\theta\left(1-\theta\right)\left[\left(\Delta_{1}\right)^{2}+\left(\Delta_{2}\right)^{2}\right] & \tilde{\theta}=\theta\\
		\frac{\left(1-\tilde{\theta}\right)\theta^{2}}{\tilde{\theta}}\left\{ \left(\left(1-\frac{\theta}{\tilde{\theta}}\right)\left(-\Delta_{2}\right)+\frac{\theta}{\tilde{\theta}}\left(\Delta_{1}\right)\right)^{2}+\left(\Delta_{2}\right)^{2}\right\}  & \tilde{\theta}>\theta
	\end{cases}
	\]
	Using Jensen's inequality, for $\tilde{\theta}<\theta$, we have
	\begin{align*}
		&\varrho(\theta,\theta)-\varrho(\theta,{\tilde \theta})\\
		&=\theta\left(1-\theta\right)\left[\left(\Delta_{1}\right)^{2}+\left(\Delta_{2}\right)^{2}\right]\\
		& ~~~~~~~~~~~~~~~~~~~~~ - \frac{\tilde{\theta}\left(1-\theta\right)^{2}}{1-\tilde{\theta}}\left[\left(\Delta_{1}\right)^{2}+\left\{ \frac{\theta-\tilde{\theta}}{1-\tilde{\theta}}\left(-\Delta_{1}\right)+\frac{1-\theta}{1-\tilde{\theta}}\left(\Delta_{2}\right)\right\} ^{2}\right]\\
		& \geq\left(1-\theta\right)\tilde{\theta}\Bigg[\left(\Delta_{1}\right)^{2}\left\{ \frac{\theta}{\tilde{\theta}}-\frac{1-\theta}{1-\tilde{\theta}}\left(2-\frac{1-\theta}{1-\tilde{\theta}}\right)\right\} 
		+\left(\Delta_{2}\right)^{2}\left\{ \frac{\theta}{\tilde{\theta}}-\left(\frac{1-\theta}{1-\tilde{\theta}}\right)^{2}\right\}.
	\end{align*}
	For $\tilde{\theta}<\theta$,
	since $\frac{\theta}{\tilde{\theta}}-\frac{1-\theta}{1-\tilde{\theta}}\left(2-\frac{1-\theta}{1-\tilde{\theta}}\right)$
	and $\frac{\theta}{\tilde{\theta}}-\left(\frac{1-\theta}{1-\tilde{\theta}}\right)^{2}$ both are positive, the quantity on the right side of the above inequality is also strictly positive. Similarly
	we can show that $\varrho(\theta,\theta)-\varrho(\theta,{\tilde \theta})>0$ for $\tilde{\theta}>\theta$. Therefore, 
	$\varrho\left(\theta,\tilde{\theta}\right)$ is maximized at $\tilde{\theta}=\theta$.
\end{proof}

\begin{applem}
	\label{HDLSS Limits} Under Assumption (A1)-(A3), as $d$ tends to infinity,
	\begin{equation}
		\frac{1}{\sqrt{d}}\norm{\Zvec_{i}-\Zvec_{j}}\overset{\P}{\longrightarrow}\begin{cases}
			\sqrt{2}\sigma_{1} & i,j\leq\tau\\
			\sqrt{2}\sigma_{2} & i,j>\tau\\
			\sqrt{\nu^{2}+\sigma_{1}^{2}+\sigma_{2}^{2}} & i\leq\tau,j>\tau
		\end{cases}\label{eq:HDLSS-limits}
	\end{equation}
\end{applem}

\begin{proof}
	Note that under (A1) and (A2), the weak law of large number holds, i.e., $d^{-1}\Big|\norm{\Zvec_{i}-\Zvec_{j}}^{2}-\E\norm{\Zvec_{i}-\Zvec_{j}}^{2}\Big| \stackrel{P}{\rightarrow}0$ as $d \rightarrow \infty$, and under (A3), we can compute $\lim_{d\to\infty}d^{-1}\E\norm{\Zvec_{i}-\Zvec_{j}}^{2}$ for different choices of $i$ and $j$. 
	First note that if $\Zvec \sim F$, 
	$\E\norm{\Zvec}^{2}=\norm{\muvec_{\F}}^{2}+\text{tr}\left(\sigmat_{\F}\right)$, where
	$\muvec_{\F}$ and $\sigmat_{\F}$ are the mean vector and
	variance-covariance matrix of the distribution $\F$. 
	
	Now, for $i,j\leq\tau$, we have
	\begin{align*}
		\E\norm{\Zvec_{i}-\Zvec_{j}}^{2}
		& =\E\norm{\Zvec_{i}}^{2}+\E\norm{\Zvec_{j}}^{2}-2\E\left(\Zvec_{i}^{T}\Zvec_{j}\right) \\ & =2\norm{\muvec_{\F_{1}}}^{2}+2\text{tr}\left(\sigmat_{\F_{1}}\right)-2\norm{\muvec_{\F_{1}}}^{2}
		=2\text{tr}\left(\sigmat_{\F_{1}}\right).
	\end{align*}
	This implies $\lim_{d\to\infty}d^{-1}\E\norm{\Zvec_{i}-\Zvec_{j}}^{2}=2\lim_{d\to\infty}d^{-1}\text{tr}\left(\sigmat_{\F_{1}}\right)=2\sigma_{1}^{2}$.
	Similarly, for $i,j>\tau$. we have $\lim_{d\to\infty}d^{-1}\E\norm{\Zvec_{i}-\Zvec_{j}}^{2}=2\sigma_{2}^{2}$. 
	Again, for $i\leq\tau,j>\tau$, we have 
	\begin{align*}
		\E\norm{\Zvec_{i}-\Zvec_{j}}^{2}
		& =\E\norm{\Zvec_{i}}^{2}+\E\norm{\Zvec_{j}}^{2}-2\E\left(\Zvec_{i}^{T}\Zvec_{j}\right)\\
		& =\norm{\muvec_{\F_{1}}}^{2}+\text{tr}\left(\sigmat_{\F_{1}}\right)+\norm{\muvec_{\F_{2}}}^{2}+\text{tr}\left(\sigmat_{\F_{2}}\right)-2\muvec_{\F_{1}}^{T}\muvec_{\F_{2}}\\
		& =\norm{\muvec_{\F_{1}}-\muvec_{\F_{2}}}^{2}+\text{tr}\left(\sigmat_{\F_{1}}\right)+\text{tr}\left(\sigmat_{\F_{2}}\right).
	\end{align*}
	Therefore, 
	$
	\lim_{d\to\infty}d^{-1}\E\norm{\Zvec_{i}-\Zvec_{j}}^{2}
	=\nu^{2}+\sigma_{1}^{2}+\sigma_{2}^{2}.
	$ The proof follows from these facts using the continuous mapping theorem.
\end{proof}

\begin{applem}\label{sub-gaussian:conc_unscaled}
	Let $\Zvec_1,\ldots,\Zvec_n\in\R^d$ be independent random vectors with 
	independent coordinates such that $\max_{i,q}\|Z_i^{(q)}\|_{\psi_2}\le K$,
	where $\|U\|_{\psi_2}:=\inf\{c>0:\ \E\exp(U^2/c^2)\le 2\}$.
	Define $U_b:=\binom{b}{2}^{-1}\sum_{1\le i<j\le b}\|\Zvec_i-\Zvec_j\|_2$ for $b\in\{2,\ldots,n\}$ 
	and $\mu_b:=\E U_b$. Then, for $k=\lfloor b/2\rfloor$, there exist absolute 
	constants $c,C>0$ such that the following two inequalities hold.
	
	\noindent{(a)} For every $i\neq j$ and $x\ge 0$,
	\begin{equation}
		\label{eq:sg_kernel_tail_unscaled}
		\P\big(|\|\Zvec_i-\Zvec_j\|_2-\E\|\Zvec_i-\Zvec_j\|_2|\ge x\big)
		\;\le\; 2\exp(-c\, x^2/K^2).
	\end{equation}
	
	\noindent{(b)} For every $\varepsilon>0$,
	\begin{equation}
		\label{eq:sg_ustat_tail_unscaled}
		\P\big(|U_b-\mu_b|\ge\varepsilon\big)
		\;\le\;
		2\exp\!\left\{-\,c\,k\,\min\!\Big(\frac{\varepsilon^2}{K^2},
		\frac{\varepsilon}{K}\Big)\right\}.
	\end{equation}
\end{applem}

\begin{proof}
	\noindent{(a)}
	For each $q-1,\ldots,d$, the triangle inequality for Orlicz norms gives
	\[
	\big\|Z_i^{(q)}-Z_j^{(q)}\big\|_{\psi_2}
	\;\le\; \big\|Z_i^{(q)}\big\|_{\psi_2}+\big\|Z_j^{(q)}\big\|_{\psi_2}
	\;\le\; 2K.
	\]
	A standard consequence of $\|W\|_{\psi_2}\le 2K$ is the MGF bound \citep[see Chapter 2 from][]{wainwright2019high} for an absolute constant $c_1>0$ such that
	\[
	\E\exp\!\big(\lambda\{W-\E W\}\big)\ \le\ \exp\!\big(c_1\lambda^2K^2\big)
	\qquad\text{for all }\lambda\in\R.
	\]
	Taking $\Xvec_{ij}:=\Zvec_i-\Zvec_j$ 
	and using independence of coordinates, for any $u\in\R^d$, we get
	\begin{align*}
		& \E\exp\!\Big(\lambda\big\langle u,\Xvec_{ij}-\E \Xvec_{ij}\big\rangle\Big)
		\\ &=\prod_{q=1}^d \E\exp\!\Big(\lambda u_q
		\big\{(Z_i^{(q)}-Z_j^{(q)})-\E(Z_i^{(q)}-Z_j^{(q)})\big\}\Big)
		\le \exp\!\big(c_1\lambda^2K^2\|u\|_2^2\big).    
	\end{align*}
	Hence $\Xvec_{ij}$ is a sub-Gaussian random vector with variance proxy 
	$\sigma^2:=c_1K^2$, uniformly in $d$.
	So, a standard concentration inequality for Lipschitz 
	functions of sub-Gaussian vectors (the map $x\mapsto\|x\|_2$ is 
	$1$-Lipschitz) gives absolute constants $c_2,C>0$ such that for all 
	$x\ge 0$,
	\[
	\P\Big(\big|\|\Xvec_{ij}\|_2-\E\|Xvec_{ij}\|_2\big|\ge x\Big)
	\;\le\; 2\exp\!\left(-\frac{c_2\,x^2}{\sigma^2}\right)
	= 2\exp\!\left(-\frac{c_2\,x^2}{c_1 K^2}\right).
	\]
	Therefore, the bound 
	$\big\|\,\|\Zvec_i-\Zvec_j\|_2-\E\|\Zvec_i-\Zvec_j\|_2\,\big\|_{\psi_2}\le CK$
	follows immediately by taking $c_2/c_1=c$. Note that here the concentration bound is independent of $d$.
	
	\smallskip
	\noindent\textit{(b)} 
	Define the centered kernel $h_{ij}:=\|\Zvec_i-\Zvec_j\|_2-\E\|\Zvec_i-\Zvec_j\|_2$ that
	satisfies $\|h_{ij}\|_{\psi_2}\le CK$. 
	Since $x^2/K^2\ge x/K - 1/4$ 
	(folllows from $(x/K-1/2)^2\ge 0$), from \eqref{eq:sg_kernel_tail_unscaled} it follows that there exist constants $L_1,L_2>0$ depending 
	only on $c$ such that
	\[
	\P\big(|h_{ij}|\ge x\big)
	\;\le\; L_1\exp(-L_2\,x/K)
	\qquad\text{for all }x\ge 0.
	\]
	Now $U_b-\mu_b = \binom{b}{2}^{-1}\sum_{1\le i<j\le b}h_{ij}$ is a 
	mean-zero order-$2$ U-statistic. Applying 
	the U-statistic concentration inequality of Ning et al.\ (their 
	Lemma~A.3) with $k=\lfloor b/2\rfloor$ yields an absolute 
	constant $c_4>0$ such that
	\[
	\P\big(|U_b-\mu_b|\ge\varepsilon\big)
	\;\le\; 2\exp\!\left\{-c_4\,k\,\min\!\left(\frac{\varepsilon^2}{K^2},
	\frac{\varepsilon}{K}\right)\right\},
	\]
	where $L_1$ and $L_2$ have been absorbed into $c_4$. 
\end{proof}

\subsection[Appendix B: Proofs of the lemma and theorems]{Appendix B: Proofs of the lemma and theorems stated in the main article}

\begin{proof}[\textbf{Proof of Lemma~\ref{lem:1}.}]
	\textbf{ } Note that under the given condition, $\E\norm{\Xvec_1-\Xvec_2}$, $\E\norm{\Xvec_1-\Yvec_1}$ and $\E\norm{\Yvec_1-\Yvec_2}$ are also finite. The `if' part is fairly easy to establish. Clearly, $\F_{1}=\F_{2}$ implies $\E\norm{\Xvec_1-\Xvec_2}=\E\norm{\Xvec_1-\Yvec_1}=\E\norm{\Yvec_1-\Yvec_2}$,
	which leads to $\mathcal{D}\left(\F_{1},\F_{2}\right)=0$.
	
	For the `only if' part, note that $\mathcal{D}\left(\F_{1},\F_{2}\right)=0$ implies $\mathbb{E}\norm{\Xvec_1-\Xvec_2}=\mathbb{E}\norm{\Yvec_1-\Xvec_2}$
	and $\E\norm{\Xvec_1-\Yvec_2}=\E\norm{\Yvec_1-\Yvec_2}$, and hence, ${\cal E}(\F_1,\F_2)=2\E\norm{\Xvec_1-\Yvec_1}-\E\norm{\Xvec_1-\Xvec_2}-\E\norm{\Yvec_1-\Yvec_2}=0$. Now, by Lemma 2.2 of \cite{Baringhaus2004}, we get $\F_1=\F_2$.
\end{proof}

\begin{proof}[\textbf{Proof of Theorem~\ref{thm:large-sam-consistency}}]
	(a) 
	For $\tilde\theta\in[\delta_n,1-\delta_n]$, define 
	\[
	\varrho_n(\theta,\tilde\theta)
	:=\frac{\lfloor n\tilde{\theta}\rfloor\big(n-\lfloor n\tilde{\theta}\rfloor\big)}{n^{2}}\,
	\mathcal{D}_{n}\!\left(\mathcal{X}_{\lfloor n\tilde{\theta}\rfloor},
	\mathcal{Y}_{\lfloor n\tilde{\theta}\rfloor}\right)
	\]
	\[
	\text{and  } e_n(\omega):=\sup_{\tilde{\theta}\in[\delta_n,1-\delta_n]}
	\big|\varrho_n(\theta,\tilde{\theta})(\omega)-\varrho(\theta,\tilde{\theta})\big|.
	\]
	
	By Lemma A.~\ref{Uniform Convergence LST-2}, 
	there exists a measurable set $\Omega_0$ with
	$\P(\Omega_0)=1$ such that $e_n(\omega)\to 0$ for every $\omega\in\Omega_0$. 
	Since $\widehat\theta_n=\widehat T_n/n$ is a maximizer of
	$\tilde\theta\mapsto \varrho_n(\theta,\tilde\theta)$, we have 
	\begin{align}
		\varrho(\theta,\widehat\theta_n)(\omega)
		&\ge \varrho_n(\theta,\widehat\theta_n)(\omega)-e_n(\omega) \nonumber\\
		&\ge \varrho_n(\theta,\theta)(\omega)-e_n(\omega)
		\ge \varrho(\theta,\theta)-2e_n(\omega),
		\label{eq:Thm1_rho_lower_pathwise}
	\end{align}
	where 
	$e_n(\omega)\to 0$ as $n \rightarrow \infty$.
	Since $\tilde\theta\mapsto \varrho(\theta,\tilde\theta)$ is a continuous function defined on a compact support and it has a unique
	maximum at $\tilde\theta=\theta$ (see lemma A.\ref{Uniform Convergence LST-2}), 
	for every $\varepsilon>0$
	there exists $\Delta(\varepsilon)>0$ such that
	\begin{equation}
		\sup_{|\tilde{\theta}-\theta|\ge \varepsilon}\varrho(\theta,\tilde{\theta})
		\le \varrho(\theta,\theta)-\Delta(\varepsilon).
		\label{eq:Thm1_separation_simple}
	\end{equation}
	
	Since $e_n(\omega)\to 0$, for any $\varepsilon>0$. there exists 
	$N(\omega,\varepsilon) \in {\mathbb N}$ such that for all $n\ge N(\omega,\varepsilon)$, $e_n(\omega)<\Delta(\varepsilon)/4.$
	For such $n$, \eqref{eq:Thm1_rho_lower_pathwise} gives $\varrho(\theta,\widehat\theta_n)(\omega)\ge \varrho(\theta,\theta)-\Delta(\varepsilon)/2.
	$ If $|\widehat\theta_n(\omega)-\theta|\ge\varepsilon$ for some $n\ge N(\omega,\varepsilon)$, then
	\eqref{eq:Thm1_separation_simple} would imply $\varrho(\theta,\widehat\theta_n)(\omega)\le \varrho(\theta,\theta)-\Delta(\varepsilon),
	$ which is a contradiction. Therefore,
	\[
	|\widehat\theta_n(\omega)-\theta|<\varepsilon
	\qquad\text{for all } n\ge N(\omega,\varepsilon).
	\]
	Since $\varepsilon>0$ and $\omega\in\Omega_0$ are arbitrary, we conclude that
	$\widehat\theta_n\stackrel{a.s.}{\rightarrow}\theta$. 
	
	\noindent 
	(b) Recall that we reject $H_0$ if $\widehat{S}_{n}=\max_{t\in\mathcal{A}_{n}}w(t)\,\mathcal{D}_{n}(\mathcal{X}_{t},\mathcal{Y}_{t})$ is larger than the permutation cut-off
	\[
	\widehat c_{1-\alpha}
	=\inf\left\{ x\in\mathbb{R}^{+}:
	\frac{1}{n!}\sum_{\boldsymbol{\pi}\in\Pi_{n}}\mathds{1}\!\left[\widehat{S}_{n}^{(\boldsymbol{\pi})}\le x\right]\ge 1-\alpha\right\}.
	\]
	Under $H_0$, $(\Zvec_{\pi(1)},\dots,\Zvec_{\pi(n)})$ is exchangeable, so
	$\P_{H_0}(\widehat{S}_{n}>\widehat c_{1-\alpha})\le \alpha$.
	By Lemma~A.\ref{lem:perm_stat_vanish_BG}, we have
	$\widehat c_{1-\alpha}\xrightarrow{\P}0$. On the other hand, under $H_1$, by continuity of the maximum functional,
	\[
	\widehat{S}_{n}
	=\max_{\tilde{\theta}\in[\delta_n,1-\delta_n]}
	\varrho_n(\theta,\tilde{\theta})\xrightarrow{a.s.}
	\max_{\tilde{\theta}\in[\delta_n,1-\delta_n]}\varrho(\theta,\tilde{\theta})
	=\theta(1-\theta)\,D(\F_1,\F_2),
	\]
	which is positive when $\theta\in(0,1)$ and $\F_1\neq \F_2$. 
	Hence, the change-point is detected with probability tending to $1$ as $n$ tends to infinity.
\end{proof}

\begin{proof}[\textbf{Proof of Theorem~\ref{thm:HDLSS-consistency}.}]
	\label{proof_thm_2}
	It follows from Lemma A.\ref{HDLSS Limits} that
	\begin{align*}
		\lim_{d\to\infty}\frac{1}{d}D\left(\F_{1},\F_{2}\right)
		&=\left(\sqrt{\sigma_{1}^{2}+\sigma_{2}^{2}+\nu^{2}}-\sqrt{2}\sigma_{1}\right)^{2}+\left(\sqrt{\sigma_{1}^{2}+\sigma_{2}^{2}+\nu^{2}}-\sqrt{2}\sigma_{2}\right)^{2}\\
		&=\frac{1}{2}\left[\left(2\sqrt{\sigma_{1}^{2}+\sigma_{2}^{2}+\nu^{2}}-\sqrt{2}\sigma_{1}-\sqrt{2}\sigma_{2}\right)^{2}+\left(\sqrt{2}\sigma_{2}-\sqrt{2}\sigma_{1}\right)^{2}\right].
	\end{align*}
	This is non-negative, and equals $0$ if and only if $\sigma_1=\sigma_2$ and
	$2\sqrt{2\sigma^{2}+\nu^{2}}-2\sqrt{2}\sigma=0$
	(taking $\sigma_1=\sigma_2=\sigma$), which is equivalent to $\nu^2=0$.
	Hence $\lim_{d\to\infty}\frac{1}{d}\mathcal{D}\left(\F_{1},\F_{2}\right)>0$
	if and only if $\nu^2+(\sigma_1-\sigma_2)^2 > 0$.
	
	\noindent
	(a) Denote the limiting pairwise distances (see Lemma A.\ref{HDLSS Limits}) $\sqrt{2}\sigma_{1},\sqrt{2}\sigma_{2}$
	and $\sqrt{\nu^{2}+\sigma_{1}^{2}+\sigma_{2}^{2}}$ as $\delta_{11},\delta_{22}$
	and $\delta_{12}$ respectively. One can show that
	\begin{align*}
		& \frac{1}{\sqrt{d}}T_{12}\left(t\right)\overset{P}{\longrightarrow}\begin{cases}
			\frac{1}{\left(n-t\right)}\left(\left(\tau-t\right)\delta_{11}+\left(n-\tau\right)\delta_{12}\right) & t\leq\tau\\
			\frac{1}{t}\left(\tau\delta_{12}+\left(t-\tau\right)\delta_{22}\right) & t>\tau
		\end{cases}\label{T_12}\\
		& \frac{1}{\sqrt{d}}T_{11}\left(t\right)\overset{P}{\longrightarrow}\begin{cases}
			\delta_{11} & t\leq\tau\\
			{t \choose 2}^{-1}\left({\tau \choose 2}\delta_{11}+{t-\tau \choose 2}\delta_{22}+\tau\left(t-\tau\right)\delta_{12}\right) & t>\tau
		\end{cases}
	\end{align*}
	\begin{align*}
		& \frac{1}{\sqrt{d}}T_{22}\left(t\right)\overset{P}{\longrightarrow}\begin{cases}
			{n-t \choose 2}^{-1}\bigg({\tau-t \choose 2}\delta_{11}+{n-\tau \choose 2}\delta_{22}+\left(n-\tau\right)\left(\tau-t\right)\delta_{12}\bigg) & t\leq\tau\\
			\delta_{22} & t>\tau
		\end{cases}
	\end{align*}
	Using $\omega_{1}:=\delta_{12}-\delta_{11}$
	, $\omega_{2}:=\delta_{12}-\delta_{22}$ and also defining $\alpha_{1}=\frac{\tau}{t}$, $\beta_{1}=\frac{n-\tau}{n-t}$, $\alpha_{2}=\frac{\tau-1}{t-1}$, $\beta_{2}=\frac{n-\tau-1}{n-t-1}$, $\alpha_{2}^{\prime}=1-\alpha_{2}$,
	and $\beta_{2}^{\prime}=1-\beta_{2}$, we get 
	\begin{align*}
		& \frac{1}{\sqrt{d}}\left\{ T_{12}\left(t\right)-T_{11}\left(t\right)\right\} \overset{P}{\longrightarrow}\beta_{1}\omega_{1}\mathbb{I}_{t\leq\tau}+\alpha_{1}\left(-\alpha_{2}^{\prime}\omega_{2}+\alpha_{2}\omega_{1}\right)\mathbb{I}_{t>\tau}\\
		& \frac{1}{\sqrt{d}}\left\{ T_{12}\left(t\right)-T_{22}\left(t\right)\right\} \overset{P}{\longrightarrow}\beta_{1}\left(-\beta_{2}^{\prime}\omega_{1}+\beta_{2}\omega_{2}\right)\mathbb{I}_{t\leq\tau}+\alpha_{1}\omega_{2}\mathbb{I}_{t>\tau}
	\end{align*} Combining these two terms, we obtain $\frac{t\left(n-t\right)}{dn^{2}}\mathcal{D}_{n}\left(\mathcal{X}_{t},\mathcal{Y}_{t}\right) \stackrel{P}{\rightarrow}f(t,\tau)$, where
	\begin{align*}
		& f(t,\tau)=\begin{cases}
			\frac{t\left(n-t\right)}{n^{2}}\beta_{1}^{2}\left(\omega_{1}^{2}+\left(\beta_{2}\omega_{2}-\beta_{2}^{\prime}\omega_{1}\right)^{2}\right) & t<\tau\\
			\frac{\tau\left(n-\tau\right)}{n^{2}}\left(\omega_{1}^{2}+\omega_{2}^{2}\right) & t=\tau\\
			\frac{t\left(n-t\right)}{n^{2}}\alpha_{1}^{2}\left(\left(\alpha_{2}\omega_{1}-\alpha_{2}^{\prime}\omega_{2}\right)^{2}+\omega_{2}^{2}\right) & t>\tau
		\end{cases}.
	\end{align*}
	Now we show that $f\left(t;\tau\right)$
	is maximized at $t=\tau$. For $t<\tau$, note that all $\beta_{i}$'s lie
	in the interval $\left(0,1\right)$. Hence,
	\begin{align*}
		f\left(\tau;\tau\right)-f\left(t;\tau\right)
		& = \frac{\tau\left(n-\tau\right)}{n^{2}}\left(\omega_{1}^{2}+\omega_{2}^{2}\right)-\frac{t\left(n-t\right)}{n^{2}}\beta_{1}^{2}\left(\omega_{1}^{2}+\left(\beta_{2}\omega_{2}-\beta_{2}^{\prime}\omega_{1}\right)^{2}\right)\\
		& >\frac{\tau\left(n-\tau\right)}{n^{2}}\left(\omega_{1}^{2}+\omega_{2}^{2}\right)-\frac{t\left(n-t\right)}{n^{2}}\beta_{1}^{2}\left(\omega_{1}^{2}\left(2-\beta_{2}\right)+\beta_{2}\omega_{2}^{2}\right)\\
		& \text{(applying Jensen's inequality)}\\
		& =\frac{\left(n-\tau\right)t}{n^{2}}\bigg[\omega_{1}^{2}\left\{ \frac{\tau}{t}-\frac{n-\tau}{n-t}\left(2-\frac{n-\tau-1}{n-t-1}\right)\right\} \\
		& \qquad \qquad \qquad + \omega_{2}^{2}\left\{ \frac{\tau}{t}-\frac{n-\tau}{n-t}\left(\frac{n-\tau-1}{n-t-1}\right)^{2}\right\} \bigg]>0.
	\end{align*}
	Similarly, we can show that $f\left(\tau;\tau\right)-f\left(t;\tau\right)>0$ for $t>\tau$. Hence,
	\begin{align*}
		& \hat{T}_{n}=\hspace{-0.1in}\underset{t\in\left\{ \floor{n\delta_{n}},\dots,\floor{n\left(1-\delta_{n}\right)}\right\} } {\text{argmax}}\frac{t\left(n-t\right)}{dn^{2}}\mathcal{D}_{n}\left(\mathcal{X}_{t},\mathcal{Y}_{t}\right)\overset{\P}{\longrightarrow}
		\hspace{-0.1in} \underset{t\in\left\{ \floor{n\delta_{n}},\dots,\floor{n\left(1-\delta_{n}\right)}\right\} }{\text{argmax}}f\left(t;\tau\right)=\tau
	\end{align*}

	\noindent (b)
	Let $\widehat S_n:=\max_{t\in\mathcal A_n} w(t)\,\mathcal D_n(\mathcal X_t,\mathcal Y_t)$ be the test statistic,
	where $w(t)=t(n-t)\big/n^2$ and $\mathcal{A}_n = \{\lfloor n\delta_n\rfloor, \ldots, \lceil n(1-\delta_n)\rceil\}$.
	For a permutation $\pi$ of $\{1,\ldots,n\}$, let $\widehat{S}_n^{(\pi)}$ be the permutation analog of $\widehat{S}_n$. For any fixed $t\in\{2,\ldots,n-2\}$ and the permutation $\pi$, let $k$ be the number of $\F_1$ observations among
	$\mathcal{X}_t^\pi= \{Z_{\pi(1)},\ldots,Z_{\pi(t)}\}$. Define
	\[
	w_{11}^{X}:=\binom{a}{2}\Big/\binom{t}{2},\quad
	w_{22}^{X}:=\binom{b}{2}\Big/\binom{t}{2},\quad
	w_{12}^{X}:=ab\Big/\binom{t}{2},
	\]
	\[
	w_{11}^{Y}:=\binom{c}{2}\Big/\binom{n-t}{2},\quad
	w_{22}^{Y}:=\binom{d}{2}\Big/\binom{n-t}{2},\quad
	w_{12}^{Y}:=cd\Big/\binom{n-t}{2},
	\]
	\[
	w_{11}^{XY}:=ac\Big/\{t(n-t)\},\quad
	w_{22}^{XY}:=bd\Big/\{t(n-t)\},\quad
	w_{12}^{XY}:=1-w_{11}^{XY}-w_{22}^{XY},
	\]
	where $a:=k, b:=t-k, c:=\tau-k \text{ and } d:=n-t-\tau+k.$
	
	Note that under $H_1$ in the HDLSS regime, the scaled pairwise distances converge to constants
	$\delta_{11},\delta_{22},\delta_{12}$ depending on whether the pair is from $(\F_1,\F_1)$, $(\F_2,\F_2)$,
	or $(\F_1,\F_2)$. So as $d\to\infty$,
	\begin{equation}
		\label{eq:HDLSS_Tij_exactweights_b_new_short}
		\begin{aligned}
			\frac{1}{\sqrt d}T_{11}^{(\pi)}(t)&\longrightarrow
			w_{11}^{X}\delta_{11}+w_{22}^{X}\delta_{22}+w_{12}^{X}\delta_{12}
			=\delta_{12}-\omega_1 w_{11}^{X}-\omega_2 w_{22}^{X},\\
			\frac{1}{\sqrt d}T_{22}^{(\pi)}(t)&\longrightarrow
			w_{11}^{Y}\delta_{11}+w_{22}^{Y}\delta_{22}+w_{12}^{Y}\delta_{12}
			=\delta_{12}-\omega_1 w_{11}^{Y}-\omega_2 w_{22}^{Y},\\
			\frac{1}{\sqrt d}T_{12}^{(\pi)}(t)&\longrightarrow
			w_{11}^{XY}\delta_{11}+w_{22}^{XY}\delta_{22}+w_{12}^{XY}\delta_{12}
			=\delta_{12}-\omega_1 w_{11}^{XY}-\omega_2 w_{22}^{XY}.
		\end{aligned}
	\end{equation}
	Hence we have
	\begin{equation}\label{eq:Delta_ex}
		\begin{aligned}
			\Delta_{1,\mathrm{ex}}(t,k)
			&:= \lim_{d\to\infty}\frac{1}{\sqrt{d}}
			\bigl(T_{12}^{(\pi)}(t)-T_{11}^{(\pi)}(t)\bigr)
			= \omega_1\bigl(w_{11}^{X}-w_{11}^{XY}\bigr)
			+\omega_2\bigl(w_{22}^{X}-w_{22}^{XY}\bigr),\\[6pt]
			\Delta_{2,\mathrm{ex}}(t,k)
			&:= \lim_{d\to\infty}\frac{1}{\sqrt{d}}
			\bigl(T_{12}^{(\pi)}(t)-T_{22}^{(\pi)}(t)\bigr)
			= \omega_1\bigl(w_{11}^{Y}-w_{11}^{XY}\bigr)
			+\omega_2\bigl(w_{22}^{Y}-w_{22}^{XY}\bigr).
		\end{aligned}
	\end{equation}
	Define $p:={k}/{t} \text{ and } q:={(\tau-k)}/{(n-t)},
	$.
	Now, using $a=tp$, $b=t(1-p)$, $\binom{a}{2}\Big/\binom{t}{2}=a(a-1)\Big/\{t(t-1)\}$, and similarly for $(c,d)$, we can re-write the weights as
	\[
	w_{11}^{X}=p^2-\frac{p(1-p)}{t-1},\,
	w_{22}^{X}=(1-p)^2-\frac{p(1-p)}{t-1},\,
	w_{11}^{XY}=pq,\,
	w_{22}^{XY}=(1-p)(1-q),
	\]
	\[
	w_{11}^{Y}=q^2-\frac{q(1-q)}{n-t-1},\,
	w_{22}^{Y}=(1-q)^2-\frac{q(1-q)}{n-t-1}.
	\]
	These imply
	\begin{align}
		\label{eq:Delta_ex_decomp_final}
		\Delta_{1,\mathrm{ex}}(t,k)=(p-q)(\kappa p-\omega_2)-\frac{\kappa p(1-p)}{t-1},
		\\
		\Delta_{2,\mathrm{ex}}(t,k)=(p-q)(\omega_2-\kappa q)-\frac{\kappa q(1-q)}{n-t-1},
	\end{align}
	for $\kappa=\omega_1+\omega_2$. Now, as $d\to\infty$, $\frac{1}{d}\,\widehat{S}_n^{(\pi)}\ \longrightarrow\ \max_{t\in\mathcal A_n} J_{\mathrm{ex}}(t,k),
	$ where
	\begin{equation}
		\label{thm2:J-ex-form}
		J_{\mathrm{ex}}(t,k):=w(t)\Big(\Delta_{1,\mathrm{ex}}(t,k)^2+\Delta_{2,\mathrm{ex}}(t,k)^2\Big),
	\end{equation}
	Recall that, the limiting value (as $d$ tends to infinity) of $\frac1d\hat{S}_n$ is
	\begin{equation}
		\label{eq:Jpure_exact_final_new}
		J_{\mathrm{pure}}:=f(\tau;\tau)
		=w(\tau)\,(\omega_1^2+\omega_2^2).\end{equation}
	Now we will show that if $\pi\notin\mathcal{B}(\rho_{n,\tau})$ then $J_{\text{ex}}(t,k)<J_\text{pure}$. In fact we will find an upper bound of $J_{ex}(t,k)$ and show that this bound is smaller than $J_\text{pure}$ if $\pi\notin\mathcal{B}(\rho_{n,\tau})$. Clearly,
	$p_n\leq\Pr(\Pi\in\mathcal{B}(\rho_{n,\tau}))\le\alpha$, where $p_n$ is the p-value of the permutation test. This implies that the change-point is detected with probability tending to $1$ as $d$ grows to infinity.
	
	To find a bound for $J_{\mathrm{ex}}(t,k)$, note that \eqref{eq:Delta_ex_decomp_final} can be re-written as $\Delta_{1,\mathrm{ex}}=A_1-B_1,\,\, \Delta_{2,\mathrm{ex}}=A_2-B_2,$ where
	$A_1:=(p-q)(\kappa p-\omega_2),\, B_1:={\kappa p(1-p)}/({t-1}),\,
	A_2:=(p-q)(\omega_2-\kappa q),\, B_2:={\kappa q(1-q)}/({n-t-1}).
	$
	Using $(x-y)^2\le x^2+y^2+2|xy|$ and summing the two squares, we get
	\[
	\Delta_{1,\mathrm{ex}}^2+\Delta_{2,\mathrm{ex}}^2
	\le (A_1^2+A_2^2)+(B_1^2+B_2^2)+2\big(|A_1||B_1|+|A_2||B_2|\big).
	\]
	Note that 
	$|\kappa p-\omega_2| \le \sup_{u\in[0,1]}|\kappa u-\omega_2|=\max\{|\omega_1|,|\omega_2|\}:=M_*,$ say. Similarly, we have $|\omega_2-\kappa q|
	=\sup_{u\in[0,1]}|\omega_2-\kappa u|
	=M_*$
	Therefore,
	\[
	A_1^2+A_2^2
	=(p-q)^2\Big\{(\kappa p-\omega_2)^2+(\omega_2-\kappa q)^2\Big\}
	\le 2|p-q|^2M_*^2,
	\]
	Since $0\le p(1-p),q(1-q)\le 1/4$, we also have
	\begin{align*}
		B_1^2\le \frac{\kappa^2}{16(t-1)^2},
		& \qquad B_2^2\le \frac{\kappa^2}{16(n-t-1)^2}\\
		2|A_1||B_1|\le \frac{|p-q|\kappa M_*}{2(t-1)}, &\qquad
		2|A_2||B_2|\le \frac{|p-q|\kappa M_*}{2(n-t-1)}.    
	\end{align*}
	So, from \eqref{thm2:J-ex-form}, for $t\ge 2$ and $n-t\ge 2$, we get
	\begin{align}
		\label{eq:Jex_pointwise_bound_final_new}
		& J_{\mathrm{ex}}(t,k) \le w(t)\Bigg[
		2|p-q|^2M_*^2
		+\frac{|p-q|\kappa M_*}{2}\Big\{\frac{1}{t-1}+\frac{1}{n-t-1}\Big\} \nonumber
		\\
		& 
		\qquad \qquad \qquad \qquad \qquad \qquad \qquad+ \frac{\kappa^2}{16}\Big\{\frac{1}{(t-1)^2}+\frac{1}{(n-t-1)^2}\Big\}
		\Bigg].
	\end{align}
	
	Now let $m:=\lfloor n\delta_n\rfloor$ and assume $m\ge 2$. For $t\in\mathcal A_n=\{m,\ldots,n-m\}$,
	the functions $t\mapsto t(n-t)/(t-1)$ and $t\mapsto t(n-t)/(t-1)^2$ are decreasing, Hence
	\[
	\max_{t\in\mathcal A_n}\frac{w(t)}{t-1}=\frac{w(m)}{m-1},\qquad
	\max_{t\in\mathcal A_n}\frac{w(t)}{(t-1)^2}=\frac{w(m)}{(m-1)^2},
	\]
	and the same identities hold with $t$ replaced by $n-t$. Also using $w(t)\le 1/4$, from \eqref{eq:Jex_pointwise_bound_final_new}
	we get the following uniform bound for all $t\in\mathcal A_n$ and $k=pt$.
	\begin{equation}
		\label{eq:Jex_bound_uniform_final_new}
		J_{\mathrm{ex}}(t,k)
		\le
		\frac12 M_*^2|p-q|^2
		+
		\Big(\frac{\kappa M_*\,w(m)}{m-1}\Big)|p-q|
		+
		\frac{\kappa^2}{8}\frac{w(m)}{(m-1)^2}.
	\end{equation}
	
	Now, for any fix $\rho\in(0,1)$, call a configuration $(t,k)$ \emph{good} if $|p-q|\le \rho$. So, for  all good configurations with $t\in\mathcal A_n$, we have 
	\begin{align*}
		J_{\mathrm{ex}}(t,k)\le \frac12 M_*^2\rho^2 + B_n\rho + C_n,\\ 
		\text{where }
		B_n:=\frac{\kappa M_*\,w(m)}{m-1}
		\text{ and }
		C_n:=\frac{\kappa^2}{8}\frac{w(m)}{(m-1)^2}.
	\end{align*}
	So, if $\frac12 M_*^2\rho^2 + B_n\rho + C_n < J_{\mathrm{pure}}$, then we have $J_{\mathrm{ex}}(t,k) < J_{\mathrm{pure}}$
	holds all $t\in\mathcal A_n$. Clearly, this is satisfied if $\rho<\rho_{0}$, where 
	\begin{equation}
		\label{eq:rho_threshold_def_final_new}
		\rho_{0}
		=
		\left(
		\frac{-B_n+\sqrt{B_n^2+2M_*^2\big(J_{\mathrm{pure}}-C_n\big)}}{M_*^2}
		\right)\wedge 1.
	\end{equation}
	However, $\rho_0$ depends on $\omega_1=\delta_{12}-\delta_{11}$ and $\omega_2=\delta_{12}-\delta_{22}$. 
	So, now we look for a universal upper bound, which is free from all population quantities. 
	Note that $\kappa\le 2M_*$ and $M_*^2\le \omega_1^2+\omega_2^2\le 2M_*^2$. Therefore,
	\[
	\frac{B_n}{M_*^2}=\frac{\kappa}{M_*}\frac{w(m)}{m-1}
	\le \frac{2w(m)}{m-1}, \quad 
	\frac{C_n}{M_*^2}=\Big(\frac{\kappa}{M_*}\Big)^2\frac{w(m)}{8(m-1)^2}
	\le \frac{w(m)}{2(m-1)^2}
	\]
	and ${J_{\mathrm{pure}}}/{M_*^2} = w(\tau)({\omega_1^2+\omega_2^2})/{M_*^2} \ge w(\tau)$.
	Using the fact that $x\mapsto -a+\sqrt{a^2+2(x-b)}$ is increasing in $x$ 
	and decreasing in both $a$ and $b$, we get
	$\rho_0 >\rho_{n,\tau}$, where 
	$\rho_{n,\tau}
	:=
	\Bigg(
	-2\,\frac{w(m)}{m-1}
	+\sqrt{
		4\Big(\frac{w(m)}{m-1}\Big)^2
		+2\Big\{
		w(\tau)-\frac{w(m)}{2(m-1)^2}
		\Big\}}
	\Bigg)_+\ \wedge\ 1$.
	So, the condition $\rho <\rho_{n,\tau}$ (as stated in the theorem) ensures $\rho<\rho_0$ and hence if $\pi \notin \mathcal{B}(\rho_{n,\tau})$, we have $J_\mathrm{ex}(t,k)<J_\mathrm{pure}$.
	\end{proof}
	
	\begin{proof}[\textbf{Proof of Lemma~\ref{lem:sarkar}.}]
The proof follows directly from Lemma~1 of \cite{sarkar2018some}.
In that paper, the authors consider kernels of the form
$\psi(|x-y|^{2})$, whereas we work with $\psi(|x-y|)$.
Define $\tilde{\psi}(u)=\psi(\sqrt{u})$ for $u\ge0$.
Then
\[
\varphi_{h,\psi}(\xvec,\yvec)
= h\!\left\{\frac{1}{d}\sum_{k=1}^{d}\tilde{\psi}\!\big(|x^{(k)}-y^{(k)}|^{2}\big)\right\},
\]
so that our kernel is covered by the framework of \cite{sarkar2018some}.
Moreover, for $u>0$,
\[
\tilde{\psi}'(u)=\frac{\psi'(\sqrt{u})}{2\sqrt{u}},
\]
and hence the assumption that $t\mapsto \psi'(t)/t$ is non-constant completely
monotone implies that $\tilde{\psi}'$ is non-constant completely monotone.
Therefore, all conditions of Lemma~1 in \cite{sarkar2018some} are satisfied.
Consequently,
$\mathcal{E}^{(d)}_{h,\psi}(\F_1,\F_2)\ge0$, with equality if and only if
$\F_1$ and $\F_2$ have identical one-dimensional marginal distributions.
\end{proof}

\begin{proof}[\textbf{Proof of Theorem~\ref{thm:HDLSS-gen-consistency}.}]
Define $\Wvec=\Zvec_i-\Zvec_j$, where $\Zvec_i\sim\F_r$ and $\Zvec_j\sim\F_s$ are independent. By Assumption~(A4), $\mathrm{Var}(d^{-1}\sum_{q=1}^d \psi(|W^{(q)}|))\to 0$ as $d\to\infty$ and hence WLLN holds for the sequence $\{\psi(|W^{(q)}|):q\ge 1\}$. Since $h$ is continuous, this implies
\[
\varphi_{h,\psi}(\Zvec,\Zvec^*) - \phi^*_{h,\psi}(\F_r,\F_s)\overset{P}{\longrightarrow} 
0\qquad \text{as } d\to\infty.
\]
Let $\delta_{rs}^{h,\psi}$ be the limiting value of $\phi^*_{h,\psi}(\F_r,\F_s)$ as $d\to\infty$. So under the given condition, we have $\varphi_{h,\psi}(\Zvec_i,\Zvec_j)\overset{P}{\to}\delta^{h,\psi}_{rs}$.
Therefore, for any $t$ we have
\begin{align*}
	& \frac{1}{\sqrt{d}}T^{h,\psi}_{12}\left(t\right)\overset{P}{\longrightarrow}\begin{cases}
		\frac{1}{\left(n-t\right)}\left(\left(\tau-t\right)\delta^{h,\psi}_{11}+\left(n-\tau\right)\delta^{h,\psi}_{12}\right) & t\leq\tau\\
		\frac{1}{t}\left(\tau\delta^{h,\psi}_{12}+\left(t-\tau\right)\delta^{h,\psi}_{22}\right) & t>\tau
	\end{cases}\\
	& \frac{1}{\sqrt{d}}T^{h,\psi}_{11}\left(t\right)\overset{P}{\longrightarrow}\begin{cases}
		\delta^{h,\psi}_{11} & t\leq\tau\\
		{t \choose 2}^{-1}\left({\tau \choose 2}\delta^{h,\psi}_{11}+{t-\tau \choose 2}\delta^{h,\psi}_{22}+\tau\left(t-\tau\right)\delta^{h,\psi}_{12}\right) & t>\tau
	\end{cases}\\
	& \frac{1}{\sqrt{d}}T^{h,\psi}_{22}\left(t\right)\overset{P}{\longrightarrow}\begin{cases}
		{n-t \choose 2}^{-1}\bigg({\tau-t \choose 2}\delta^{h,\psi}_{11}+{n-\tau \choose 2}\delta^{h,\psi}_{22}+\left(n-\tau\right)\left(\tau-t\right)\delta^{h,\psi}_{12}\bigg) & t\leq\tau\\
		\delta^{h,\psi}_{22} & t>\tau
	\end{cases}
\end{align*}
Since $\lim_{d\to\infty}\mathcal{E}^{(d)}_{h,\psi}(\F_1,\F_2) = 
2\delta^{h,\psi}_{12}-\delta^{h,\psi}_{11}-\delta^{h,\psi}_{22}>0$, the proofs of part~(a) and part~(b) follows from the same argument as in the  proofs of part~(a) and part~(b) of Theorem~\ref{thm:HDLSS-consistency} with 
$(\delta^{h,\psi}_{11},\delta^{h,\psi}_{12},\delta^{h,\psi}_{22})$ 
replaced by $(\delta_{11},\delta_{12},\delta_{22})$.
\end{proof}

\begin{proof}[\textbf{Proof of Theorem~\ref{thm:sparse1}.}]
\label{proof_thm_4}
Define $Z_d(i,j):=d^{-1}\|\Zvec_i-\Zvec_j\|^2$ for each pair
$(i,j)$. Under (A1)--(A2),
$\Var(Z_d(i,j))=O(v_d^2/d^2)$, so
$Z_d(i,j)=\E Z_d(i,j)+O_P(v_d/d)$.
The expectations take the form
$\E Z_d(i,j)=2\,\tr(\Sigma_{1,d})/d$ if $i,j\le\tau$,
$\E Z_d(i,j)=2\,\tr(\Sigma_{2,d})/d$ if $i,j>\tau$, and
$\E Z_d(i,j)=(\tr(\Sigma_{1,d})+\tr(\Sigma_{2,d})+M_d)/d$
otherwise, where $M_d:=\|\mu_{1,d}-\mu_{2,d}\|^2$. Define
\begin{equation}
	\label{delta_ij:def}
	\delta_{11}(d):=\sqrt{\tfrac{2\,\tr(\Sigma_{1,d})}{d}},\quad
	\delta_{22}(d):=\sqrt{\tfrac{2\,\tr(\Sigma_{2,d})}{d}},\quad
	\delta_{12}(d):=\sqrt{\tfrac{\tr(\Sigma_{1,d})+\tr(\Sigma_{2,d})+M_d}{d}},
\end{equation}
so that $\delta_{rs}^2(d)=\E Z_d(i,j)$ for $\Zvec_i\sim\F_r$,
$\Zvec_j\sim\F_s$ independent.
Using the identity
$\sqrt{a}-\sqrt{b}=(a-b)/(\sqrt{a}+\sqrt{b})$, we get
\[
\frac{\|\Zvec_i-\Zvec_j\|}{\sqrt{d}}-\delta_{rs}(d)
\;=\;\frac{Z_d(i,j)-\delta_{rs}^2(d)}
{\|\Zvec_i-\Zvec_j\|/\sqrt{d}+\delta_{rs}(d)}.
\]
The numerator is $O_P(v_d/d)$. Since $M_d\ge 0$ implies
$\delta_{12}(d)\ge\min\{\delta_{11}(d),\delta_{22}(d)\}$, all three
$\delta_{rs}(d)$'s are at least
$\sqrt{2\min\{\tr(\Sigma_{1,d}),\tr(\Sigma_{2,d})\}/d}$.
Under the assumption
$$\liminf_{d\to\infty}\min\{\tr(\Sigma_{1,d}),
\tr(\Sigma_{2,d})\}/d>0$$, for all large $d$ one has
$\min_r\tr(\Sigma_{r,d})\ge c\,d$ for some $c>0$, and hence
the denominator satisfies
$\|\Zvec_i-\Zvec_j\|/\sqrt{d}+\delta_{rs}(d)\ge\min\{\delta_{11}(d),\delta_{12}(d),
\delta_{22}(d)\}\ge\sqrt{2c}>0$ uniformly in $d$.
Hence, uniformly over all $\binom{n}{2}$ pairs and all
$t\in\mathcal{A}_n$,
\begin{equation}\label{eq:pw_approx}
	\frac{\|\Zvec_i-\Zvec_j\|}{\sqrt{d}}
	\;=\;\delta_{rs}(d)+O_P\!\left(\frac{v_d}{d}\right).
\end{equation}

Next, define
$\omega_1(d):=\delta_{12}(d)-\delta_{11}(d)$,\;
$\omega_2(d):=\delta_{12}(d)-\delta_{22}(d)$,
$\Delta_d:=\tr(\Sigma_{2,d})-\tr(\Sigma_{1,d})$,\;
$r_d:=\sqrt{\omega_1^2(d)+\omega_2^2(d)}$,\;
and $s_d:=d\,r_d^2$.
Using the identity $a-b=(a^2-b^2)/(a+b)$, from~\eqref{delta_ij:def},
we get
\[
\omega_1(d)=\frac{M_d+\Delta_d}
{d\,(\delta_{12}(d)+\delta_{11}(d))},\qquad
\omega_2(d)=\frac{M_d-\Delta_d}
{d\,(\delta_{12}(d)+\delta_{22}(d))}.
\]
Now, note that under~(A1), the coordinate-wise fourth moments of
$\Zvec_i-\Zvec_j$ are uniformly bounded. So there exists $C>0$
such that $\tr(\Sigma_{1,d})\le Cd$, $\tr(\Sigma_{2,d})\le Cd$,
and $M_d\le Cd$ for all $d\ge 1$. Consequently,
from~\eqref{delta_ij:def}, $\max\{\delta_{11}(d),\delta_{12}(d),
\delta_{22}(d)\}\le C_0$ for some $C_0>0$ and all $d\ge 1$.

Since $|M_d+\Delta_d|+|M_d-\Delta_d|\ge 2\max\{M_d,|\Delta_d|\}$
and $\delta_{12}(d)+\delta_{rr}(d)\le 2C_0$ for $r=1,2$, we have
\[
r_d\;\ge\;\frac{|\omega_1(d)|+|\omega_2(d)|}{\sqrt{2}}
\;\ge\;
\frac{\max\{M_d,|\Delta_d|\}}
{\sqrt{2}\,C_0\,d},
\]
and therefore
\begin{equation}\label{eq:snr}
	\frac{v_d/d}{r_d}
	\;\le\;
	\frac{\sqrt{2}\,C_0\,v_d}
	{\max\{M_d,|\Delta_d|\}}.
\end{equation}
Since $M_d/v_d\to\infty$ or $|\Delta_d|/v_d\to\infty$ by hypothesis,
the right side of~\eqref{eq:snr} tends to zero.
Hence $v_d/d=o(r_d)$.

Using the same argument as in
proof of Theorem ~\ref{thm:HDLSS-consistency} and ~\eqref{eq:pw_approx}, we get uniformly over $t\in\mathcal{A}_n$,
\begin{align}
	\label{T_ij-contrasts}
	\nonumber
	\frac{T_{12}(t)-T_{11}(t)}{\sqrt{d}}
	&=\beta_1\omega_1(d)\,\mathbb{I}_{\{t\le\tau\}}
	+\alpha_1(\alpha_2\omega_1(d)-\alpha_2'\omega_2(d))\,
	\mathbb{I}_{\{t>\tau\}}
	+O_P\!\left(\frac{v_d}{d}\right),\\[4pt]
	\frac{T_{12}(t)-T_{22}(t)}{\sqrt{d}}
	&=\beta_1(\beta_2\omega_2(d)-\beta_2'\omega_1(d))\,
	\mathbb{I}_{\{t\le\tau\}}
	+\alpha_1\omega_2(d)\,\mathbb{I}_{\{t>\tau\}}
	+O_P\!\left(\frac{v_d}{d}\right).
\end{align}
Note that $\omega_1(d)/r_d$ and $\omega_2(d)/r_d$ are bounded ($\omega_1^2(d)+\omega_2^2(d)=r^2_d ~\forall d\ge1$). Dividing both sides \ref{T_ij-contrasts} by $r_d$ and 
taking their sum of squares, we have 
\[
\sup_{t\in\mathcal{A}_n}\left|
\frac{w(t)\,\mathcal{D}_n(t)}{s_d}
-w(t)\,G_d(t;\tau)\right|=o_P(1),
\]
where $s_d=d\,r_d^2$ and
\[
G_d(t;\tau):=
\begin{cases}
	\displaystyle
	\beta_1^2\,
	\frac{\omega_1^2(d)+(\beta_2\omega_2(d)-\beta_2'\omega_1(d))^2}
	{\omega_1^2(d)+\omega_2^2(d)},
	& t\le\tau,\\[12pt]
	\displaystyle
	\alpha_1^2\,
	\frac{(\alpha_2\omega_1(d)-\alpha_2'\omega_2(d))^2+\omega_2^2(d)}
	{\omega_1^2(d)+\omega_2^2(d)},
	& t>\tau,
\end{cases}
\]
with the same weights $\alpha_1,\alpha_2,\beta_1,\beta_2$
as in proof of Theorem ~\ref{thm:HDLSS-consistency}. 

At $t=\tau$: $\beta_1=1$, $\beta_2=1$, $\beta_2'=0$, giving
$G_d(\tau;\tau)=1$. For $t\neq\tau$, one can show that $w(t)\,G_d(t;\tau)<w(\tau)$ for every
$(\omega_1(d),\omega_2(d))/r_d\in \mathcal{S}^1$, where $\mathcal{S}^1$ denotes the unit circle in $\mathbb{R}^2$ with the  center at the origin (follows using similar arguments as used in proving $f(t;\tau)<f(\tau;\tau)$ for all $t\neq\tau$ in the proof of Theorem \ref{thm:HDLSS-consistency}). Now, since $\mathcal{A}_n$ is finite,
it directly implies $P(\widehat{T}_n=\tau)\to 1$ as $d\to\infty$.

Part~(b) follows from similar arguments as used in proving the consistency of the permutation test in the proof of 
Theorem~\ref{thm:HDLSS-consistency}. 
\end{proof}

\begin{proof}[\textbf{Proof of Theorem~\ref{thm:sparse2}}]
Since $h(t)=t$ for both Proposed-$\ell_1$ ($\psi(t)=t$) and
Proposed-exp ($\psi(t)=1-e^{-t}$), the generalised distance simplifies to
$\varphi_{h,\psi}(\Zvec_i,\Zvec_j)
= d^{-1}\sum_{q=1}^{d}\psi(|Z_i^{(q)}-Z_j^{(q)}|)$.
For $\Zvec_i\sim\mathrm{F}_r$ and $\Zvec_j\sim\mathrm{F}_s$
($r,s\in\{1,2\}$) independent, define
\begin{equation}\label{eq:ars}
	a_{rs}(d) \;:=\;
	\E\!\left[\frac{1}{d}\sum_{q=1}^{d}
	\psi\!\bigl(|Z_i^{(q)}-Z_j^{(q)}|\bigr)\right],
\end{equation}
and set $ \omega_1^\psi(d) := a_{12}(d)-a_{11}(d),\,
\omega_2^\psi(d) := a_{12}(d)-a_{22}(d),$ $
s_d^\psi        := \bigl(\omega_1^\psi(d)\bigr)^2
+\bigl(\omega_2^\psi(d)\bigr)^2,\,
r_d^\psi        := \sqrt{s_d^\psi}. $
This implies $\mathcal{E}_\psi(\mathrm{F}_1,\mathrm{F}_2)
= d\,(\omega_1^\psi(d)+\omega_2^\psi(d))$. So.
$\mathcal{E}_\psi(\mathrm{F}_1,\mathrm{F}_2)/v_{\psi,d}\to\infty$ implies $d(|\omega_1^\psi(d)|+|\omega_2^\psi(d)|)/v_{\psi,d}\to\infty$ as $d \rightarrow \infty$. Since $r_d^\psi \ge (|\omega_1^\psi(d)|+|\omega_2^\psi(d)|)/\sqrt{2}$,
this implies $v_{\psi,d}/(d\,r_d^\psi)\to 0$ as $d\rightarrow\infty$.

Now, Assumption~(A4)(a) bounds the second moments of
$\psi(|W^{(q)}|)$ uniformly in $q$, while
condition~(A4)(b) ensures $\sum_{q\ne q'}|\mathrm{Corr}(\psi(|W^{(q)}|),\psi(|W^{(q')}|))|=o(d^2)$ for $\Wvec=\Zvec_i-\Zvec_j$. So, WLLN holds for $\{\psi(|W^{(q)}|):q\ge 1\}$. Hence, for every fixed pair $(i,j)$,
\begin{equation}\label{eq:wlln5}
	\frac{1}{d}\sum_{q=1}^{d}\psi\!\bigl(|Z_i^{(q)}-Z_j^{(q)}|\bigr)
	\;=\; a_{rs}(d) \;+\; O_P\!\Bigl(\tfrac{v_{\psi,d}}{d}\Bigr),
\end{equation}
uniformly over all $\tbinom{n}{2}$ pairs. So dividing~\eqref{eq:wlln5} by $r_d^\psi$ makes the remainder
$o_P(1)$. This in turn implies 
\begin{align}
	\label{T_ij-contrasts-thm5}
	\nonumber
	T_{12}^{h,\psi}(t)-T_{11}^{h,\psi}(t)
	&=\beta_1\omega_1^\psi(d)\,\mathbb{I}_{\{t\le\tau\}}
	+\alpha_1\bigl(\alpha_2\omega_1^\psi(d)-\alpha_2'\omega_2^\psi(d)\bigr)\,
	\mathbb{I}_{\{t>\tau\}}
	+O_P\!\left(\frac{v_{\psi,d}}{d}\right),\\[4pt]
	T_{12}^{h,\psi}(t)-T_{22}^{h,\psi}(t)
	&=\beta_1\bigl(\beta_2\omega_2^\psi(d)-\beta_2'\omega_1^\psi(d)\bigr)\,
	\mathbb{I}_{\{t\le\tau\}}
	+\alpha_1\omega_2^\psi(d)\,\mathbb{I}_{\{t>\tau\}}
	+O_P\!\left(\frac{v_{\psi,d}}{d}\right).
\end{align}
%
where $\alpha_1,\alpha_2,\alpha_2',\beta_1,\beta_2,\beta_2'$ are as in proof of Theorems \ref{thm:HDLSS-consistency} and \ref{thm:sparse1}.
Dividing both sides of \eqref{T_ij-contrasts-thm5} by $r_d^\psi$ and taking the sum of their squares, yields
\[
\sup_{t\in\mathcal{A}_n}\left|
\frac{w(t)\,\mathcal{D}_n^{h,\psi}(t)}{s^\psi_d}
-w(t)\,G_d^{h,\psi}(t;\tau)\right|=o_P(1),
\]
where $G_d^{h,\psi}(t;\tau)$ is similar to $G_d(t;\tau)$ (see the proof of Theorem \ref{thm:sparse1}) with $\omega_i^\psi(d)$'s replacing $\omega_i(d)$. Here also one can show that $w(t)G_d^{h,\psi}(t;\tau)<w(\tau)G_d^{h,\psi}(\tau;\tau)$ for all $t\neq\tau$. Rest of the proof follows from the proofs of part (a) and (b) of Theorem \ref{thm:sparse1}.
\end{proof}

\begin{proof}[\textbf{Proof of Theorem~\ref{thm:HDHSS 1}}]
When $\F_1,\F_2$ are sub-exponential, for two independent random variables $\Zvec_i\sim F_r$ and $\Zvec_j\sim F_s$ ($r,s\in\{1,2\}$) there exists a constant $K>0$ such that for all $i\neq j$ (see Lemma A.~\ref{sub-exponential:conc})
\[
\sup_{\|v\|_2=1}\ \|\langle v,\Zvec_i-\Zvec_j\rangle\|_{\psi_1}\le K.
\]
From Lemma A.~\ref{sub-exponential:conc}, we can also conclude that $\|V_{ij}\|_{\psi_1}\le K_0\sqrt d$ for all $i\neq j$, where $V_{ij}=\|\Zvec_i-\Zvec_j\|$ and $K_0$ is a constant depending only on $K$. Applying Lemma A.~\ref{lem:block+perm-U}
to $\Delta_{ij}(t):=T_{ij}(t)-\mu_{ij}(t)$ for $i,j\in\{1,2\}$ and $r>0$, we get
\begin{align}
	\P\big(|\Delta_{11}(t)|\ge r\big)&\le 2\exp\Big\{-c_0\,t\,\phi\!\Big(\frac{r}{K_0\sqrt d}\Big)\Big\},\nonumber\\
	\P\big(|\Delta_{22}(t)|\ge r\big)&\le 2\exp\Big\{-c_0\,(n-t)\,\phi\!\Big(\frac{r}{K_0\sqrt d}\Big)\Big\}~~\text{and }\nonumber\\
	\P\big(|\Delta_{12}(t)|\ge r\big)&\le 2\exp\Big\{-c_0\,m_t\,\phi\!\Big(\frac{r}{K_0\sqrt d}\Big)\Big\}, \label{Delta_concentration}
\end{align}
where $\phi(u):=\min\{u^2,u\}$ and $m_t=\min\{t,n-t\}$. 

For ease of notation, let us denote $\mathcal{D}_n(\X_t,\Y_t)$ by $D(t)$ and its population analog $\{\mu_{12}(t)-\mu_{11}(t)\}^2+\{\mu_{12}(t)-\mu_{22}(t)\}^2$ by $D^*(t)$.
Define $A_t:=|\mu_{12}(t)-\mu_{11}(t)|$, $B_t:=|\mu_{12}(t)-\mu_{22}(t)|$ and $\eta_{rs}:=\E\|\Xvec-\Yvec\|$ for $\Xvec\sim \F_r$, $\Yvec\sim \F_s$ being independent.
Also define $\nu_1:=\eta_{12}-\eta_{11}$ and $\nu_2:=\eta_{12}-\eta_{22}$. Using
a straightforward calculation, one can verify that 
\begin{align*}
	\mu_{12}(t)-\mu_{11}(t) &= \beta_1\nu_1\,\mathbb{I}_{t\le\tau}
	+ \alpha_1(\alpha_2\nu_1-\alpha_2'\nu_2)\,\mathbb{I}_{t>\tau},\\
	\mu_{12}(t)-\mu_{22}(t) &= \beta_1(\beta_2\nu_2-\beta_2'\nu_1)\,\mathbb{I}_{t\le\tau}
	+ \alpha_1\nu_2\,\mathbb{I}_{t>\tau}.
\end{align*}
where $\alpha_{1}=\frac{\tau}{t}$, $\beta_{1}=\frac{n-\tau}{n-t}$, $\alpha_{2}=\frac{\tau-1}{t-1}$, $\beta_{2}=\frac{n-\tau-1}{n-t-1}$, $\alpha_{2}^{\prime}=1-\alpha_{2}$,
and $\beta_{2}^{\prime}=1-\beta_{2}$ as in the proof of Theorems 2-5.
For $t\le\tau$, note that $\beta_1,\beta_2,\beta_2^\prime$ all lie in $[0,1]$. The triangle inequality gives
\[
A_t+B_t \;\le\; \beta_1|\nu_1| + \beta_1\bigl(\beta_2|\nu_2|+\beta_2'|\nu_1|\bigr)
= \beta_1(1+\beta_2')|\nu_1|+\beta_1\beta_2|\nu_2|.
\]
Clearly, $\beta_1\beta_2 \le 1$. We also have
\begin{align*}
	\beta_{1}(1+\beta_{2}')&=\frac{n-\tau}{n-t}\cdot\frac{n+\tau-2t-1}{n-t-1} =\frac{n-t+t-\tau}{n-t}\cdot\frac{n-t-1+\tau-t}{n-t-1}\\
	& =1-\frac{(\tau-t)(\tau-t-1)}{(n-t)(n-t-1)}\;\le\;1 ~~(\text{since } t\le \tau<n)
\end{align*}
This implies $A_{t}+B_{t}\le|\nu_{1}|+|\nu_{2}|$. Note that the equality holds only when $\beta_1\beta_2=1$ and $\beta_1(1+\beta_2^{'})=1$. This is possible only for $t=\tau$.

Similarly, for $t\ge \tau$, note that $\alpha_1 $, $\alpha_2$ and $\alpha_2'$, all lie
in $[0,1]$. Now, the triangle inequality gives
\[
A_t+B_t \;\le\; \alpha_1\bigl(\alpha_2|\nu_1|+\alpha_2'|\nu_2|\bigr)
+\alpha_1|\nu_2|
= \alpha_1\alpha_2|\nu_1|+\alpha_1(1+\alpha_2')|\nu_2|.
\]
Since $\alpha_1\alpha_2\le 1$
and
\[
\alpha_1(1+\alpha_2')
= \frac{\tau}{t}\cdot\frac{2t-\tau-1}{t-1}
= 1 - \frac{(t-\tau)(t-\tau-1)}{t(t-1)} \;\le\; 1 ~~(\text{since } t \ge \tau),
\]
we have $A_t+B_t\le|\nu_1|+|\nu_2|$. Here also, the equality holds if and only if $t=\tau$.
Therefore $A_t+B_t\le A_{\tau}+B_{\tau} =|\mu_{12}(\tau)-\mu_{11}(\tau)|+|\mu_{12}(\tau)-\mu_{22}(\tau)|=|\eta_{12}-\eta_{11}|+|\eta_{12}-\eta_{22}|=M_\mu$ for all $t\in\mathcal{A}_n$.


For any $r>0$, define the event $\mathcal{E}_t(r) := \bigl\{\max_{rs\in\{11,12,22\}}|\Delta_{rs}(t)|\le r\bigr\}$. If it holds, we 
can derive an upper bound on $|D(t)-D^*(t)|$ as follows.
Define $T_\ell(t)=T_{12}(t)-T_{\ell\ell}(t)$ and $\mu_\ell(t)=\mu_{12}(t)-\mu_{\ell\ell}(t)$ for $\ell\in\{1,2\}$.
So, we have $D(t)=T_1^2(t)+T_2^2(t)$ and $D^*(t)=\mu_1^2(t)+\mu_2^2(t)$. Note that
$T_\ell(t)-\mu_\ell(t)=\Delta_{12}(t)-\Delta_{\ell\ell}(t)$ for $\ell=1,2$. Also, 
on $\mathcal{E}_t(r)$, for $\ell=1,2$, we have $|T_\ell(t)-\mu_\ell(t)|\le 2r$ (follows from triangle inequality).
Now, using the result $|a^2-b^2|=|a-b||a+b|\le|a-b|(|a-b|+2|b|)$ with  $a=T_\ell(t)$ and $b=\mu_\ell(t)$, and noting that $A_t=|\mu_1(t)|$, $B_t=|\mu_2(t)|$, one gets
\begin{align*}
	|D(t)-D^*(t)|
	&\le \bigl|T_1^2(t)-\mu_1^2(t)\bigr| + \bigl|T_2^2(t)-\mu_2^2(t)\bigr| \\
	&\le 2r\,(2r+2A_t) + 2r\,(2r+2B_t) \\
	&= 4(A_t+B_t)\,r + 8r^2 \le 4M_\mu\,r+8r^2.
\end{align*}
Hence, by taking $r=r^*(\varepsilon):=\min\!\Big\{\dfrac{\varepsilon}{8M_\mu},\,
\sqrt{\dfrac{\varepsilon}{16}}\Big\}$,
we have $|D(t)-D^*(t)|\le \varepsilon$ on $\mathcal{E}_t(r^*(\varepsilon))$. A union bound over $t\in\mathcal{A}_n$ gives
\begin{align*}
	\P\Bigl(\sup_{t\in\mathcal{A}_n}|D(t)-D^*(t)|>\varepsilon\Bigr)
	&\le \P\Bigl(\bigcup_{t\in\mathcal{A}_n}\mathcal{E}_t(r^*(\varepsilon))^c\Bigr) 
	\le \sum_{t\in\mathcal{A}_n}\P\bigl(\mathcal{E}_t(r^*(\varepsilon))^c\bigr) \\
	&\le \sum_{t\in\mathcal{A}_n}\sum_{rs\in\{11,12,22\}}\P\bigl(|\Delta_{rs}(t)|>r^*(\varepsilon)\bigr) \\
	&\le 6|\mathcal{A}_n|\exp\Bigl(-c_0\,m_*\,\phi\!\Bigl(\frac{r^*(\varepsilon)}{K_0\sqrt{d}}\Bigr)\Bigr),
\end{align*}
where $m_*:=\min_{t\in\mathcal A_n}m_t$ (the last inquality follows from \eqref{Delta_concentration}). Let $F(t):=w(t)D^*(t)$, where $w(t)=t(n-t)/n^2$. Since $0\le w(t)\le 1/4$, this implies
\begin{equation}
	\label{eq:unif_D_bound_scaled}
	\P\Big(\sup_{t\in\mathcal A_n}\big|w(t)D(t)-F(t)\big|>\tfrac{\varepsilon}{4}\Big)
	\le 6|\mathcal A_n|\exp\!\Big(-c_0\,m_*\,\phi\!\Big(\tfrac{r^*(\varepsilon)}{K_0\sqrt{d}}\Big)\Big).    
\end{equation}

From the condition of Theorem~\ref{thm:HDHSS 1},for every fixed $\delta > 0$, we have
\begin{equation}
	\label{eq:HDHSS_one_condition_unscaled}
	r^*\!\big(\Delta_{n,d}(\delta)\big)\ \gg\ \sqrt{\frac{d\log n}{n\delta_n}}.
\end{equation}

\smallskip
\noindent
{(\text{a})}
Define $\mathcal{A}^\circ_\delta:=\{t\in\mathcal{A}_n:|t-\tau|\ge\delta n\}$
and recall $\Delta_{n,d}(\delta):=F(\tau)-\max_{t\in\mathcal{A}^\circ_\delta}F(t)>0$.
We claim that given the event $\mathcal{G}:=\big\{\sup_{t\in\mathcal{A}_n}|w(t)D(t)-F(t)| \le\tfrac{1}{2}\Delta_{n,d}(\delta)\big\}$,  $\widehat{T}_n$  satisfies $|\widehat{T}_n-\tau|<\delta n$.
On $\mathcal{G}$, for any $t\in\mathcal{A}^\circ_\delta$,
\[
w(t)D(t)\le F(t)+\tfrac{1}{2}\Delta_{n,d}(\delta)
\le \max_{t'\in\mathcal{A}^\circ_\delta}F(t')+\tfrac{1}{2}\Delta_{n,d}(\delta)
= F(\tau)-\tfrac{1}{2}\Delta_{n,d}(\delta),
\]
From the definition of ${\cal G}$, we also have $w(\tau)D(\tau)\ge F(\tau)-\tfrac{1}{2}\Delta_{n,d}(\delta)$.
These two together imply $w(\tau)D(\tau)\ge w(t)D(t)$ for all $t\in\mathcal{A}^\circ_\delta$ and
$\widehat{T}_n\notin\mathcal{A}^\circ_\delta$, or equivalently, $|\widehat{T}_n-\tau|<\delta n$.

Therefore, applying \eqref{eq:unif_D_bound_scaled} with $\frac{\varepsilon}{4}=\frac{1}{2}\Delta_{n,d}(\delta)$, we obtain
\begin{align}
	\P\Big(\Big|\frac{\widehat T_n}{n}-\theta\Big|>\delta\Big) \le P({\cal G}^c)
	\le 6|\mathcal A_n|\exp\!\Big(-c_0\,m_*\,\phi\!\Big(\tfrac{r^*(2\Delta_{n,d}(\delta))}{K_0\sqrt{d}}\Big)\Big). \label{concentration_T_n_hat}    
\end{align}
Now, we prove the almost sure convergence of $\widehat{T}_n/n$ 
follows from the Borel--Cantelli lemma. First note that
\[
\Delta_{n,d}(\delta)\;\le\; F(\tau)=w(\tau)(|\nu_1|^2+|\nu_2|^2)
\;\le\;\tfrac{1}{4}(|\nu_1|+|\nu_2|)^2=\frac{M_\mu^2}{4}.
\]
Since $r^*(x)\le\sqrt{x/16}$ for all $x>0$, we have 
$r^*(\Delta_{n,d}(\delta))\le M_\mu/8$ and $r^*(2\Delta_{n,d}(\delta))
\le\sqrt{2}\,M_\mu/8$. Under the sub-exponential 
assumption, $\eta_{rs}\le K_0\sqrt{d}$ for all $r,s\in\{1,2\}$, 
giving $M_\mu\le 4K_0\sqrt{d}$. Consequently, 
\[
\frac{r^*(2\Delta_{n,d}(\delta))}{K_0\sqrt{d}}
\;\le\;\frac{\sqrt{2}\,M_\mu}{8K_0\sqrt{d}}
\;\le\;\frac{\sqrt{2}\cdot 4K_0\sqrt{d}}{8K_0\sqrt{d}}
\;=\;\frac{1}{\sqrt{2}}\;<\;1.
\]
Since $r^*$ is monotonically increasing and $\phi(x) = x^2$ for $x<1$, 
\[
c_0\,m_*\,\phi\!\left(\frac{r^*(2\Delta_{n,d}(\delta))}{K_0\sqrt{d}}\right)
\;\ge\; c_0\,m_*\,\phi\!\left(\frac{r^*(\Delta_{n,d}(\delta))}{K_0\sqrt{d}}\right)
\;=\; \frac{c_0\,m_*}{K_0^2\,d}
\,\bigl(r^*(\Delta_{n,d}(\delta))\bigr)^2.
\]
From \eqref{eq:HDHSS_one_condition_unscaled}, we have  
$\bigl(r^*(\Delta_{n,d}(\delta))\bigr)^2/K_0^2\,d\ \gg \log n/(n\delta_n)$, 
and hence $$
c_0\,m_*\,\phi\!\left(\frac{r^*(2\Delta_{n,d}(\delta))}{K_0\sqrt{d}}\right)
\gg\log n  ~~( \text{since } m_*\asymp n\delta_n)
.$$
This proves the summability of the right side of \ref{concentration_T_n_hat}, and hence 
an application of the Borel--Cantelli 
lemma yields $\widehat\theta_n=\widehat{T}_n/n\stackrel{a.s.}{\longrightarrow}\theta$.

\smallskip
\noindent{(b)} We show that $\P(\widehat{S}_n > c_{1-\alpha})\to 1$ by establishing 
that $P(\widehat{S}_n>F(\tau)/4) \rightarrow 1$ and  $P(c_{1-\alpha}<F(\tau)/4)\rightarrow 1$.
Define the event $\mathcal{E}_{n,d} := \{\sup_{t\in\mathcal{A}_n}|w(t)D(t)-F(t)|
\le F(\tau)/4\}$ and note that 
\[
\P(\mathcal{E}_{n,d}^c)
\;\le\;
6|\mathcal{A}_n|\exp\!\Big(-c_0\,m_*\,
\phi\!\big(r^*(F(\tau))/K_0\sqrt{d}\big)\Big)
\]
(follows from \eqref{eq:unif_D_bound_scaled} using $\varepsilon = F(\tau)$). Since $F(\tau)\ge\Delta_{n,d}(\delta)$ for every $\delta>0$, and $r^*$ is monotonically 
increasing, under the condition~\eqref{eq:HDHSS_one_condition_unscaled}, arguments similar to those as in part~(a) give $\P(\mathcal{E}_{n,d}^c)\to 0$.
Now, given the event $\mathcal{E}_{n,d}$,
we have
$
\widehat{S}_n =\sup_{t\in\mathcal{A}_n}w(t)D(t) \ge w(\tau)D(\tau) \ge F(\tau) - \tfrac{1}{4}F(\tau) 
\;>\; \tfrac{1}{4}F(\tau).
$
Hence $\P(\widehat{S}_n > F(\tau)/4)\ge P(\mathcal{E}_{n,d})\to 1$.

Recall that $c_{1-\alpha}$ is the $(1-\alpha)$-quantile of the 
conditional distribution of $\widehat{S}_n^{(\Pi)}$ given the data, and $\Pi$ 
is a uniform random permutation of $\{1,\ldots,n\}$ independent of the 
data. 
So, $c_{1-\alpha}>F(\tau)/4$ implies that given the data, more than $\alpha$ proportions of permutations give $\widehat{S}_n^{(\Pi)}>F(\tau)/4$. 
Now, applying Markov's inequality to the random variable 
$p_n(\Zvec_1,\ldots,\Zvec_n) := \linebreak \P(\widehat{S}_n^{(\Pi)}>F(\tau)/4\mid\mathbf{Z}_1,\ldots,\mathbf{Z}_n)$, we get
\begin{align*}
	\P\!\big(c_{1-\alpha}>F(\tau)/4\big)
	&\;\le\;
	\P\!\big(p_n(\Zvec_1,\Zvec_2,\ldots,\Zvec_n)>\alpha\big)\\
	&\;\le\;
	\frac{1}{\alpha}\,\E[p_n(\Zvec_1,\Zvec_2,\ldots,\Zvec_n)]
	\;=\;
	\frac{1}{\alpha}\,\P\!\big(\widehat{S}_n^{(\Pi)}>F(\tau)/4\big),
\end{align*}
Now, applying \eqref{eq:max_bound_final} with $\varepsilon = F(\tau)/4$, we have
\begin{align}
	\P\!\big(\widehat{S}_n^{(\Pi)}>F(\tau)/4\big)
	\;\le\;& \;
	C|\mathcal{A}_n|\exp\!\Big(-c\,m_*\min\Big(\frac{F(\tau)}{4d},
	\sqrt{\frac{F(\tau)}{4d}}\Big)\Big) \nonumber\\
	&+\;C|\mathcal{A}_n|\exp\!\Big(-c\,m_*\frac{F(\tau)}{4M_\mu^{2}}\Big). \label{permutation_conc}
\end{align}
Using $(a^2+b^2)\ge(a+b)^2/2$, we get  
$F(\tau)/M_\mu^2 = w(\tau)D^*(\tau)/M^2_\mu\ge w(\tau)/2$. So, the exponent of the second term is at least 
$c\,m_*\,w(\tau)/8 \gg \log n$ (here $m_* \asymp n\delta_n \gg \log n$ follows from the condition of the theorem). 

Since 
$F(\tau)\ge\Delta_{n,d}(\delta)$ and $r^*(x)\le\sqrt{x/16}$, condition 
\eqref{eq:HDHSS_one_condition_unscaled}
give \linebreak
$
{\Delta_{n,d}(\delta)}/{d}
\;\gg\;
{\log n}/{n\delta_n},
$
and since $m_*\asymp n\delta_n$, we have
\[
m_*\cdot\frac{F(\tau)}{d}
\;\ge\;
m_*\cdot\frac{\Delta_{n,d}(\delta)}{d}
\;\gg\;
(n\delta_n)\cdot\frac{\log n}{n\delta_n}
\;=\;\log n.
\]
Therefore both terms in \eqref{permutation_conc} tend to zero, giving $\P(c_{1-\alpha}>F(\tau)/4)\to 0$. Combining the two convergence results, we get $\P(\widehat{S}_n>c_{1-\alpha})\to 1$.
\end{proof}

\noindent
{\bf Remark}: {\underline {Relaxed condition for HDHSS consistency under Sub-Gaussianity}}

\vspace{0.1in}
Under the coordinatewise sub-Gaussian assumption, Lemma~A.\ref{sub-gaussian:conc_unscaled} gives a 
concentration bound for the centered distance kernel (see \eqref{eq:sg_ustat_tail_unscaled}), which is somewhat similar to the concentration bound stated in Lemma~A.\ref{lem:block+perm-U} for the sub-exponential family (see \eqref{sub-expo-conc}). However, the bound in \eqref{eq:sg_ustat_tail_unscaled} is much sharper and unlike \eqref{sub-expo-conc}, it is
independent of $d$.  Repeating similar arguments as in the proof of 
Theorem \ref{thm:HDHSS 1}, one can prove the HDHSS consistency of the proposed method for the sub-Gaussian family, but under a more relaxed condition $r^*\!\big(\Delta_{n,d}(\delta)\big)\ \gg\ \sqrt{\frac{\log n}{n\delta_n}}$, where unlike  \eqref{eq:HDHSS_one_condition_unscaled}, where we do not have
factor $d$ on the right-hand side.

\begin{proof}[\textbf{Proof of Theorem~\ref{HDHSS 2}}]
Denote $\mathcal{D}_n^{h,\psi}(\mathcal{X}_t,\mathcal{Y}_t)$ 
by $D^{h,\psi}(t)$ and its population analogue by
$D^{h,\psi}_*(t) := \{\E T_{12}^{h,\psi}(t)-\E T_{11}^{h,\psi}(t)\}^{2}
+\{\E T_{12}^{h,\psi}(t)-\E T_{22}^{h,\psi}(t)\}^{2}
$. 
Also, define $F^{h,\psi}(t):=w(t)D^{h,\psi}_*(t)$, where $w(t)=t(n-t)/n^2$.
The key difference from the proof of Theorem~6 is that, since
$\varphi_{h,\psi}$ is bounded with $\|\varphi_{h,\psi}\|_\infty\le L$,
concentration of the U-statistics $T_{rs}^{h,\psi}(t)$ follows from
McDiarmid's inequality rather than sub-exponential concentration.

\smallskip
\noindent{(a)}
Define $T_1^{h,\psi}(t):=T_{12}^{h,\psi}(t)-T_{11}^{h,\psi}(t)$ and
$T_2^{h,\psi}(t):=T_{12}^{h,\psi}(t)-T_{22}^{h,\psi}(t)$. So, here
$D^{h,\psi}(t)=(T_1^{h,\psi}(t))^2+(T_2^{h,\psi}(t))^2$
where $|T_\ell^{h,\psi}(t)|\le 2L$ for $\ell\in\{1,2\}$.
Fix $t\in\mathcal{A}_n$ and let $m_t:=\min\{t,n-t\}$.
Note that if $\Zvec_k$ is replaced by an i.i.d.\ copy $\Zvec^\prime_k$, then each of
$T_{11}^{h,\psi}(t),T_{22}^{h,\psi}(t),T_{12}^{h,\psi}(t)$
changes by at most $CL/m_t$ (only pairs involving index $k$ are affected) for some constant $C>0$. To show this, let us consider the following two cases.
\noindent\textit{Case 1 ($k\le t$):} The number of terms in 
$T_{11}^{h,\psi}(t)$, $T_{22}^{h,\psi}(t)$, and $T_{12}^{h,\psi}(t)$ 
involving $\Zvec_k$ are $t-1$, $0$, and $n-t$, respectively. Using 
$|\varphi_{h,\psi}|\le L$ and the respective normalizations, the changes 
for replacing $\Zvec_k$ by $\Zvec_k'$ are at most
\[
\frac{(t-1)L}{\binom{t}{2}} = \frac{2L}{t},
\qquad 0, \qquad
\frac{(n-t)\cdot L}{t(n-t)} = \frac{L}{t},
\]
respectively. Hence, the difference in $T_\ell^{h,\psi}(t)=|T_\ell^{h,\psi,(k)}(t)-T_\ell^{h,\psi}(t)|$ is bounded by $c_{k,t}=3L/t$ both for $\ell=1$ and $2$.

\smallskip
\noindent\textit{Case 2 ($k> t$):} The number of terms in 
$T_{11}^{h,\psi}(t)$, $T_{22}^{h,\psi}(t)$, and $T_{12}^{h,\psi}(t)$ 
involving $\Zvec_k$ are $0$, $n-t-1$, and $t$, respectively. So, the changes in these terms are at most
\[
0, \qquad
\frac{(n-t-1)L}{\binom{n-t}{2}} = \frac{2L}{n-t},
\qquad
\frac{t\cdot L}{t(n-t)} = \frac{L}{n-t},
\]
respectively. Hence, the difference in $T_\ell^{h,\psi}(t)$ is bounded by $c_{k,t}=3L/(n-t)$ both for $\ell=1$ and $2$.

Since bounded-difference condition holds with $\sum_{k=1}^{n}\!c_{k,t}^2=L^2/cm_t$ (where $c=\frac{1}{18}$), 
using McDiarmid's inequality, for all $r>0$, we have 
\[
\P\!\big(|T_\ell^{h,\psi}(t)-\E T_\ell^{h,\psi}(t)|>r\big)
\;\le\; 2\exp\!\Big(-\frac{c\,m_t\,r^2}{L^2}\Big),
\qquad \ell\in\{1,2\}.
\]
Since $|x^2-y^2|\le(|x|+|y|)|x-y|$ and
$|T_\ell^{h,\psi}(t)|,|\E T_\ell^{h,\psi}(t)|\le 2L$, we get
\[
\big|D^{h,\psi}(t)-D^{h,\psi}_*(t)\big|
\le 4L\sum_{\ell=1}^2\big|T_\ell^{h,\psi}(t)-\E T_\ell^{h,\psi}(t)\big|.
\]
Since $w(t)\le 1/4$, for any fixed $\eta>0$, the use of  $r=\eta/(8L)$ and a union bound
over $\ell\in\{1,2\}$ yields
\begin{align*}
	\Pr\Big(\big|w(t)D^{h,\psi}(t)-F^{h,\psi}(t)\big|>\eta\Big)
	\le 4\exp\!\Big(-c\,m_t\,\frac{\eta^2}{L^4}\Big).    
\end{align*}
Since $m_t\ge n\delta_n$, a union bound over $t\in\mathcal{A}_n$
gives
\begin{align}
	\Pr\!\Big(\sup_{t\in\mathcal{A}_n}
	\big|w(t)D^{h,\psi}(t)-F^{h,\psi}(t)\big|>\eta\Big)
	\le C|\mathcal{A}_n|\exp\!\Big(-c\,n\delta_n\,\frac{\eta^2}{L^4}\Big).  \label{eq:concentration:D_h_psi}   
\end{align}

Define $\mathcal{A}^{\circ}_{\delta}:=\{t\in\mathcal{A}_n:|t-\tau|\ge\delta n\}$
and recall $\Delta_{n,d}^{h,\psi}(\delta)
:=F^{h,\psi}(\tau)-\max_{t\in\mathcal{A}^{\circ}_{\delta}}F^{h,\psi}(t)>0$.
We claim that given the event
\[
\mathcal{G}_n^{h,\psi}:=\Big\{\sup_{t\in\mathcal{A}_n}
\big|w(t)D^{h,\psi}(t)-F^{h,\psi}(t)\big|
\le\tfrac{1}{2}\Delta_{n,d}^{h,\psi}(\delta)\Big\},
\]
the estimator $\widehat{T}_n^{h,\psi}$ satisfies
$|\widehat{T}_n^{h,\psi}-\tau|<\delta n$.
On $\mathcal{G}_n^{h,\psi}$, for any $t\in\mathcal{A}^{\circ}_{\delta}$,
\begin{align*}
	w(t)D^{h,\psi}(t)
	&\le F^{h,\psi}(t)+\tfrac{1}{2}\Delta_{n,d}^{h,\psi}(\delta)\\
	&\le \max_{t'\in\mathcal{A}^{\circ}_{\delta}}F^{h,\psi}(t')
	+\tfrac{1}{2}\Delta_{n,d}^{h,\psi}(\delta)
	= F^{h,\psi}(\tau)-\tfrac{1}{2}\Delta_{n,d}^{h,\psi}(\delta).    
\end{align*}
From the definition of $\mathcal{G}_n^{h,\psi}$ we also have
$w(\tau)D^{h,\psi}(\tau)\ge F^{h,\psi}(\tau)-\tfrac{1}{2}\Delta_{n,d}^{h,\psi}(\delta)$.
These two together imply
$w(\tau)D^{h,\psi}(\tau)\ge w(t)D^{h,\psi}(t)$ for all
$t\in\mathcal{A}^{\circ}_{\delta}$, and hence
$\widehat{T}_n^{h,\psi}\notin\mathcal{A}^{\circ}_{\delta}$, or equivalently,
$|\widehat{T}_n^{h,\psi}-\tau|<\delta n$. So, we have
$P({\cal G}_n^{h,\psi}) \le P(|\widehat{T}_n^{h,\psi}-\tau|\le \delta n)$. 

Now, for $\eta=\tfrac{1}{2}\Delta_{n,d}^{h,\psi}(\delta)$,
since $n\delta_n\eta^2=\frac{1}{4}n\delta_n (\Delta_{n,d}^{h,\psi}(\delta))^2\gg \frac{1}{4}\log n$ (follows under the given condition),
the right side of \eqref{eq:concentration:D_h_psi} summable in $n$, i.e., $\sum_{n\ge 1} P\big(({\cal G}_n^{h,\psi})^c\big) < \infty$. This implies
\[ \sum_{n \ge 1} P\left(\Big|\frac{\widehat{T}_n^{h,\psi}}{n}-\theta\Big|> \delta \right) =\sum_{n \ge 1} P(|\widehat{T}_n^{h,\psi}-\tau|> \delta n) \le \sum_{n\ge 1} P\big(({\cal G}_n^{h,\psi})^c\big) < \infty. \]

{
	Since $\delta$ is arbitrary, using the Borel--Cantelli lemma, we have
	
	$\widehat{\theta}_n^{h,\psi}:=\widehat{T}_n^{h,\psi}/n\to\theta$ a.s.
}



\smallskip
\noindent{(b)} We prove $\P(\widehat{S}_n^{h,\psi}>\hat c_{1-\alpha})\to 1$ by
establishing $\P(\widehat{S}_n^{h,\psi}>F^{h,\psi}(\tau)/4)\to 1$ and
$\P(\hat c_{1-\alpha}<F^{h,\psi}(\tau)/4)\to 1$. Define
$$\mathcal{E}_{n,d}^{h,\psi}
:=\{\sup_{t\in\mathcal{A}_n}|w(t)D^{h,\psi}(t)-F^{h,\psi}(t)|
\le F^{h,\psi}(\tau)/4\}.$$ The bounded-difference calculations in part~(a)
with $\eta=F^{h,\psi}(\tau)/4$ and the fact
$F^{h,\psi}(\tau)\ge\Delta_{n,d}^{h,\psi}(\delta)\gg L^2\sqrt{\log n/(n\delta_n)}$,
gives $\P((\mathcal{E}_{n,d}^{h,\psi})^c)\to 0$. Now, given $\mathcal{E}_{n,d}^{h,\psi}$,
we have $\widehat{S}_n^{h,\psi}\ge w(\tau)D^{h,\psi}(\tau)\ge \tfrac34 F^{h,\psi}(\tau)
>\tfrac14 F^{h,\psi}(\tau)$. Therefore,
$\P(\widehat{S}_n^{h,\psi}>F^{h,\psi}(\tau)/4)\to 1$.

Now, define $\widehat{S}_n^{h,\psi,(\Pi)}$ as the permutation analog of $\widehat{S}_n^{h,\psi}$ and $c_{1-\alpha}$ is the $(1-\alpha)$-th quantile of the conditional distribution of $\widehat{S}_n^{h,\psi,(\Pi)}$ given the data. Applying Markov's inequality to the random variable
$p_n^{h,\psi}(\Zvec_1,\ldots,\Zvec_n)=
\P(\widehat{S}_n^{h,\psi,(\Pi)}>F^{h,\psi}(\tau)/4\mid\Zvec_1,\ldots,\Zvec_n)$, as in
the proof of Theorem~\ref{thm:HDHSS 1},
\begin{align}
	\P\!\big(\hat c_{1-\alpha}>F^{h,\psi}(\tau)/4\big)
	\;\le\; \tfrac{1}{\alpha}\,\P\!\big(\widehat{S}_n^{h,\psi,(\Pi)}>F^{h,\psi}(\tau)/4\big).
	\label{perm-consistency-generalized}
\end{align}
Again, 
from Lemma~A.\ref{lem:perm_conc_BG_gen}: for every $\varepsilon>0$, we have
\begin{equation}\label{permutation_conc_BG}
	\P\!\big(\widehat{S}_n^{h,\psi,(\Pi)}>\varepsilon\big)
	\;\le\;
	C|\mathcal{A}_n|\exp\!\Big(-c\,m_*\,\tfrac{\varepsilon}{L^2}\Big)
	+C|\mathcal{A}_n|\exp\!\Big(-c\,m_*\,\tfrac{\varepsilon}{(M_\mu^{h,\psi})^2}\Big),
\end{equation}
where $M_\mu^{h,\psi}:=|\mu_{12}^{h,\psi}-\mu_{11}^{h,\psi}|
+|\mu_{12}^{h,\psi}-\mu_{22}^{h,\psi}|$ with
$\mu_{rs}^{h,\psi}:=\E\varphi_{h,\psi}(\Xvec,\Yvec)$, and
$\Xvec\sim F_r$, $\Yvec\sim F_s$ are independent.
Since
$F^{h,\psi}(\tau)/(M_\mu^{h,\psi})^2\ge w(\tau)/2$ (follows from the inequality $(a^2+b^2)\ge(a+b)^2/2$), for $\varepsilon=F^{h,\psi}(\tau)/4$,
the exponent of the second term is at least $c\,m_*\,w(\tau)/8\gg\log n$. Since
$F^{h,\psi}(\tau)\ge\Delta_{n,d}^{h,\psi}(\delta)\gg L^2\sqrt{\log n/(n\delta_n)}$
and $m_*\asymp n\delta_n$, the  exponent of the first term is at least
$c\sqrt{n\delta_n\log n}\gg\log n$. So, for $\varepsilon=F^{h,\psi}(\tau)/4$, both terms
in~\eqref{permutation_conc_BG} converge to $0$ as $n$ tends to infinity. So, from \eqref{perm-consistency-generalized}, 
$\P(\hat c_{1-\alpha}>F^{h,\psi}(\tau)/4) \rightarrow 0$  as $n \rightarrow \infty$. Combining these two results, we get
$\P(\widehat{S}_n^{h,\psi}>\hat c_{1-\alpha})\to 1$ as $n$ tends to infinity.
\end{proof}

\vskip .65cm
\noindent
Spandan Ghoshal, Theoretical Statistics and Mathematics Unit, Indian Statistical Institute, Kolkata
\\
E-mail: spandan.ghoshal\_r@isical.ac.in
\vskip 2pt

\noindent
Bilol Banerjee, Department of Statistics and Data Science, National University of Singapore
\\
E-mail: bilol@nus.edu.sg
\vskip 2pt

\noindent
Anil K. Ghosh, Theoretical Statistics and Mathematics Unit, Indian Statistical Institute, Kolkata
\\
E-mail: akghosh@isical.ac.in

\end{document}